\documentclass[12pt]{amsart}
\usepackage{graphicx}
\usepackage{amsmath}
\usepackage{amsfonts}
\usepackage{amssymb}
\pdfoutput=1
\usepackage{setspace}
\usepackage{datetime}
\usepackage[colorlinks,
            linkcolor=black,
            anchorcolor=blue,
            citecolor=black]{hyperref}
\usepackage{geometry}
\allowdisplaybreaks[4]

\newtheorem{thm}{Theorem}[section]
\newtheorem{cor}[thm]{Corollary}

\newtheorem{lem}[thm]{Lemma}

\theoremstyle{definition}
\newtheorem{defn}[thm]{Definition}
\theoremstyle{remark}
\newtheorem{rem}[thm]{Remark}
\theoremstyle{conclusion}

\theoremstyle{conjecture}

\numberwithin{equation}{section}
\newcommand{\R}{\mathbb{R}}

\newcommand{\N}{\mathbb{N}}

\newcommand{\be}{\begin{equation}}
\newcommand{\ee}{\end{equation}}

\begin{document}
\title[Classification results for critical Schr\"{o}dinger system]{Critical quasi-linear Schr\"{o}dinger system with $p$-Laplacian}

\author{Neng cheng, Wei Dai, Zhao Liu}

\address{School of Mathematical Sciences, Jiangxi Science and Technology Normal University, Nanchang 330038, P. R. China}
\email{Chengn1108@126.com}

\address{School of Mathematical Sciences, Beihang University (BUAA), Beijing 100191, P. R. China, and Key Laboratory of Mathematics, Informatics and Behavioral Semantics, Ministry of Education, Beijing 100191, P. R. China}
\email{weidai@buaa.edu.cn}

\address{School of Mathematical Sciences, Jiangxi Science and Technology Normal University, Nanchang 330038, P. R. China}
\email{liuzhao@mail.bnu.edu.cn}

\thanks{Wei Dai is supported by the NNSF of China (No. 12222102 \& No. 12571113), the National Science and Technology Major Project (2022ZD0116401) and the Fundamental Research Funds for the Central Universities. Zhao Liu is supported by the NNSF of China (No. 12261041) and the Natural Science Foundation of Jiangxi Province (No. 20232ACB211002).}

\begin{abstract}
In this paper, we mainly consider positive solution to the $D^{1,p}(\R^{N})$-critical quasi-linear Schr\"{o}dinger system with $p$-Laplacian:
\begin{equation*}\begin{cases}
-\Delta_p u = u^{\alpha}v^{\beta}  \, \ \ \ \ \ \text{in}\,\ \ \R^N, \\
-\Delta_p v = u^{\beta}v^{\alpha}  \,\  \ \ \ \ \text{in}\,\ \ \R^N,
\end{cases}\end{equation*}
where $1<p<N$, $N\geq2$, $0\leq \alpha \leq \beta,$ and $u,v\in D^{1,p}(\R^N)$. We establish regularity and the sharp estimates on asymptotic behaviors for any positive solution $(u,v)$. Then, we prove that all positive solutions are radially symmetric and strictly decreasing about some point. Furthermore, we obtain the uniqueness and complete classification of positive solutions. Our results extend the uniqueness results in \cite{LM,QS} for $p=2$ to general cases $1<p<N$.
\end{abstract}

\maketitle {\small {\bf Keywords:} Quasi-linear Schr\"{o}dinger system with $p$-Laplacian;  Radial symmetry;  Sharp estimates on asymptotic behaviors; Classification results; The method of moving planes.\\

}

\section{Introduction}
In this paper, we are mainly concerned with positive $D^{1,p}(\R^{N})$-weak solutions to the following $p$-Laplacian Schr\"{o}dinger system with critical exponent:
\begin{align}\label{eq1.1}
\left\{ \begin{array}{ll} \displaystyle
-\Delta_p u = u^{\alpha}v^{\beta}  \, \ \ \ \ \ \mbox{in}\,\ \ \R^N, \\ \\
-\Delta_p v = u^{\beta}v^{\alpha}  \,\  \ \ \ \ \mbox{in}\,\ \ \R^N, \\ \\
u, v \in D^{1,p}(\R^N),\,\,\,u, v>0 \quad \mbox{in}\,\, \ \ \R^N,
\end{array}
\right.\hspace{1cm}
\end{align}
where the full range of $p$ is $1<p<N$, $N\geq2$, and $\Delta_p u:=div (|\nabla u|^{p-2} \nabla u)$ is the usual $p$-Laplace operator, $\alpha, \beta$ are real parameters satisfying
$$ \alpha+ \beta = p^{\star}-1,\quad 0\leq \alpha \leq \beta, $$
$p^{\star}=\frac{Np}{N-p}$ is called Sobolev critical exponent. The function space
$$ D^{1,p}(\R^N):=\left\{ u\in L^{p^{\star}}(\R^N) \, \bigg| \, \int_{\R^N} |\nabla u|^p \mathrm{d}x <+\infty \right\}$$
is the completion of $C_0^\infty$ with respect to the norm $\| u \|:=\| \nabla u \|_{L^p(\R^N)}$. The quasi-linear Schr\"{o}dinger system \eqref{eq1.1} is $D^{1,p}(\R^{N})$-critical in the sense that both \eqref{eq1.1} and the $D^{1,p}$-norms $\| \nabla u \|_{L^p(\R^N)}$, $\| \nabla v\|_{L^p(\R^N)}$ are invariant under the scalings $u\mapsto u_{\lambda}(\cdot):=\lambda^{\frac{N-p}{p}}u(\lambda\cdot)$ and $v\mapsto v_{\lambda}(\cdot):=\lambda^{\frac{N-p}{p}}v(\lambda\cdot)$. System \eqref{eq1.1} is also invariant under arbitrary rotations and translations.

\begin{defn}[$D^{1,p}(\R^{N})$-weak solution]\label{weak}
A $D^{1,p}(\R^{N})$-weak solution of system \eqref{eq1.1} is a pair of functions $(u,v)\in D^{1,p}(\R^N)\times D^{1,p}(\R^N)$ such that, for any $\psi \in C_0^\infty(\R^N)$,
\begin{align}
\left\{ \begin{array}{ll} \displaystyle
\int_{\R^N} |\nabla u|^{p-2} \langle \nabla u, \nabla\psi\rangle \mathrm{d}x = \int_{\R^N} u^{\alpha}v^{\beta}\psi(x)\mathrm{d}x, \\ \\
\int_{\R^N} |\nabla v|^{p-2} \langle \nabla v, \nabla\psi\rangle \mathrm{d}x = \int_{\R^N} u^{\beta}v^{\alpha}\psi(x)\mathrm{d}x.
\end{array}
\right.\hspace{1cm}
\end{align}
\end{defn}

In the special case $p=2$, system \eqref{eq1.1} reduces to the second order semi-linear Schr\"{o}dinger systems with critical exponents
\begin{align}\label{eq1.1'}
\left\{ \begin{array}{ll} \displaystyle
-\Delta u = u^{\alpha}v^{\beta}  \, \ \ \ \ \ \mbox{in}\,\ \ \R^N, \\ \\
-\Delta v = u^{\beta}v^{\alpha}  \,\  \ \ \ \ \mbox{in}\,\ \ \R^N.
\end{array}
\right.\hspace{1cm}
\end{align}
Li and Ma \cite{LM} proved the following uniqueness result for system \eqref{eq1.1'}.
\begin{thm}
    Assume $n\geq 3$, $p=2$, $1\leq \alpha < \beta,$ and $\alpha+\beta =p^\star-1.$ Then any positive classical solution $(u,v)\in L^{\frac{2N}{N-2}}\times L^{\frac{2N}{N-2}}$ of \eqref{eq1.1} satisfies $u=v$.
\end{thm}
Subsequently, Quittner and Souplet \cite{QS} extend the uniqueness result in \cite{LM} from $1\leq \alpha < \beta$ to general $0\leq \alpha \leq \beta$. For more classification results and Liouville type results on semi-linear systems, c.f. e.g. \cite{CDQ0,CGP,CDQ,DQ1,DQ0,Di,LLW,LYZ,SZ,Soup} and the references therein.

In this paper, we aim to prove the classification of solutions to the $D^{1,p}(\mathbb{R}^{N})$-critical quasi-linear Schr\"{o}dinger system \eqref{eq1.1}, which extends the uniqueness results in \cite{LM,QS} for $p=2$ to general cases $1<p<N$. Comparing with the special case $p=2$, we have to deal with the following essential difficulties for general $p\neq2$:
\begin{itemize}
  \item Nonlinear feature of the $p$-Laplacian $\Delta_p$, linear operations can not be exchanged with $\Delta_p$, comparison principles can not be deduced from maximum principles and are much more difficult to be proved.
  \item Absence of Kelvin type transforms, we need to derive the sharp asymptotic estimates on both $(u,v)$ and $(|\nabla u|,|\nabla v|)$ first.
  \item Lack of the Green integral representation formula, indirect methods via the equivalent integral equation do not work anymore.
  \item The $p$-Laplacian is singular elliptic or degenerate elliptic if $1<p<2$ or $p>2$ respectively, there is no uniform form of the (strong) comparison principles for $p$-Laplacian. Hence, we need to discuss the different two cases $1<p<2$ and $p>2$ separately.
\end{itemize}
Thus it would be an quite interesting and challenging problem to study the sharp asymptotic estimates, radial symmetry, strictly radial monotonicity and classification of solutions to \eqref{eq1.1}.

\smallskip

Assuming $u\equiv v$, the quasi-linear Schr\"{o}dinger system \eqref{eq1.1} becomes the following single $D^{1,p}(\mathbb{R}^{N})$-critical quasi-linear equation:
\begin{align}\label{criticeqution}
-\Delta_p u = u^{p^{\star}-1},\,\,\,\,\,\,u\geq0 \qquad \text{in} \,\, \mathbb{R}^{N}.
\end{align}
Serrin \cite{S} and Serrin and Zou \cite{SJZH02}, among other things, proved regularity results for general quasi-linear equations and sharp lower bound on the asymptotic estimate of super-$p$-harmonic functions (see Lemma \ref{estimate2}). Guedda and Veron \cite{GV} proved the uniqueness of radially symmetrical positive solution $U(x)=U(r)$ with $r=|x|$ to \eqref{criticeqution} satisfying $U(0)=b>0$ and $U'(0)=0$, and hence reduced the classification of solutions to the radial symmetry of solutions. For radial symmetry of positive $D^{1,p}(\R^N)$-weak solutions to $D^{1,p}(\mathbb{R}^{N})$-critical quasi-linear equations via the method of moving planes, c.f. \cite{CDL,CDGL,DLL,LD,DP,LDSMLMSB,DLPFRM,LDMR,LDBS04,LPLDHT,OSV,BS16,VJ16} and the references therein.

\smallskip

In the singular elliptic case $1<p<2$, under the $C^{1}\cap W^{1,p}$-regularity assumptions on the positive solutions and several assumptions on the (local) nonlinear term, the radial symmetry and strictly radial monotonicity of the positive solution was obtained in \cite{DLPFRM,LDMR} by using the moving plane technique. Later, in \cite{LDSMLMSB}, Damascelli, Merch\'{a}n, Montoro and Sciunzi refined the moving plane method used in \cite{DLPFRM,LDMR} and derived the radial symmetry of positive $D^{1,p}(\R^N)$-weak solutions to \eqref{criticeqution} without regularity assumptions on solutions. However, they still required in \cite{LDSMLMSB} that the nonlinear term is locally Lipschitz continuous, i.e., $p^{\star} \geq 2$. Subsequently, by applying scaling arguments and the doubling Lemma (see \cite{PPPQPS}), V\'{e}tois \cite{VJ16} first derived a preliminary estimate on the decay rete of positive $D^{1,p}(\R^N)$-weak solution $u$ to \eqref{criticeqution} (i.e., $u(x)\leq \frac{C}{|x|^{\frac{N-p}{p}}}$ for $|x|$ large), combining this preliminary estimate with scaling arguments, the Harnack-type inequalities (c.f. \cite{LDBS,JSLB}), regularity estimates and the boundedness estimate of the quasi-norm $L^{s,\infty}(\mathbb{R}^{N})$ with $s=\frac{N(p-1)}{N-p}$ (Lemma 2.2 in \cite{VJ16}) implied sharp asymptotic estimates of the positive $D^{1,p}(\R^N)$-weak solution $u$ (i.e., $u\sim\left(1+|x|^{\frac{N-p}{p-1}}\right)^{-1}$) and the sharp upper bound on decay rate of $|\nabla u|$ (i.e., $|\nabla u|\leq \frac{C}{|x|^{\frac{N-1}{p-1}}}$ for $|x|$ large) for $1<p<N$. These sharp asymptotic estimates combined with the results in \cite{DLPFRM,LDMR} successfully extended the radial symmetry results for \eqref{criticeqution} in \cite{LDSMLMSB} to the full singular elliptic range $1<p<2$.

\smallskip

For equation \eqref{criticeqution} in the degenerate elliptic case $p>2$, Sciunzi \cite{BS16} developed a new technique based on scaling arguments, the study of limiting profile at infinity, the uniqueness up to multipliers of $p$-harmonic maps in $\mathbb{R}^{N}\setminus\{0\}$ and the sharp asymptotic estimates on $u$ proved in \cite{VJ16}, and derived the sharp lower bound on decay rate of $|\nabla u|$ (i.e., $|\nabla u|\geq \frac{c}{|x|^{\frac{N-1}{p-1}}}$ for $|x|$ large). Thanks to the sharp asymptotic estimates on $u$ and $|\nabla u|$, Sciunzi \cite{BS16} finally proved the classification of positive $D^{1,p}(\R^N)$-weak solution $u$ to \eqref{criticeqution} via the moving planes technique in conjunction with the weighted Poincar\'{e} type inequality and Hardy's inequality and so on.

\smallskip

In \cite{DLL}, Dai, Li and Liu proved the radial symmetry and sharp asymptotic estimates of positive solutions to the following $D^{1,p}(\R^{N})$-critical quasi-linear nonlocal Schr\"{o}dinger-Hartree equation:
\begin{align}\label{a0}
-\Delta_p u =\left(|x|^{-2p}\ast |u|^{p}\right)|u|^{p-2}u \qquad &\mbox{in} \,\, \mathbb{R}^N.
\end{align}
In order to prove the radial symmetry and strictly radial monotonicity of solutions, one of the key ingredients is to prove the sharp asymptotic estimates for positive $D^{1,p}(\R^N)$-weak solution $u$ and $|\nabla u|$. However, due to the nonlocal feature of the convolution $|x|^{-2p}*u^p$ in the Hartree type nonlinearity in \eqref{a0}, it is impossible to use the scaling arguments and the Doubling Lemma as in \cite{VJ16} to get preliminary estimates on upper bounds of asymptotic behaviors for any positive solutions $u$. Moreover, it is also quite difficult to obtain the boundedness estimate of the quasi-norm $\|u \|_{L^{s,\infty}(\R^N)}$ with $s=\frac{N(p-1)}{N-p}$ as in Lemma 2.2 of \cite{VJ16} and hence derive the sharp upper bound estimates from the preliminary estimates as in \cite{VJ16}. In \cite{DLL}, by showing a better preliminary estimates on upper bounds of asymptotic behaviors through the De Giorgi-Moser-Nash iteration method (i.e., $u(x) \leq\frac{C}{1+|x|^{\frac{N-p}{p}+\sigma}}$ for some $\sigma>0$) and combining Theorem 1.5 in \cite{XCL} and Theorem 2.3 in \cite{SJZH02}, the authors are able to overcome these difficulties and establish regularity and the sharp estimates on both upper and lower bounds for any positive solution $u$ and $|\nabla u|$. Subsequently, Cao, Dai and Li \cite{CDL} derived radial symmetry and sharp asymptotic estimates of nonnegative solutions to more general weighted doubly $D^{1,p}$-critical quasi-linear nonlocal elliptic equations with Hardy potential.

\smallskip

Ciraolo, Figalli and Roncoroni \cite{CFR} proved classification of positive $D^{1,p}(\R^N)$-weak solutions to $D^{1,p}(\R^N)$-critical anisotropic $p$-Laplacian equations of type \eqref{criticeqution} in convex cones. Under some energy growth conditions or suitable control of the solutions at $\infty$ in some cases, Catino, Monticelli and Roncoroni \cite{CMR} proved classification results on general positive solution $u$ to \eqref{criticeqution} (possibly having infinite $D^{1,p}(\R^N)$-energy) for $1<p<N$. For $\frac{N+1}{3}\leq p<N$, Ou \cite{Ou} classified general positive solution $u$ to \eqref{criticeqution} (possibly having infinite energy). In \cite{CDGL}, the authors proved Liouville theorems and classification results for (anisotropic) $p$-Laplacian equations in convex cones without finite-energy condition. Their results are the subcritical counterpart of the classification result for the critical case in \cite{CFR}, and extend the Liouville type theorems in $\mathbb{R}^{N}$ in e.g. \cite{SJZH02} to general convex cones $\mathcal{C}$. Their results also extend the classification result of \cite{Ou} in $\mathbb{R}^{N}$ to general convex cones $\mathcal{C}$, and remove the finite-energy assumption in \cite{CFR} for the standard non-anisotropic $p$-Laplacian. For more literatures on quantitative and qualitative properties of solutions for quasi-linear equations involving $p$-Laplacians with local nonlinearities, please c.f. \cite{ACP,A,B,CDPSYS,CMR,CV,DDGL,LD,DFSV,DP,LDBS04,LDBS,DDH,DYZ,ED83,Di,FSV,GL,LPLDHT,Lieberman,OSV,BSR14,JSLB,SJZH02,Tei,PT,JLV,ZL} and the references therein.

\medskip

We first establish the following regularity result and sharp asymptotic estimates on $(u,v)$ and $(|\nabla u|,|\nabla v|)$, due to the absence of Kelvin type transforms.
\begin{thm}
[Regularity and sharp asymptotic estimates]\label{th2.1-}
Assume $1<p<N$, and let $(u,v)$ be a pair of positive $D^{1,p}(\R^{N})$-weak solutions to the $D^{1,p}(\mathbb{R}^{N})$-critical quasi-linear Schr\"{o}dinger system \eqref{eq1.1}. Then $u, v \in C_{loc}^{1,\theta} (\R^N) \cap L^\infty(\R^N)$. Moreover, there exist $c_0, C_0,c'_0, C'_0, R_0 > 0$ such that
\begin{equation}\label{eq0806}
  \frac{c_0}{1+|x|^\frac{N-p}{p-1}} \leq u(x) , v(x) \leq \frac{C_0}{1+|x|^\frac{N-p}{p-1}} \,\,\,\,\,\,\,\,\mbox{in}\,\,\ \R^N,
\end{equation}
\begin{align}\label{eq0806+}
   \frac{c'_0}{|x|^\frac{N-1}{p-1}} \leq |\nabla u(x)|, |\nabla v(x)|\leq \frac{C'_0}{|x|^\frac{N-1}{p-1}} \,\,\,\,\,\,\,\,\mbox{in}\,\,\ \R^N\setminus B_{R_0}(0).
\end{align}
\end{thm}

\medskip

Thanks to the sharp asymptotic estimates on $(u,v)$ and $(|\nabla u|,|\nabla v|)$ in Theorem \ref{th2.1-}, inspired by \cite{DLL,LDSMLMSB,BS16,VJ16} etc., we are ready to prove the following theorem on radial symmetry and strictly radial monotonicity of positive $D^{1,p}(\mathbb{R}^{N})$-weak solutions to \eqref{eq1.1} via the method of moving planes.
\begin{thm}\label{th2}
Assume $1<p<N$ and let $ (u,v)$ be a pair of positive $D^{1,p}(\mathbb{R}^{N})$-weak solutions of the $D^{1,p}(\mathbb{R}^{N})$-critical quasi-linear Schr\"{o}dinger system \eqref{eq1.1}. Then, $(u,v)$ are radially symmetric and strictly decreasing about some point $x_0\in\R^N$. Furthermore, the positive solutions $(u,v)$ must assume the form
\[u(x)=v(x)=\left[ \frac{\lambda^{\frac{1}{p-1}}\left(N^{\frac{1}{p}}\left(\frac{N-p}{p-1}\right)^{\frac{p-1}{p}}\right)}{\lambda^{\frac{p}{p-1}}+|x-x_0|^{\frac{p}{p-1}}} \right]^{\frac{N-p}{p}}.\]
\end{thm}

\smallskip

For more literatures on the classification results and other quantitative and qualitative properties of solutions for various PDE and IE problems via the methods of moving planes, please refer to \cite{CGS,CL,CL1,CLL,CLO,CDQ,DLS,DQ,LD,DP,LDSMLMSB,DLPFRM,LDMR,LDBS04,GNN,LPLDHT,Lin,BS16,S,VJ16,WX} and the references therein.

\smallskip

The rest of our paper is organized as follows. In Section \ref{sc2}, we will give some preliminaries on important inequalities, (strong) comparison principles, strong maximum principle and H\"{o}pf's Lemma, related regularity results and properties of the critical set etc.. Section \ref{sc3} is devoted to the proof of the key ingredients: Theorem \ref{th2.1-}, i.e., the regularity and the sharp asymptotic estimates for positive weak solutions $(u,v)$ and their gradients $(|\nabla u|,|\nabla v|)$. Finally, we will carry out our proof of Theorem \ref{th2} in Section \ref{sc4}.

\smallskip

In the following, we will use the notation $B_{R}$ to denote the ball $B_{R}(0)$ centered at $0$ with radius $R$, and use $C$ to denote a general positive constant that may depend on $N$, $p$ and $(u,v)$, and whose value may differ from line to line.

\section{Preliminaries}\label{sc2}
In this section, we will give some useful tools in our proofs of Theorems  \ref{th2.1-} and \ref{th2}, including some important inequalities, (strong) comparison principles, strong maximum principle and H\"{o}pf's Lemma, related regularity results and properties of the critical set etc. For clarity of presentation, we will split this Section into two sub-sections.

\subsection{Some key inequalities}\label{sc2.1}

\medskip

In order to apply the moving planes method to $p$-Laplace equations, we will frequently use the following basic point-wise estimate (Lemma \ref{BSICI}), the weighted Hardy-Sobolev inequality (Lemma \ref{HDSI}) and the strong comparison principles (Lemmas \ref{th3.1}-\ref{th3.2}).
\begin{lem}[Lemma 2.1 in \cite{LD}] \label{BSICI}
For any $ p>1 $, the following elementary point-wise inequality
\begin{align}\label{eq0802}
[|\eta|^{p-2}\eta-|\eta^\prime|^{p-2}\eta^\prime][\eta-\eta^\prime] &\geq C' (|\eta|+|\eta^\prime|)^{p-2}|\eta-\eta^\prime|^2; \\
\left| |\eta|^{p-2}\eta-|\eta^\prime|^{p-2}\eta^\prime \right| &\leq C'' (|\eta|-|\eta^\prime|)^{p-1},\,\,\,\,\,\,\,\mbox{if}\,\, 1<p\leq 2 \nonumber
\end{align}
hold for all $\eta,\eta^\prime \in \R^N$ with $|\eta|+|\eta^\prime|>0$, where $p>1$ and $C', C''>0$ depends only on $N$ and $p$.
\end{lem}

\begin{lem}[Weighted Hardy-Sobolev inequality, c.f. \cite{LDMR}]\label{HDSI}
Let $ u \in C_0^1(\R^N)$ and $q\geq 1$. Then for any $ s > q-N $ and $ R\geq 0 $ we have
$$ \int_{\R^N\setminus B_R(0)} |u|^q |x|^{s-q} \mathrm{d}x \leq \left( \frac{q}{N-q+s} \right)^q \int_{\R^N\setminus B_R(0)} |\nabla u|^q |x|^s \mathrm{d}x.$$
%The same inequality holds if we substitute u with its positive (negative) part.
\end{lem}

\begin{lem}[Theorem 1.4 in \cite{LDBS}, c.f. also Theorem 2.1 in \cite{OSV}]\label{th3.1}
Let $\frac{2N+2}{N+2}<p<2$ or $p>2$ and $u,v \in C^1(\overline{\Omega})$, where $\Omega$ is a bounded smooth domain of $\R^N$. Suppose that either $u$ or $v$ is a weak solution of
\begin{align*}
\left\{ \begin{array}{ll}
-\Delta_p u =f(x,u) \,\,\,\,\,\, & \mbox{in}\,\,\Omega, \\
u>0 \,\,\,\,\,\,& \mbox{in}\,\,\Omega, \\
u=0 \,\,\,\,\,\,& \mbox{on}\,\,\partial\Omega \\
\end{array}
\right.\hspace{1cm}
\end{align*}
with $f:\overline{\Omega}\times [0,\infty)\to \R$ is a continuous function which is positive and of class $C^1$ in $\Omega\times(0,\infty)$. Assume that
$$ -\Delta_p u + \Lambda u \leq  -\Delta_p v + \Lambda v \qquad \,\,\,\mbox{and} \quad \,u\leq v\,\quad\mbox{in}\,\,\Omega,$$
where $\Lambda\in\R$. Then $u\equiv v$ in $\Omega$ unless $u<v$ in $\Omega$.
\end{lem}

\begin{lem}[Theorem 1.4 in \cite{LD}]\label{th3.2}
Suppose $\Omega$ is a bounded domain of $\R^N$, $1<p<\infty$ and let $u,v \in C^1(\Omega)$ weakly satisfy
$$ -\Delta_p u + \Lambda u \leq  -\Delta_p v + \Lambda v \quad\,\,\,\mbox{and}\quad\,u\leq v\,\,\,\,\,\mbox{in}\,\,\Omega,$$
and denote by $Z_v^u:=\{ x\in\Omega\,\mid\,|\nabla u|=|\nabla v| =0 \}$. Then if there exists $x_0\in\Omega\setminus Z_v^u$ with $u(x_0)=v(x_0)$, then $u\equiv v$ in the connected component of $\Omega\setminus Z_v^u$ containing $x_0$. The same result still holds if, more generally,
$$ -\Delta_p u - f(u) \leq  -\Delta_p v - f(v) \quad\,\,\,\mbox{and}\,\quad u\leq v \quad\,\mbox{in}\,\,\Omega$$
with $f:\R \to \R$ locally Lipschitz continuous.
\end{lem}

We recall a version of the strong maximum principle and the H\"{o}pf Lemma for the $p$-Laplacian, c.f. \cite{JLV}, see also Theorem 2.1 in \cite{LDBS04}, Theorem 2.4 in \cite{LDSMLMSB} or \cite{LD,PS}.
\begin{lem}[Strong Maximum Principle and H\"{o}pf's Lemma, e.g. \cite{JLV}, Theorem 2.1 in \cite{LDBS04}]\label{hopf}
Let $\Omega$ be a domain in $\mathbb{R}^{N}$ and suppose that $u\in C^{1}(\Omega)$, $u\geq0$ in $\Omega$ weakly solves
\[-\Delta_{p}u+cu^{q}=g\geq0 \qquad \text{in} \,\,\Omega\]
with $1<p<+\infty$, $q\geq p-1$, $c\geq0$ and $g\in L^{\infty}_{loc}(\Omega)$. If $u\not\equiv0$, then $u>0$ in $\Omega$. Moreover, for any point $x_{0}\in\partial\Omega$ where the interior sphere condition is satisfied and $u\in C^{1}(\Omega\cup\{x_{0}\})$ and $u(x_{0})=0$, we have that $\frac{\partial u}{\partial s}(x_{0})>0$ for any inward directional derivative (this means that if $y$ approaches $x_{0}$ in a ball $B\subseteq\Omega$ that has $x_{0}$ on its boundary, then $\lim\limits_{y\rightarrow x_{0}}\frac{u(y)-u(x_{0})}{|y-x_{0}|}>0$).
\end{lem}

\subsection{Regularity results and Properties of the critical Set}\label{sc2.2}

In this subsection, we will give the integrability properties of $\frac{1}{|\nabla u|}, \frac{1}{|\nabla v|}$ and the connectivity properties of the critical set $Z_{u},Z_{v}$ for the nonnegative weak solutions $u,v$ of \eqref{eq1.1}, and the weighted Poincar\'{e} type inequality.

\smallskip

From Theorem \ref{th2.1-}, the nonnegative weak solutions $(u,v)$ of \eqref{eq1.1} satisfies $u, v\in L^{\infty}(\mathbb{R}^{N})\cap C^{1,\theta}_{loc} (\R^N)$. Moreover, using standard elliptic regularity, the solution $u\in C_{loc}^{1,\theta} (\Omega)$ belongs to the class $C^2(\Omega\setminus Z_u)$, where $\Omega\subset\subset\R^N$ and $Z_u=\{ x\in\Omega \ | \ |\nabla u|=0 \}$ (see \cite{ED83,DGNT,PT}), which is the
same for $v$. We need the following Hessian and reversed gradient estimates from \cite{BSR14}.
\begin{lem}[Hessian and reversed gradient estimates, \cite{BSR14}] \label{lm.4}
Let $\Omega\subset\subset\R^N$ and $1<p<\infty$. Assume $u\in W^{1,p}(\Omega)$ be a weak solution to
\begin{align}\label{eq10.252252}
-\Delta_p u =f(x) \quad\,\,\,\mbox{in}\,\,\Omega,
\end{align}
where $f(x)$ satisfies
\begin{align}\label{eq26.30.001}
    f(x) \in L^s(\Omega) \quad with \quad s>N
\end{align}
and condition ($I_\alpha$)  i.e., for some $0<\mu<\alpha(p-1)$, there holds:
\begin{align}\label{eq26.30.002}
f(x)\in W^{1,m}(\Omega)
\end{align}
for some $m>\frac{N}{2(1-\mu)}$, and given any $x_0\in \Omega$, there exist $C(x_0,\mu)>0$ and $\rho(x_0,\mu)>0$ such that, $B_{\rho(x_0,\mu)}(x_{0})\subset \Omega$ and
\begin{align}\label{eq26.30.003}
|f(x)|\geq C(x_0,\mu) |x-x_0|^\mu \qquad \mbox{in}\,\,B_{\rho(x_0,\mu)}(x_{0}).
\end{align}
Then, for any $x_0\in Z_u$ and $B_{2\rho}(x_0)\subset \Omega$, we have
\begin{align}\label{eq10.25.52}
\int_{B_\rho(x_0)} \frac{|\nabla u|^{p-2-\beta}(x)}{|x-y|^\gamma} |\nabla  u_{x_i}(x)|^2 \mathrm{d}x < C,\quad\,\,i=1,\cdots,N
\end{align}
uniformly for any $y\in \Omega$, where $0\leq\beta<1$, $\gamma=0$ if $N=2$, $\gamma<N-2$ if $N \geq 3$ and
$$C=C(N,p,x_0,\gamma,\beta,u,f,\rho)>0.$$
Moreover, we have, uniformly for any $y\in \Omega$,
\begin{align}\label{eq10.25.2.1}
\int_{B_\rho(x_0)} \frac{1}{|\nabla u|^t} \frac{1}{|x-y|^\gamma}\mathrm{d}x <\widetilde{C},
\end{align}
where $\max\{p-2,0\}\leq t < p-1$ and $\gamma=0$ if $N=2$, $\gamma<N-2$ if $N \geq 3$ and $$\widetilde{C}=\widetilde{C}(N,p,x_0,\gamma,t,u,f,\rho)>0.$$
\end{lem}

\smallskip

From Lemma \ref{lm.4}, we can infer the following corollary.
\begin{cor}\label{re2333}
Let $1<p<N$ and $u$ be a weak positive solution of \eqref{eq1.1}. Setting $\rho_u:=|\nabla u|^{p-2}, \rho_v:=|\nabla v|^{p-2}$, for any $\Omega \subset\subset\R^N$, we have
\begin{align}\label{eq10.25.2}
\int_{\Omega} \frac{1}{\rho_u^t} \frac{1}{|x-y|^\gamma}\mathrm{d}x < C \quad and \quad \int_{\Omega} \frac{1}{\rho_v^t} \frac{1}{|x-y|^\gamma}\mathrm{d}x < C
\end{align}
uniformly for any $y\in\Omega$, where $\max\{p-2,0\}\leq (p-2)t<p-1$ and $\gamma=0$ if $N=2$, $\gamma<N-2$ if $N \geq 3$.
\end{cor}
\begin{proof}
 From Theorem \ref{th2.1-}, we know that $u^{\alpha}v^{\beta}$ and $u^{\beta}v^{\alpha}$ satisfy \eqref{eq26.30.001} and the condition ($I_\alpha$)
 (i.e., \eqref{eq26.30.002}-\eqref{eq26.30.003}) in Lemma \ref{lm.4}. Noting that $u\in C^2(\Omega\setminus Z_u)$, it follows from the finite covering theorem and \eqref{eq10.25.2.1} in Lemma \ref{lm.4} that
the first inequality of \eqref{eq10.25.2} holds. As the same inequality  for $v$ holds, it finishes the proof of Corollary \ref{re2333}.
\end{proof}

\begin{rem}\label{re2334}
Let $(u,v)$ be a pair of positive weak solutions of \eqref{eq1.1} and let $Z_u = \{x \in \R^N|\ |\nabla u(x)| = 0 \}$. Clearly, it follows from Theorem \ref{th2.1-} that the critical set $Z_u\subset B_{R_0}$ and $Z_u$ is a closed set. Moreover, \eqref{eq10.25.2} in Corollary \ref{re2333} implies that the Lebesgue measure of $Z_{u}$ is zero, i.e., $|Z_u|=0$. In the same argument, we have $|Z_v|=0$.
\end{rem}

From Lemma \ref{th3.2}, we can clearly see that it is very useful to know whether $\Omega\setminus {\left(Z_u\cup Z_v\right)}$ are connected or not. We have the following lemma on the connectivity properties of the critical set from \cite[Theorem 4.1 and Remark 4.1]{LDBS04}.
\begin{lem}[Properties of the Critical Set, Theorem 4.1 in \cite{LDBS04}]\label{lm.7}
Assume $1<p<N$ and let $u, v$ be a  weak positive solution of \eqref{eq1.1}. Let us define the critical set
$$Z_u = \{ x\in\Omega\ | \ |\nabla u(x)|=0 \}, \,\, Z_v = \{ x\in\Omega\ | \ |\nabla v(x)|=0 \},$$
where $\Omega$ is a general bounded domain. %and $B_{R_0}\subset\Omega$.
Then $\Omega\setminus {\left(Z_u\cup Z_v\right)}$ do not contain any connected component $C$ such that $\overline{C} \subset \Omega$. Moreover, if we assume that $\Omega$ is a smooth bounded domain with connected boundary, it follows that $\Omega\setminus {\left(Z_u\cup Z_v\right)}$ are connected.
\end{lem}

Finally, we give the weighted Poincar\'{e} type inequality inequality, which plays a key role in the proof of Theorem \ref{th2} via the moving planes techniques in the degenerate elliptic case $p>2$. To this end, we need the following definition.
\begin{defn}[Weighted Sobolev spaces]\label{lm.3-}
Assume $\Omega\subset\R^N$ is a bounded domain and let $\rho\in L^1(\Omega)$ be a positive function. We define $H^{1,p}(\Omega,\rho)$ as the completion of $C^1(\overline{\Omega})$ (or $C^\infty(\overline{\Omega})$) with the norm
$$ \| u\|_{ H^{1,p}(\Omega,\rho) } := \| u \|_{L^p(\Omega)} + \| \nabla u\|_{L^p(\Omega,\rho)}$$
and $\| \nabla u\|^p_{L^p(\Omega,\rho)} = \int_\Omega  \rho  |\nabla u |^p \mathrm{d}x$. In this way $H^{1,2}(\Omega,\rho)$ is a Hilbert space, and we also define $H_0^{1,p}(\Omega,\rho)$ as the closure of $C_0^\infty(\overline{\Omega})$ in $H^{1,p}(\Omega,\rho)$.
\end{defn}

\begin{lem}[Weighted Poincar\'{e} type inequality, \cite{LDBS04,FV,LPLDHT}]\label{lm.3}
Assume $\Omega\subset\R^N$ is a bounded domain and let $\rho\in L^1(\Omega)$ be a positive function such that
\begin{align}\label{eq10.25.3}
\int_\Omega \frac{1}{\rho |x-y|^\gamma}\mathrm{d}y < C, \,\,\,\quad \forall \,\, x\in\Omega,
\end{align}
where $\gamma<N-2$ if $N \geq 3$ and $\gamma=0$ if $N = 2$. Let $u\in H^{1,2}(\Omega,\rho)$ be such that
\begin{align}\label{eq10.25.4}
|u(x)|\leq C \int_\Omega \frac{|\nabla u(y)|}{ |x-y|^{N-1}}\mathrm{d}y,\quad\,\,\,\forall \,\, x\in\Omega.
\end{align}
Then we have
\begin{equation}\label{wp}
  \int_\Omega u^2(x) \mathrm{d}x \leq C(\Omega) \int_\Omega \rho |\nabla u(x)|^2 \mathrm{d}x,
\end{equation}
where $C(\Omega)\to 0$ if $|\Omega|\to 0$. The same inequality holds for $u\in H^{1,2}_0(\Omega,\rho)$.
\end{lem}

\begin{rem}\label{rem0}
Assume $1<p<N$ and let $u$ be a positive weak solution of \eqref{eq1.1}. By Theorem \ref{th2.1-}, we know that, for any $p\geq2$, $\rho:=|\nabla u|^{p-2}\in L^1(\Omega)$ for any $\Omega\subset\subset\R^N$. In addition, from \eqref{eq10.25.2} in Corollary \ref{re2333} and Remark \ref{re2334}, we deduce that \eqref{eq10.25.3} holds for $\rho=|\nabla u|^{p-2}$ and $\rho>0$ a.e.. Hence the weighted Poincar\'{e} type inequality \eqref{wp} in Lemma \ref{lm.3} holds with $\rho=|\nabla u|^{p-2}$ for $p\geq2$. Obviously, we have the same conclusion for $v$.
\end{rem}

\section{Proof of Theorem \ref{th2.1-}: the sharp estimates on asymptotic behaviors of positive solutions}\label{sc3}

In this section, we will first carry out our proof of Theorem \ref{th2.1-} and establish the sharp estimates on asymptotic behaviors of positive weak solutions to the equation \eqref{eq1.1}.

\smallskip

In order to investigate the sharp asymptotic behavior of positive solution $(u, v)$ to \eqref{eq1.1} and $|\nabla u|,|\nabla v|$, we first consider the following system:
\begin{align}\label{eq10.25.1}
\left\{ \begin{array}{ll}
\int_{\R^N} |\nabla u|^{p-2} \langle \nabla u, \nabla\phi\rangle \mathrm{d}x = \int_{\R^N}u^{\alpha}v^{\beta}\phi(x)\mathrm{d}x, \,\,\,\,&  \phi \in C_0^\infty(\R^N). \\ \\
\int_{\R^N} |\nabla v|^{p-2} \langle \nabla v, \nabla\psi\rangle \mathrm{d}x = \int_{\R^N}u^{\beta}v^{\alpha}\psi(x)\mathrm{d}x, \,\,\,\,&  \psi \in C_0^\infty(\R^N). \\ \\
u,v \in D^{1,p}(\R^N), \,\,\,\,u,v>0\,\,\,\,  \mbox{in}\, \  \R^N.
\end{array}
\right.\hspace{1cm}
\end{align}

\subsection{Local boundedness}\label{sc3.1}
Inspired by \cite{CDPSYS,ED83,PT}, we can prove the $C^{1,\alpha}$-regularity and a better preliminary estimate on decay rate in the following Lemmas \ref{lm.8} and \ref{lm.5} via the De Giorgi-Moser-Nash iteration techniques, which goes back to e.g. \cite{TMOSER}.
\begin{lem}[$C^{1,\alpha}$-regularity]\label{lm.8}
Let $(u, v)$ be a pair of solutions to the equations \eqref{eq1.1} with $p>1$. Then $(u, v)\in C^{1,\alpha}(\R^N)\times C^{1,\alpha}(\R^N)$ for some $0<\alpha<\min\{1,\frac{1}{p-1}\}$.
\end{lem}
\begin{proof}
Lemma \ref{lm.8} can be proved via the De Giorgi-Morse-Nash iteration method. Let $\Omega \subset\subset \R^N$ be any bounded domain. For any $q\geq1$ and $k>0$, let
$$F(t)=\left\{ \begin{array}{ll}
                t^q,\,\,\,\,&  \mbox{if}\,\, 0\leq t\leq k, \\
                qk^{q-1}t-(q-1)k^q,\,\,\,\,&  \mbox{if}\,\, t\geq k, \\
                  \end{array}
        \right.$$
and
$$G(t)=F(t)[F^\prime(t)]^{p-1}=\left\{ \begin{array}{ll}
                q^{p-1} t^{(q-1)p+1},\,\,\,\,&  \mbox{if}\,\, 0\leq t\leq k, \\
                q^{p-1} k^{(q-1)p} [qt-(q-1)k],\,\,\,\,&  \mbox{if}\,\, t\geq k. \\
                  \end{array}
        \right.$$
It is easy to see that $G$ is a piecewise $C^1$--function with only one corner at $t=k$ and
\begin{align}\label{eq10.28.1}
 G^\prime(t)=\left\{ \begin{array}{ll}
                [(q-1)p+1]q^{-1}[F^\prime (t)]^p,\,\,\,\,&  \mbox{if}\,\, 0\leq t\leq k, \\
                {[F^\prime(t)]}^p,\,\,\,\,&  \mbox{if}\,\, t\geq k, \\
                  \end{array}
        \right.
\end{align}
and
\begin{align}\label{eq10.28.2}
 t^{p-1} G(t)\leq q^{p-1}{[F(t)]}^p,
\end{align}
where $[(q-1)p+1]q^{-1}\geq1$ and $C$ is independent of $k$.

Now we choose $q>1$, %such that $qp \leq p^{\star}$,
and let
$$\phi=\xi^p G(u) \quad \text{and} \quad \psi=\xi^p G(v), $$
where $\xi \in C^\infty_0 (\Omega)$ to be determined later. Then the equation \eqref{eq10.25.1} yields
$$ \int_{\Omega} \xi^p G^\prime(u) |\nabla u|^p \mathrm{d}x + p\int_{\Omega} \xi^{p-1}| \nabla u|^{p-2} G(u) \nabla u \cdot \nabla \xi \mathrm{d}x = \int_{\Omega} u^{\alpha}v^{\beta} \xi^p G(u) \mathrm{d}x.$$
$$ \int_{\Omega} \xi^p G^\prime(v) |\nabla v|^p \mathrm{d}x + p\int_{\Omega} \xi^{p-1}| \nabla v|^{p-2} G(v) \nabla v \cdot \nabla \xi \mathrm{d}x = \int_{\Omega} u^{\beta}v^{\alpha} \xi^p G(v) \mathrm{d}x.$$
Noting that $G(x)$ is increasing, for any $\varepsilon\in(0,1)$, by direct computations and Young's inequality with $\varepsilon$, the previous identity becomes
\begin{align*}
&\quad \int_{\Omega} \xi^p [F^\prime(u)]^p |\nabla u|^p \mathrm{d}x \leq \int_{\Omega} \xi^p G^\prime(u) |\nabla u|^p \mathrm{d}x \\
&\leq p \int_{\Omega} \xi^{p-1}| \nabla u|^{p-1} [F^\prime(u)]^{p-1} F(u) |\nabla \xi| \mathrm{d}x + \int_{\Omega} u^{\alpha}v^{\beta} \xi^p G(u) \mathrm{d}x \\
&\leq \varepsilon \int_{\Omega} \xi^p [F^\prime(u)]^p |\nabla u|^p \mathrm{d}x + C_\varepsilon \int_{\Omega}  [F(u) |\nabla \xi|]^p \mathrm{d}x\\
&\quad + C \int_{\Omega \cap \{u> v\}} u^{p^{\star}-1} \xi^p G(u)\mathrm{d}x + C \int_{\Omega \cap \{u\leq v\}} v^{p^\star-1} \xi^p G(v) \mathrm{d}x\\
&\leq \varepsilon \int_{\Omega} \xi^p [F^\prime(u)]^p |\nabla u|^p \mathrm{d}x + C_\varepsilon \int_{\Omega}  [F(u) |\nabla \xi|]^p \mathrm{d}x\\
&\quad + Cq^{p-1} \int_{\Omega \cap \{u> v\}} u^{p^{\star}-p} [\xi F(u)]^p\mathrm{d}x + Cq^{p-1} \int_{\Omega \cap \{u\leq v\}} v^{p^\star-p}  [\xi F(v)]^p \mathrm{d}x\\
&\leq \varepsilon \int_{\Omega} \xi^p [F^\prime(u)]^p |\nabla u|^p \mathrm{d}x +  C_\varepsilon \int_{\Omega}  [F(u) |\nabla \xi|]^p \mathrm{d}x\\
&\quad+ Cq^{p-1}\|u\|_{L^{p^{\star}}(\Omega\cap supp(\xi))}^{p^{\star}-p}\left(\int_{\Omega} [\xi F(u)]^{p^{\star}} \mathrm{d}x\right)^\frac{p}{p^{\star}}+Cq^{p-1}\|v\|_{L^{p^{\star}}(\Omega\cap supp(\xi))}^{p^{\star}-p} \left(\int_{\Omega} [\xi F(v)]^{p^{\star}} \mathrm{d}x\right)^\frac{p}{p^{\star}},
\end{align*}
where we also have used \eqref{eq10.28.1} and \eqref{eq10.28.2}, as well as the definition of $G(u)$. Then, taking $\varepsilon=\frac{1}{2}$, we get
\begin{align*}
\quad \int_{\Omega} \xi^p  [F^\prime(u)]^p |\nabla u|^p \mathrm{d}x &\leq C \int_{\Omega}  [F(u) |\nabla \xi|]^p \mathrm{d}x\\
& \quad+ Cq^{p-1}\|u\|_{L^{p^{\star}}(\Omega\cap supp(\xi))}^{p^{\star}-p}\left(\int_{\Omega} [\xi F(u)]^{p^{\star}} \mathrm{d}x\right)^\frac{p}{p^{\star}}\\
& \quad+Cq^{p-1}\|v\|_{L^{p^{\star}}(\Omega\cap supp(\xi))}^{p^{\star}-p} \left(\int_{\Omega} [\xi F(v)]^{p^{\star}} \mathrm{d}x\right)^\frac{p}{p^{\star}} .
\end{align*}

Using the Sobolev inequality, we can derive
\begin{align}\label{eq10.28.3}
\left( \int_{\Omega} |\xi F(u)|^{p^{\star}}\mathrm{d}x  \right)^\frac{p}{p^{\star}} &\leq \int_{\Omega} |\xi \nabla F(u) + F(u) \nabla\xi|^p \mathrm{d}x \\
&\leq C_1 \int_{\Omega} \xi^p  [F^\prime(u)]^p |\nabla u|^p \mathrm{d}x + C_1 \int_{\Omega}  [F(u) |\nabla \xi|]^p \mathrm{d}x \nonumber\\
& \leq C \int_{\Omega}  [F(u) |\nabla \xi|]^p \mathrm{d}x \nonumber\\
& \quad+ Cq^{p-1}\|u\|_{L^{p^{\star}}(\Omega\cap supp(\xi))}^{p^{\star}-p}\left(\int_{\Omega} [\xi F(u)]^{p^{\star}} \mathrm{d}x\right)^\frac{p}{p^{\star}}\nonumber \\
& \quad+Cq^{p-1}\|v\|_{L^{p^{\star}}(\Omega\cap supp(\xi))}^{p^{\star}-p} \left(\int_{\Omega} [\xi F(v)]^{p^{\star}} \mathrm{d}x\right)^\frac{p}{p^{\star}}. \nonumber
\end{align}
Similarly, we can also derive
\begin{align}\label{eq10.28.3-}
    \left( \int_{\Omega} |\xi F(v)|^{p^{\star}}\mathrm{d}x  \right)^\frac{p}{p^{\star}} & \leq C \int_{\Omega}  [F(v) |\nabla \xi|]^p \mathrm{d}x \\
& \quad+ Cq^{p-1}\|v\|_{L^{p^{\star}}(\Omega\cap supp(\xi))}^{p^{\star}-p}\left(\int_{\Omega} [\xi F(v)]^{p^{\star}} \mathrm{d}x\right)^\frac{p}{p^{\star}}\nonumber \\
& \quad+Cq^{p-1}\|u\|_{L^{p^{\star}}(\Omega\cap supp(\xi))}^{p^{\star}-p} \left(\int_{\Omega} [\xi F(u)]^{p^{\star}} \mathrm{d}x\right)^\frac{p}{p^{\star}}.\nonumber
\end{align}
By \eqref{eq10.28.3} and \eqref{eq10.28.3-}, we have
\begin{align}\label{loc_bound}
    & \quad \left( \int_{\Omega} |\xi F(u)|^{p^{\star}}\mathrm{d}x  \right)^\frac{p}{p^{\star}} +\left( \int_{\Omega} |\xi F(v)|^{p^{\star}}\mathrm{d}x  \right)^\frac{p}{p^{\star}}\\
    &\leq C \int_{\Omega}  [F(u) |\nabla \xi|]^p \mathrm{d}x+  C \int_{\Omega}  [F(v) |\nabla \xi|]^p \mathrm{d}x \nonumber\\
    & \quad+ Cq^{p-1}\|u\|_{L^{p^{\star}}(\Omega\cap supp(\xi))}^{p^{\star}-p}\left(\int_{\Omega} [\xi F(u)]^{p^{\star}} \mathrm{d}x\right)^\frac{p}{p^{\star}}\nonumber \\
& \quad+Cq^{p-1}\|v\|_{L^{p^{\star}}(\Omega\cap supp(\xi))}^{p^{\star}-p} \left(\int_{\Omega} [\xi F(v)]^{p^{\star}} \mathrm{d}x\right)^\frac{p}{p^{\star}}.\nonumber
\end{align}
Since $u, v\in L^{p^{\star}}$, for any $x_0\in\Omega$, we can choose $B(x_0,2R)\subsetneq\Omega$ such that
$$C\|u\|_{L^{p^{\star}}(B(x_0,2R))}^{p^{\star}-p}<\frac{1}{2q^{p-1}}, \,\,\,\,\, C\|v\|_{L^{p^{\star}}(B(x_0,2R))}^{p^{\star}-p} <\frac{1}{2q^{p-1}}$$
and $supp(\xi) \subset B(x_0, 2R)$. From \eqref{loc_bound}, we get
\begin{align}\label{eq10.28.4-}
& \quad \left( \int_{\Omega} |\xi F(u)|^{p^{\star}}\mathrm{d}x  \right)^\frac{p}{p^{\star}} +\left( \int_{\Omega} |\xi F(v)|^{p^{\star}}\mathrm{d}x  \right)^\frac{p}{p^{\star}}\\
 &\ \ \ \leq C \int_{\Omega} \left( [F(u) |\nabla \xi|]^p + [F(v) |\nabla \xi|]^p \right)\mathrm{d}x \nonumber.
\end{align}
Letting $k\to\infty$ and taking the limit, we conclude
\begin{align*}
    &\quad \left( \int_{\Omega} \xi^{p^{\star}} u^{qp^{\star}} \mathrm{d}x  \right)^\frac{p}{p^{\star}}+ \left( \int_{\Omega} \xi^{p^{\star}} v^{qp^{\star}} \mathrm{d}x  \right)^\frac{p}{p^{\star}}\\
    &\ \ \ \leq C \int_{\Omega} (|\nabla \xi|^p u^{qp} + |\nabla \xi|^p v^{qp})\mathrm{d}x
\end{align*}
for any $q>1$.  % with $q p \leq p^{\star}$.
Noting that $\frac{p}{p^{\star}}<1$, we have
\begin{align}\label{eq10.28.4}
\left( \int_{\Omega} \xi^{p^{\star}}( u^{qp^{\star}} +v^{qp^{\star}} )\mathrm{d}x  \right)^\frac{p}{p^{\star}} &\leq \left( \int_{\Omega} \xi^{p^{\star}} u^{qp^{\star}} \mathrm{d}x  \right)^\frac{p}{p^{\star}}+ \left( \int_{\Omega} \xi^{p^{\star}} v^{qp^{\star}} \mathrm{d}x  \right)^\frac{p}{p^{\star}}\\
& \leq C \int_{\Omega} |\nabla \xi|^p (u^{qp} +v^{qp})\mathrm{d}x.\nonumber
\end{align}
Therefore, for any $R \leq h < h^\prime < 2R$, we can choose $\xi=1$ in $B(x_0,h)$, and $\xi=0$ in $\Omega\setminus B(x_0,h')$, $0\leq \xi \leq 1$ and $|\nabla\xi|<\frac{2}{h^\prime-h}$, and hence deduce from \eqref{eq10.28.4} that
\begin{align*}
\left( \int_{B(x_0,h)} (u^{qp^{\star}} +v^{qp^{\star}}) \mathrm{d}x  \right)^\frac{p}{p^{\star}}
\leq  \frac{C}{(h^\prime-h)^p} \int_{B(x_0,h^\prime)} (u^{qp}+ v^{qp})\mathrm{d}x,\,\,\,\,\,\,\forall \,\,q>1,
\end{align*}
which implies
\begin{align}\label{eq10.28.5}
\left( \int_{B(x_0,h)} (u^{qp^{\star}} +v^{qp^{\star}}) \mathrm{d}x  \right)^\frac{1}{qp^{\star}}
\leq  \left( \frac{C}{h^\prime-h} \right)^\frac{1}{q} \left( \int_{B(x_0,h^\prime)} (u^{qp}+ v^{qp}) \mathrm{d}x  \right)^\frac{1}{qp},\,\,\,\,\,\,\forall \,\,q>1.
\end{align}

Let $\chi:=\frac{p^{\star}}{p}=\frac{N}{N-p}$, and define $r_{i}:=R+\frac{2R}{2^i}$ for $i=1,2,\cdots$. Then $r_{i}-r_{i+1}=\frac{R}{2^{i}}$. By taking $h^\prime=r_{i}$, $h=r_{i+1}$ and $q=\chi^i$ in \eqref{eq10.28.5}, we get
\begin{align}\label{eq10.28.6}
\left( \int_{B(x_0,r_{i+1})} (u^{p\chi^{i+1}}+v^{p\chi^{i+1}}) \mathrm{d}x  \right)^\frac{1}{p\chi^{i+1}}
\leq  \left( \frac{C 2^{i}}{R} \right)^\frac{1}{\chi^i} \left( \int_{B(x_0,r_{i})} (u^{p\chi^i}+v^{p\chi^i}) \mathrm{d}x  \right)^\frac{1}{p\chi^i}.
\end{align}
By iteration, we infer that
\begin{align*}
\left( \int_{B(x_0,r_{i+1})} (u^{p\chi^{i+1}}+v^{p\chi^{i+1}}) \mathrm{d}x  \right)^\frac{1}{p\chi^{i+1}}
&\leq  2^\frac{i}{\chi^i} \left( \frac{C}{R} \right)^\frac{1}{\chi^i} \left( \int_{B(x_0,r_{i})} (u^{p\chi^i}+v^{p\chi^i}) \mathrm{d}x  \right)^\frac{1}{p\chi^i} \\
&\leq  2^{\sum_{k=1}^i \frac{k}{\chi^k}} \left( \frac{C}{R} \right)^{\sum_{k=1}^i\frac{1}{\chi^k}} \left( \int_{B(x_0,2R)} (u^{p^{\star}}+v^{p^{\star}}) \mathrm{d}x  \right)^\frac{1}{p^{\star}}. \nonumber
\end{align*}
Obviously, we have
$$\left( \int_{B(x_0,r_{i+1})} u^{p\chi^{i+1}} \mathrm{d}x  \right)^\frac{1}{p\chi^{i+1}} \leq\left( \int_{B(x_0,r_{i+1})} (u^{p\chi^{i+1}}+v^{p\chi^{i+1}}) \mathrm{d}x  \right)^\frac{1}{p\chi^{i+1}},$$
$$\left( \int_{B(x_0,2R)} (u^{p^{\star}}+v^{p^{\star}}) \mathrm{d}x  \right)^\frac{1}{p^{\star}} \leq \|u\|_{L^{p^{\star}}(B(x_0,2R))}+\|v\|_{L^{p^{\star}}(B(x_0,2R))}.$$
So we get
\begin{align}\label{eq10.28.7}
    &\quad \left( \int_{B(x_0,r_{i+1})} u^{p\chi^{i+1}} \mathrm{d}x  \right)^\frac{1}{p\chi^{i+1}} \\
    &\leq 2^{\sum_{k=1}^i \frac{k}{\chi^k}} \left( \frac{C}{R} \right)^{\sum_{k=1}^i\frac{1}{\chi^k}} \left( \|u\|_{L^{p^{\star}}(B(x_0,2R))}+\|v\|_{L^{p^{\star}}(B(x_0,2R))}\right)\nonumber
\end{align}
Noting that
$$ \sum_{k=1}^\infty \frac{k}{\chi^k} <\infty, $$
and
$$ \sum_{k=1}^\infty \frac{1}{\chi^k} =\frac{1}{\chi} \frac{1}{1-\frac{1}{\chi}} = \frac{N-p}{p}, $$
we can let $i$ sufficiently large (depending on $r$) in \eqref{eq10.28.7}, and obtain that, for any $p^{\ast}\leq r<+\infty$,
\begin{align}\label{eq10.28.9}
\|u\|_{L^r(B(x_0,R))}
&\leq \frac{C}{R^\frac{N-p}{p}} \left(\|u\|_{L^{p^{\star}}(B(x_0,2R))}+\|v\|_{L^{p^{\star}}(B(x_0,2R))}\right)\\
&\leq \frac{C}{R^\frac{N-p}{p}} \left(\|u\|_{L^{p^{\star}}(\R^N)}+\|v\|_{L^{p^{\star}}(\R^N)}\right),\nonumber
\end{align}
where $R\rightarrow0$ as $r\rightarrow+\infty$, and the constant $C>0$ is independent of $r$ and $R$. So we obtain from finite coving theorem that
\begin{align}\label{eq10.28.10}
\|u\|_{L^r(\Omega)}
\leq C (|\Omega|) \,\left(\|u\|_{L^{p^{\star}}(\R^N)}+\|v\|_{L^{p^{\star}}(\R^N)}\right), \,\,\,\,\,\,\, \forall\,\, \Omega \subset\subset \R^N.
\end{align}
Thus $u\in L^{r}_{loc}(\mathbb{R}^{N})$ for any $0<r<+\infty$. Arguing in the same way for $v$, we can also obtain $v\in L^{r}_{loc}(\mathbb{R}^{N})$ for any $0<r<+\infty$.

By using the local boundedness and regularity estimates (c.f. e.g. Chapter 4 in \cite{HL}, see also \cite{ED83,KM,PT,XCL}), we can deduce that $u,v\in C(\R^N)\cap L_{loc}^\infty(\R^N)$. Furthermore, by using the standard $C_{loc}^{1,\alpha}$-regularity estimates (c.f. e.g. Chapter 4 in \cite{HL}, see also \cite{ED83,KM,PT,XCL}), we can deduce that $u,v\in C_{loc}^{1,\alpha}(\R^N)$ for some $0<\alpha<\min\{1,\frac{1}{p-1}\}$. This completes our proof of Lemma \ref{lm.8}.
\end{proof}

\subsection{A better preliminary estimate on upper bound of asymptotic behavior}\label{sc3.2}

The following result provides a better preliminary decay estimate for solution of the generalized equation \eqref{eq10.25.1}, which is not sharp, but this estimate plays a key role in the proof of Theorem \ref{th2.1-}.

\begin{lem}\label{lm.5}
Let $(u, v)$ be a pair of positive solutions to the equation \eqref{eq10.25.1} with $p>1$. Then, there is a $\sigma>0$ and a constant $C>0$ such that
$$ u(x) , v(x) \leq\frac{C}{1+|x|^{\frac{N-p}{p}+\sigma}} \quad \,\,\, \mbox{in}\,\, \R^N.$$
\end{lem}
\begin{proof}
%The result is known. For completeness and to apply it in the following, we give a proof of this Lemma.
Let $R_0>0$ large (to be determined later) and $q>1$. For any $R>r>R_0$, take $\xi\in C^2(\R^N)$, with $\xi=0$ in $B_r(0)$, $\xi=1$ in $\mathbb{R}^{N}\setminus B_R(0)$, $0 \leq \xi\leq 1$, and $|\nabla \xi| \leq \frac{2}{(R-r)}$. Let $\phi=\xi^p u^{1+p(q-1)}, \psi=\xi^p v^{1+p(q-1)}$. The equations \eqref{eq10.25.1} imply that
$$\int_{\R^N} |\nabla u|^{p-2} \nabla u \nabla\phi \mathrm{d}x=\int_{\R^N} u^{\alpha}v^{\beta} \phi \mathrm{d}x,$$
$$\int_{\R^N} |\nabla v|^{p-2} \nabla v \nabla\psi \mathrm{d}x=\int_{\R^N} u^{\beta}v^{\alpha} \psi \mathrm{d}x.$$
Considering the first equation by H\"{o}lder's inequality, one has
\begin{align}\label{eq10.26.1}
&\quad \int_{\R^N} u^{\alpha}v^{\beta} \phi \mathrm{d}x \\
           &\leq \int_{\R^N \cap \{u\leq v\}} v^{p^{\star}+p(q-1)} \xi^p\mathrm{d}x+\int_{\R^N \cap \{u> v\}} u^{p^{\star}+p(q-1)} \xi^p\mathrm{d}x \nonumber \\
           &\leq \int_{\R^N \setminus B_{R_0}} v^{p^{\star}-p}v^{pq}\xi^p \mathrm{d}x+\int_{\R^N \setminus B_{R_0}}u^{p^{\star}-p}u^{pq} \xi^p \mathrm{d}x \nonumber \\
           &\leq \|u\|_{L^{p^{\star}}(\R^N \setminus B_{R_0})}^{p^{\star}-p}\left(\int_{\R^N \setminus B_{R_0}} |\xi u^q|^{p^{\star}} \mathrm{d}x \right)^\frac{p}{p^{\star}}+\|v\|_{L^{p^{\star}}(\R^N \setminus B_{R_0})}^{p^{\star}-p}\left(\int_{\R^N \setminus B_{R_0}} |\xi v^q|^{p^{\star}} \mathrm{d}x \right)^\frac{p}{p^{\star}} \nonumber
\end{align}
On the other hand, noticing that
\begin{align*}
    |\nabla(\xi u^q)|^p &=|\nabla \xi u^q+ qu^{q-1}\nabla u\xi|^p\\
    &\leq C|\nabla \xi|^pu^{pq}+C q^p u^{p(q-1)}|\nabla u|^p \xi^p.
\end{align*}
Then we have
\begin{align}\label{eq10.26.2}
&\quad \int_{\R^N}|\nabla u|^{p-2} \nabla u \nabla\phi \mathrm{d}x\\
           &= (1+p(q-1))\int_{\R^N \setminus B_{R_0}}u^{p(q-1)} |\nabla u|^p \xi^p \mathrm{d}x + p \int_{\R^N \setminus B_{R_0}} \xi^{p-1} u^{1+p(q-1)} |\nabla u|^{p-2} \nabla u \nabla\xi \mathrm{d}x \nonumber\\
           &\geq (1+p(q-1))\int_{\R^N \setminus B_{R_0}}u^{p(q-1)} |\nabla u|^p \xi^p \mathrm{d}x - p \int_{\R^N \setminus B_{R_0}} \xi^{p-1} u^{(p-1)(q-1)+q} |\nabla u|^{p-1} |\nabla\xi| \mathrm{d}x \nonumber\\
           &\geq (1+p(q-1))\int_{\R^N \setminus B_{R_0}}u^{p(q-1)} |\nabla u|^p \xi^p \mathrm{d}x \nonumber\\
           &\quad- p \left(\varepsilon\int_{\R^N \setminus B_{R_0}} u^{p(q-1)} |\nabla u|^p \xi^p \mathrm{d}x+C_\varepsilon\int_{\R^N} |\nabla \xi|^p u^{pq} \mathrm{d}x \right) \nonumber\\
           &\geq \frac{C_1}{q^{p-1}} \int_{\R^N} |\nabla(\xi u^q)|^p \mathrm{d}x - C\int_{\R^N} |\nabla \xi|^p u^{pq} \mathrm{d}x \nonumber\\
           &\geq \frac{C_2}{q^{p-1}} \left(\int_{\R^N \setminus B_{R_0}} |\xi u^q|^{p^{\star}} \mathrm{d}x \right)^\frac{p}{p^{\star}} - C\int_{\R^N \setminus B_{R_0}} |\nabla \xi|^p u^{pq} \mathrm{d}x. \nonumber
\end{align}
It follows from \eqref{eq10.26.1} with \eqref{eq10.26.2} that
\begin{align}\label{eq10.26.3}
&\quad\frac{C_2}{q^{p-1}} \left(\int_{\R^N \setminus B_{R_0}} |\xi u^q|^{p^{\star}} \mathrm{d}x \right)^\frac{p}{p^{\star}} - C\int_{\R^N \setminus B_{R_0}} |\nabla \xi|^p u^{pq} \mathrm{d}x \\
&\leq \|u\|_{L^{p^{\star}}(\R^N \setminus B_{R_0})}^{p^{\star}-p}\left(\int_{\R^N \setminus B_{R_0}} |\xi u^q|^{p^{\star}} \mathrm{d}x \right)^\frac{p}{p^{\star}}+\|v\|_{L^{p^{\star}}(\R^N \setminus B_{R_0})}^{p^{\star}-p}\left(\int_{\R^N \setminus B_{R_0}} |\xi v^q|^{p^{\star}} \mathrm{d}x \right)^\frac{p}{p^{\star}} . \nonumber
\end{align}
Similarly, we get
\begin{align}\label{eq10.26.3-}
&\quad\frac{C'_2}{q^{p-1}} \left(\int_{\R^N \setminus B_{R_0}} |\xi v^q|^{p^{\star}} \mathrm{d}x \right)^\frac{p}{p^{\star}} - C'\int_{\R^N \setminus B_{R_0}} |\nabla \xi|^p v^{pq} \mathrm{d}x \\
&\leq \|v\|_{L^{p^{\star}}(\R^N \setminus B_{R_0})}^{p^{\star}-p}\left(\int_{\R^N \setminus B_{R_0}} |\xi v^q|^{p^{\star}} \mathrm{d}x \right)^\frac{p}{p^{\star}}+\|u\|_{L^{p^{\star}}(\R^N \setminus B_{R_0})}^{p^{\star}-p}\left(\int_{\R^N \setminus B_{R_0}} |\xi u^q|^{p^{\star}} \mathrm{d}x \right)^\frac{p}{p^{\star}} . \nonumber
\end{align}
Take $R_0$ sufficiently large such that
$$\|u\|_{L^{p^{\star}}(\R^N \setminus B_{R_0})}^{p^{\star}-p}\leq \frac{C_2}{4q^{p-1}}, \,\,\,\|v\|_{L^{p^{\star}}(\R^N \setminus B_{R_0})}^{p^{\star}-p} \leq \frac{C'_2}{4q^{p-1}}.$$
Rescaling the constants and combining \eqref{eq10.26.3} with \eqref{eq10.26.3-}, we can obtain that
\begin{align}\label{eq10.26.4}
&\quad \left(\int_{\R^N \setminus B_{R_0}} |\xi u^q|^{p^{\star}} \mathrm{d}x \right)^\frac{p}{p^{\star}} +\left(\int_{\R^N \setminus B_{R_0}} |\xi v^q|^{p^{\star}} \mathrm{d}x \right)^\frac{p}{p^{\star}}\\
&\leq C q^{p-1} \int_{\R^N \setminus B_{R_0}} |\nabla \xi|^p (u^{pq}+ v^{pq})\mathrm{d}x,\nonumber
\end{align}
which implies
\begin{align}\label{eq10.26.8}
&\quad \left(\int_{\R^N \setminus B_{R}} \left(u^{qp^{\star}}+v^{qp^{\star}}\right) \mathrm{d}x \right)^\frac{1}{q p^{\star}}\\
&\leq C q^\frac{p-1}{pq} \left(\frac{2}{R-r}\right)^\frac{1}{q} \left(\int_{\R^N \setminus B_{r}} (u^{pq}+v^{pq}) \mathrm{d}x\right)^\frac{1}{pq},\nonumber
\end{align}
for any $R>r>R_0$.

Let $\chi=\frac{p^{\star}}{p}=\frac{N}{N-p}>1$. For any $R>R_0$, define $r_i = 2R - \frac{2R}{2^i}$ for $i=1,2,\cdots$. Then $r_{i+1}-r_i= \frac{2R}{2^i} -\frac{2R}{2^{i+1}}=\frac{R}{2^i}$. Taking $R=r_{i+1}$, $r=r_i$ and $q=\chi^i$ in \eqref{eq10.26.8}, we derive
\begin{align}\label{eq10.26.5}
&\quad \left(\int_{\R^N \setminus B_{r_{i+1}}} (u^{\chi^{i+1}p}+v^{\chi^{i+1}p}) \mathrm{d}x \right)^\frac{1}{\chi^{i+1} p}\\
& \leq C (\chi^i)^\frac{p-1}{p\chi^i} \left(\frac{2}{r_{i+1}-r_i}\right)^\frac{1}{\chi^i} \left(\int_{\R^N \setminus B_{r_i}} (u^{p \chi^i}+v^{p \chi^i}) \mathrm{d}x\right)^\frac{1}{p\chi^i}\nonumber \\
& \leq C \chi^{ \frac{p-1}{p} \frac{i}{\chi^i} } \left(\frac{2^i}{R}\right)^\frac{1}{\chi^i} \left(\int_{\R^N \setminus B_{r_i}} (u^{p \chi^i}+v^{p \chi^i}) \mathrm{d}x\right)^\frac{1}{p\chi^i} \nonumber
\end{align}
for any $R>R_0$. By iteration, we can obtain from \eqref{eq10.26.5} that
\begin{align}\label{eq10.26.6}
&\quad \left(\int_{\R^N \setminus B_{r_{i+1}}} |u|^{\chi^{i+1} p} \mathrm{d}x \right)^\frac{1}{\chi^{i+1} p}+\left(\int_{\R^N \setminus B_{r_{i+1}}} |v|^{\chi^{i+1} p} \mathrm{d}x \right)^\frac{1}{\chi^{i+1} p} \\
& \leq 2\left(\int_{\R^N \setminus B_{r_{i+1}}} (|u|^{\chi^{i+1}p}+|v|^{\chi^{i+1}p}) \mathrm{d}x \right)^\frac{1}{\chi^{i+1} p} \nonumber \\
& \leq C \prod^i_{k=1} \left[ \chi^{ \frac{p-1}{p} \frac{k}{\chi^k} } \left(\frac{2^k}{R}\right)^\frac{1}{\chi^k} \right] \left(\int_{\R^N \setminus B_{R}} (u^{p \chi}+v^{p \chi}) \mathrm{d}x\right)^\frac{1}{p\chi} \nonumber\\
& \leq C   \left( 2 \chi^\frac{p-1}{p}\right)^{\sum\limits^i_{k=1}\frac{k}{\chi^k} } \left(\frac{1}{R}\right)^{\sum\limits^i_{k=1}\frac{1}{\chi^k}}
\left(\int_{\R^N \setminus B_{R}} (u^{p^{\star}}+v^{p^{\star}}) \mathrm{d}x\right)^\frac{1}{p^{\star}} \nonumber\\
& \leq C   \left( 2 \chi^\frac{p-1}{p}\right)^{\sum\limits^i_{k=1}\frac{k}{\chi^k} } \left(\frac{1}{R}\right)^{\sum\limits^i_{k=1}\frac{1}{\chi^k}}(\left\| u \|_{L^{p^{\star}}(\R^N \setminus B_{R})}+\| v\|_{L^{p^{\star}}(\R^N \setminus B_{R})}\right)\nonumber
\end{align}
for any $R>R_0$. Since
$$\sum_{i=1}^\infty \frac{k}{\chi^k} < + \infty \qquad \text{and} \qquad \sum_{k=1}^i \frac{1}{\chi^k} < \sum_{k=1}^\infty \frac{1}{\chi^k}=\frac{1}{\chi}\frac{1}{1-\chi}=\frac{N-p}{p},$$
we have proved that for any $\overline{p}>p^{\star}$, there is a $R_0>0$ large, depending on $\overline{p}$, such that
\begin{align}\label{eq10.26.7}
&\quad \| u \|_{L^{\overline{p}}(\R^N \setminus B_{2R})}+\| v \|_{L^{\overline{p}}(\R^N \setminus B_{2R})} \\ &\leq \frac{C}{R^{ \frac{N-p}{p}-o_{\overline{p}}(1)} }
(\| u \|_{L^{p^{\star}}(\R^N \setminus B_{R})}+\| v\|_{L^{p^{\star}}(\R^N \setminus B_{R})}),\,\quad\,\,\, \,\forall \,\, R>R_0, \nonumber
\end{align}
where $o_{\overline{p}}(1)\to 0$ as $\overline{p} \to \infty$.

\smallskip

Next we estimate $\| u \|_{L^{p^{\star}}(\R^N \setminus B_{R})}$ and $\| v\|_{L^{p^{\star}}(\R^N \setminus B_{R})}$.

Take $\eta=\xi^p u$, where $\xi$ satisfies $\xi=0$ in $B_R$, $\xi=1$ in $\R^N \setminus B_{2R}$, $0\leq \xi\leq 1$, and $|\nabla \xi| \leq 2R^{-1}$. Then, similar to \eqref{eq10.26.4}, we can derive that
\begin{align}\label{eq10.26.9}
&\quad \left(\int_{\R^N } |\xi u|^{p^{\star}} \mathrm{d}x \right)^\frac{1}{p^{\star}}+\left(\int_{\R^N } |\xi v|^{p^{\star}} \mathrm{d}x \right)^\frac{1}{p^{\star}} \\
&\leq C  \left( \int_{\R^N} |\nabla \xi|^p u^{p} \mathrm{d}x \right)^\frac{1}{p}+ C \left( \int |\nabla \xi|^p  v^p dx \right)^{1/p} \nonumber \\
&= C \left( \int_{B_{2R}\setminus B_R} |\nabla \xi|^p u^{p} \mathrm{d}x \right)^\frac{1}{p}+C \left( \int_{B_{2R}\setminus B_R} |\nabla \xi|^p v^{p} \mathrm{d}x \right)^\frac{1}{p} \nonumber \\
& \leq C \left( \frac{1}{R^p} \int_{B_{2R}\setminus B_R} u^{p} \mathrm{d}x \right)^\frac{1}{p}+C \left( \frac{1}{R^p} \int_{B_{2R}\setminus B_R} v^{p} \mathrm{d}x \right)^\frac{1}{p} \nonumber\\
& \leq C \left[\frac{1}{R^p}\left(\int_{B_{2R}\setminus B_R} 1 \mathrm{d}x\right)^\frac{p}{N}\left(\int_{B_{2R}\setminus B_R} u^{p\frac{N}{N-p}} \mathrm{d}x\right)^\frac{N-p}{N} \right]^\frac{1}{p} \nonumber\\
& \quad+\left[\frac{1}{R^p}\left(\int_{B_{2R}\setminus B_R} 1 \mathrm{d}x\right)^\frac{p}{N}\left(\int_{B_{2R}\setminus B_R} v^{p\frac{N}{N-p}} \mathrm{d}x\right)^\frac{N-p}{N} \right]^\frac{1}{p} \nonumber\\
&=C \left(\int_{B_{2R}\setminus B_R} u^{p^{\star}} \mathrm{d}x \right)^\frac{1}{p^{\star}}+C \left(\int_{B_{2R}\setminus B_R} v^{p^{\star}} \mathrm{d}x \right)^\frac{1}{p^{\star}}. \nonumber
\end{align}
Thus we have
\begin{align}\label{eq10.26.10}
&\quad \int_{\R^N \setminus B_{2R}} |u|^{p^{\star}} + |v|^{p^{\star}} \mathrm{d}x \\
&\leq \int_{\R^N } |\xi u|^{p^{\star}}+|\xi v|^{p^{\star}} \mathrm{d}x \nonumber \\
&\leq C \int_{B_{2R}\setminus B_R} |u|^{p^{\star}}+|v|^{p^{\star}} \mathrm{d}x \nonumber \\
&= C \int_{\R^N \setminus B_R} |u|^{p^{\star}}+|v|^{p^{\star}} \mathrm{d}x - C \int_{\R^N \setminus B_{2R}} |u|^{p^{\star}}+|v|^{p^{\star}} \mathrm{d}x, \nonumber
\end{align}
which implies
\begin{align}\label{eq10.26.11}
\int_{\R^N \setminus B_{2R}} |u|^{p^{\star}}+|v|^{p^{\star}} \mathrm{d}x
&\leq  \frac{C}{C+1} \int_{\R^N \setminus B_R} |u|^{p^{\star}}+|v|^{p^{\star}} \mathrm{d}x.
\end{align}

Set $\tau:= \frac{C}{C+1}<1$, and let $\Psi(R)=\| u \|^{p^{\star}}_{L^{p^{\star}}(\R^N \setminus B_R)}+\| v \|^{p^{\star}}_{L^{p^{\star}}(\R^N \setminus B_R)}$. So by \eqref{eq10.26.11},
\begin{align}\label{eq10.26.12}
\Psi(2R)\leq \tau \Psi(R),\,\,\quad \forall \,\, R\geq R_0,
\end{align}
from which we get
\begin{align}\label{eq10.26.13}
\Psi(2^k R_0)\leq \tau^k \Psi(R_0).
\end{align}
For any $|x|\geq R_0$, there is a $k\in \N^+$, such that
$$ 2^k R_0 \leq |x| \leq 2^{k+1} R_0,$$
which combined with the definition of $\Psi(R)$ and \eqref{eq10.26.13} yield that
$$ \Psi(|x|) \leq \Psi(2^k R_0)\leq \tau^k \Psi(R_0) \leq \tau^{\log_2 |x|-\log_2 (2R_{0})} \Psi(R_0).$$
Since
$$ \tau^{\log_2 |x|}= 2^{\log_2\tau \cdot \log_2 |x|} = |x|^{\log_2\tau},$$
we obtain
$$ \Psi(|x|) \leq \frac{C}{|x|^{\log_2 \frac{1}{\tau}}},$$
where $\log_2 \frac{1}{\tau} >0$.
Thus, we have shown that there exists a $\sigma>0$ independent of $\overline{p}>p^{\star}$ such that
$$ \Psi(|x|) \leq \frac{C}{|x|^\sigma}.$$
By rescaling the constants and fixing $\overline{p}>p^{\star}$ sufficiently large such that $o_{\overline{p}}(1)<\frac{\sigma}{2p^{\star}}$, we obtain from \eqref{eq10.26.7} that
\begin{align}\label{eq10.26.17}
\| u \|_{L^{\overline{p}}(\R^N \setminus B_{2R})}+\| v \|_{L^{\overline{p}}(\R^N \setminus B_{2R})} \leq \frac{C}{R^{ \frac{N-p}{p}+\frac{\sigma}{2p^{\star}}}},\,\quad\,\forall \,\,R>R_0.
\end{align}
Consequently, there exists $q>\frac{N}{p}$ and $\widetilde{R}>0$ large enough, such that
\begin{equation*}
  \|u\|^{p^{\star}-p}_{L^{q(p^{\star}-p)}(B_{1}(x))}+\|v\|^{p^{\star}-p}_{L^{q(p^{\star}-p)}(B_{1}(x))}\leq \overline{C}, \qquad \forall \,\, |x|>2\widetilde{R},
\end{equation*}
where the constant $\overline{C}>0$ is independent of $x$ and $\widetilde{R}$. Therefore, for any $x$ such that $|x|=4R$ with $R>\max\{R_0,\widetilde{R}\}>1$, by the local boundedness estimates (c.f. e.g. Chapter 4 in \cite{HL}, see also \cite{ED83,KM,PT,XCL}) and the decay property in \eqref{eq10.26.17}, we have, for any fixed $\bar{p}\geq\hat{p}$ large enough,
$$ |u(x)|\leq \max_{y\in B_{\frac{1}{2}}(x)} |u(y)| \leq \widetilde{C}\|u\|_{L^{\overline{p}}(B_1(x))}\leq \widetilde{C}\|u\|_{L^{\overline{p}}(\R^N \setminus B_{2R}(0))} \leq \frac{\widetilde{C}C}{R^{ \frac{N-p}{p}+\frac{\sigma}{2p^{\star}}}},$$
$$ |v(x)|\leq \max_{y\in B_{\frac{1}{2}}(x)} |v(y)| \leq \widetilde{C}\|v\|_{L^{\overline{p}}(B_1(x))}\leq \widetilde{C}\|v\|_{L^{\overline{p}}(\R^N \setminus B_{2R}(0))} \leq \frac{\widetilde{C}C}{R^{ \frac{N-p}{p}+\frac{\sigma}{2p^{\star}}}},$$
where the constants $C$ and $\widetilde{C}=\widetilde{C}(N,p,\overline{C})$ are independent of $x$ and $R$. Combining this with Lemma \ref{lm.8} finish our proof of Lemma \ref{lm.5}.
\end{proof}

\subsection{Proof of Theorem \ref{th2.1-}}\label{sc3.3}

We have proved in Lemmas \ref{lm.8} and \ref{lm.5} that $u, v\in L^\infty(\R^N)\cap C^{1,\alpha}(\R^N)$ and $(u,v)$ satisfy a better preliminary asymptotic upper bound estimate.

\smallskip

Now, we are ready to prove the sharp asymptotic estimates on any positive weak solutions $(u,v)$ to the system \eqref{eq1.1} and $|\nabla u|, |\nabla v|$ in Theorem \ref{th2.1-}. The proof of Theorem \ref{th2.1-} will be divided into the following three steps.

\medskip

\noindent{\bf Step 1.} We will show that there exists some constants $c_0,C_0>0$ such that
\begin{align}\label{eq10.26.57}
\frac{c_0}{1+|x|^\frac{N-p}{p-1}} \leq u(x), v(x) \leq \frac{C_0}{1+|x|^\frac{N-p}{p-1}} \,\,\,\,\,\,\,\,\mbox{in}\,\,\R^N.
\end{align}

We will show \eqref{eq10.26.57} by applying the following two key Lemmas \ref{estimate1}-\ref{estimate2}.
 \begin{lem}[Theorem 1.5 in \cite{XCL}]\label{estimate1}
Let $\Omega$ be an exterior domain in $\R^N$ such that $\Omega^c=\R^N\setminus \Omega$ is bounded and
function $V(x) \in L^\frac{N}{p}(\Omega)$
satisfying $V(x) \leq K|x|^{-\gamma}$ in $\Omega$ for some positive constants $K$ and $\gamma>p$. If $u\in D^{1,p}(\Omega)$ is a weak subsolution to
$$-\Delta_p u =V(x) u^{p-1}  \,\,\,\,\mbox{in}\, \ \Omega$$
with $1<p<N$, then there exists a positive constant $C=C(N,p,\gamma,K)>0$ such that
$$ u(x)\leq C M |x|^\frac{N-p}{p-1}\,\,\,\,\mbox{in}\,\,\, \R^N\setminus B_{R_1}(0),$$
where $M = \sup\limits_{\partial B_{R_1}(0)} u^+ $ and $R_1 > 1$ is a constant depending on $N, p, K, \gamma$.
\end{lem}

\begin{lem}[Theorem 2.3 in \cite{SJZH02}]\label{estimate2}
Let $u$ be a weak solution of
$$-\Delta_p u \geq 0 \qquad \mbox{in}\,\, \R^N\setminus \Omega$$
with $\Omega$ compact and $1<p<N$. Then there exist positive constants $m = \inf\limits_{\partial (\R^N\setminus\Omega)} u$ and $C = \inf\limits_{\partial (\R^N\setminus\Omega)} |x|^\frac{N-p}{p-1}$ such that
$$ u(x)\geq C m |x|^\frac{N-p}{p-1} \qquad \mbox{in}\,\,\,\, \R^N\setminus\Omega.$$
%
%Let $\Omega$ be an exterior domain in $\R^N$ such that $\Omega^c=\R^N\setminus \Omega$ is bounded. If
%$u$
%is a positive weak solution of
%$$-\Delta_p u \geq 0 \,\,\,\,\mbox{in}\, \Omega$$
%with $1<p<N$,
%then
%$$ u(x)\geq C M |x|^\frac{N-p}{p-1}\,\,\,\,\mbox{in}\,\,\, \Omega,$$
%where $m = \inf_{\partial \Omega} u$ and $C = \inf_{\partial \Omega} |x|^\frac{N-p}{p-1}$.
\end{lem}

Obviously, we can deduce the sharp lower bound in \eqref{eq10.26.57} for positive weak solutions $u, v$ of \eqref{eq1.1} from Lemma \ref{estimate2}.

Defining that
$$f(u,v):=u^{\alpha-p+1}v^{\beta}, \ \ g(u,v):=u^{\beta}v^{\alpha-p+1}.$$
We will show that $f,g$ satisfy the two conditions in Lemma \ref{estimate1}, i.e. $f(u,v), g(u,v)\in L^\frac{N}{p}(\R^N)$ and
\begin{equation}\label{a6}
  f(u,v), g(u,v) \leq K|x|^{-\gamma}
\end{equation}
for some $K>0$, $\gamma>p$ and $|x|$ large. For the sake of simplicity, we only carry out the proof for $f$, $g$ can be handled similarly. One can carry out the proof into two cases.

\emph{Case (i)}. $\alpha-p+1\geq0.$ It is obviously that
$$f(u,v)=u^{\alpha-p+1}v^{\beta}\leq u^{p^\star-p}+v^{p^\star-p}.$$
Since the positive weak solutions $u,v$ of \eqref{eq1.1} belong to $D^{1,p}(\R^N),$ it follows that $u^{p^\star-p}, v^{p^\star-p}\in L^\frac{N}{p}(\R^N),$
which implies $f(u,v)\in L^\frac{N}{p}(\R^N)$. Recalling that $\frac{N-p}{p}(p^\star-p)=p$ and the better preliminary estimate in Lemma \ref{lm.5},
one has $f(u,v)\leq K|x|^{-\gamma}$ for some $K>0$, $\gamma>p$ and $|x|$ large.

\emph{Case (ii)}. $\alpha-p+1<0.$ We will show that for any small $\theta>0$,
\begin{equation}\label{uuvv}
u(x),v(x)\leq C|x|^{-\frac{N-p}{p-1}+\theta} \,\,\,\mbox{for}\,\,\,|x|>1,
\end{equation}
where some positive constant $C$ depends on $\theta$.
According to direct calculations, for any $\hat{\delta} \in (0, \frac{N-p}{p-1}),$ there exists a constant $c>0$ such that
$$-\Delta_p \frac{1}{|x|^{\hat{\delta}}}=\frac{c}{|x|^{\hat{\delta}+2+(p-2)(\hat{\delta}+1)}}.$$
On the other hand, it is easy to check that if $\delta>\frac{N-p}{p},$ then
$$\delta+2+(p-2)(\delta+1)<(p^\star-1)\delta.$$
So, we can find a small $\overline{\theta}>0$, such that for $|x|>1,$
$$-\Delta_p \left[\frac{1}{|x|^{\frac{N-p}{p}+\sigma+\overline{\theta}}}\right]=\frac{c}{|x|^{\frac{N-p}{p}+\sigma+\overline{\theta}+2+(p-2)(\frac{N-p}{p}+\sigma+\overline{\theta}+1)}}\geq\frac{c}{|x|^{(p^\star-1)(\frac{N-p}{p}+\sigma)}}.$$
It follows from Lemma \ref{lm.5},
$$-\Delta_p u=u^\alpha v^\beta \leq \frac{C}{|x|^{(p^\star-1)(\frac{N-p}{p}+\sigma)}}, \qquad |x|\geq R_0.$$
By the comparison principle, we have
$$u(x)\leq \frac{C}{|x|^{\frac{N-p}{p}+\sigma+\overline{\theta}}}.$$
We continue this procedure up to $i,$ with $\frac{N-p}{p}+\sigma+i\overline{\theta}<\frac{N-p}{p-1}.$ So, for any $\theta >0$ small,
\eqref{uuvv} follows.
By Lemma \ref{estimate2} and \eqref{uuvv}, for $\theta >0$ sufficiently small and $\alpha-p+1<0$, we have
\begin{align*}
    f(u,v)=u^{\alpha-p+1}v^\beta &\leq C\frac{1}{|x|^{\frac{N-p}{p-1}(\alpha-p+1)}} \cdot \frac{1}{|x|^{(\frac{N-p}{p-1}-\theta)\beta}}\\
    &=\frac{C}{|x|^{\frac{N-p}{p-1}(p^\star-p)-\theta \beta}}\\
    &=\frac{C}{|x|^{\frac{N-p}{p}(p^\star-p)+\frac{N-p}{p(p-1)}(p^\star-p)-\theta \beta}}\\
    &\leq \frac{C}{|x|^{\frac{N-p}{p}(p^\star-p)+\sigma'}}=\frac{C}{|x|^{p+\sigma'}},
\end{align*}
where $\sigma'>0$ for suitable $\theta$. Recalling $N\geq 2,$ it is easy to know that $f(u,v)\in L^\frac{N}{p}(\R^N), f(u,v)\leq K|x|^{-\gamma}$ for some $K>0$, $\gamma>p$ and $|x|$ large.

Consequently, we deduce the desired sharp upper bound in \eqref{eq10.26.57} from Lemma \ref{estimate1} and the sharp lower bound in \eqref{eq10.26.57} from Lemma \ref{estimate2}, respectively. This proves the sharp asymptotic estimates \eqref{eq0806} for $u,v$ in Theorem \ref{th2.1-}.
%$$\frac{c_0}{1+|x|^\frac{N-p}{p-1}} \leq u(x) \leq \frac{C_0}{1+|x|^\frac{N-p}{p-1}},\,\,\,\,\,\,\mbox{in}\,\R^N.$$

\medskip

Next, in Steps 2 and 3, we will prove the sharp asymptotic estimates \eqref{eq0806+} for $|\nabla u|, |\nabla v|$ in Theorem \ref{th2.1-} by using the arguments in \cite{BS16,VJ16}.

\smallskip

\noindent{\bf Step 2.}
We are to show that, for some constant $C'_0>0$,
$$ |\nabla u(x)|, |\nabla v(x)| \leq \frac{C'_0}{|x|^\frac{N-1}{p-1}}\,\,\,\,\,\,\,\,\mbox{for}\,\,|x|>R_0\geq 1.$$

For any $R>0$ and $y\in \R^N$, we define
$$ u_R (y): =R^\frac{N-p}{p-1}u(Ry).$$
Firstly, considering the first equation of system \eqref{eq1.1}, we deduce that
\begin{align}\label{eq2901}
-\Delta_p u_R(y)
=-R^N \Delta_p u(x)=R^p f(u(Ry),v(Ry))u^{p-1}_R(y) \,\,\,\,\,\,\,\mbox{in}\,\,\R^N.
\end{align}
It follows from $f(u(Ry),v(Ry)) \leq K|x|^{-\gamma}$ with $\gamma>p$ that
\begin{align}\label{eq2902}
-\Delta_p u_R(y)
&=R^p f(u(Ry),v(Ry)) u^{p-1}_R(y) \leq \frac{K}{R^{\gamma-p} |y|^{\gamma}} u^{p-1}_R(y) \, \,\,\,\,\,\,\,\mbox{for}\,\,R>R_0\,\,\,\mbox{and}\,\,|y|>1,
\end{align}
which implies that
\begin{align}\label{eq2903}
-\Delta_p u_R(y)
\leq Ku^{p-1}_R(y) \,\,\,\,\,\,\,\,\,\,\mbox{in}\,\,\R^N\setminus B_1.
\end{align}

On the other hand, according to \eqref{eq10.26.57}, we have, for any $R\geq R_0$,
\begin{align}\label{eq2905}
u_R (y) =R^\frac{N-p}{p-1} u(Ry)\leq C_0 R^\frac{N-p}{p-1} |Ry|^{-\frac{N-p}{p-1}}= C_0 |y|^{-\frac{N-p}{p-1}} \leq C_0 \,\,\,\,\,\,\mbox{in}\,\,\R^N\setminus B_1.
\end{align}
By \eqref{eq2903}--\eqref{eq2905} and the regularity estimates in DiBenedetto \cite{ED83} and Tolksdorf \cite{PT}, we have
\begin{align}\label{eq2906}
\|\nabla u_R\|_{L^\infty (B(0,4) \setminus B(0,2))}\leq C_1
\end{align}
for some constant $C_1>0$ depending on $C_{0}$. Finally, for any $x\in\R^N\setminus B(0,3R)$, by applying \eqref{eq2906} with $R=\frac{|x|}{3}$ and $y\in B(0,4) \setminus B(0,2)$, we obtain
\begin{align}\label{eq2907}
|\nabla_x u(x)|= R^\frac{1-N}{p-1} |\nabla_y u_R(y) | \leq C_1 |x|^\frac{1-N}{p-1}\,\,\,\,\,\,\,\mbox{for}\,\,|x|\geq 3 R_0.
\end{align}
Arguing in the same way, we can derive the same conclusion for $v$.

\smallskip

\noindent{\bf Step 3.}
We will show that, for some constant $c'_0$,
$$ |\nabla u(x)|, |\nabla v(x)| \geq \frac{c'_0}{|x|^\frac{N-1}{p-1}} \quad\,\,\mbox{for}\,\,|x|>R_0.$$

The proof will be carried out by contradiction arguments. We assume on the contrary that there exist sequences of radii $R_n$ and points $x_n\in\R^N$ with $R_n\to\infty$ as $n\to\infty$ and $|x_n|=R_n$, such that
\begin{align}\label{eq2908}
|\nabla u(x_n)|\leq \frac{\theta_n}{R_n^\frac{N-1}{p-1}}
\end{align}
with $\theta_n\to 0$ as $n\to\infty$. For $0<a<A$ fixed, we set
$$ u_{R_n} (x): =R_n^\frac{N-p}{p-1}   u(R_n  x), \qquad \forall \,\, x\in\overline{B_A \setminus B_a}.$$
It follows from \eqref{eq10.26.57} that for any $n$ large enough such that $|R_n x|\geq R_{n}a>R_0$, we have
\begin{align}\label{eq2910}
\frac{c_0}{|x|^\frac{N-p}{p-1}}   \leq    u_{R_n} (x)=R_n^\frac{N-p}{p-1} u(R_n  x) \leq \frac{C_0}{|x|^\frac{N-p}{p-1}},%\,\,\,\,\,\mbox{in}\,\,\overline{B_A \setminus B_a},
\end{align}
which implies that
\begin{align}\label{eq2911}
\frac{c_0}{A^\frac{N-p}{p-1}}   \leq    u_{R_n} (x) \leq \frac{C_0}{a^\frac{N-p}{p-1}}\,\,\,\,\,\,\,\,\mbox{in}\,\,\overline{B_A \setminus B_a},
\end{align}
and in particular,
\begin{align}\label{eq2912}
u_{R_n} (x) \leq \frac{C_0}{A^\frac{N-p}{p-1}}\,\,\,\,\,&\mbox{on}\,\,\partial B_A, \\
u_{R_n} (x) \geq \frac{c_0}{a^\frac{N-p}{p-1}}\,\,\,\,\,&\mbox{on}\,\,\partial B_a. \nonumber
\end{align}

On the other hand, similar to \eqref{eq2902}, we obtain
\begin{align}\label{eq2913}
0\leq-\Delta_p u_{R_n} (x) \leq \frac{K}{a^{\gamma} {R_n}^{\gamma-p}} u^{p-1}_{R_{n}}(x) \leq K u^{p-1}_{R_n}(x) \,\,\,\,\,\,\,\mbox{in}\,\,\overline{B_A \setminus B_a}
\end{align}
for any $n$ large enough such that $|R_n x|\geq R_{n}a>R_0$ and $a^{\gamma} {R_n}^{\gamma-p}\geq1$. Therefore, from \eqref{eq2913} and \eqref{eq2911}, by the regularity results in \cite{PT}, we get that
$$ \|u_{R_n}\|_{C^{1,\theta} (K)}<+\infty \,\,\,\,\quad \mbox{for}\,\,0<\theta<1$$
for any compact set $K\subset B_A \setminus B_a$. Without loss of generality, we may redefine $a$ and $A$ and assume that the $C^{1,\alpha}$ estimates hold in $\overline{B_A \setminus B_a}$. Hence, up to subsequences, we have
\begin{align}\label{eq2914}
u_{R_n} (x) \to  u_{a,A}(x),\,\quad \mbox{as}\,\,n\to\infty \,\,\,\,\,\quad\,\, \mbox{in}\,\,C^{1,\alpha^\prime}(B_A \setminus B_a)
\end{align}
for $0<\alpha^\prime<\alpha$. From \eqref{eq2913}, we deduce that
\begin{align}\label{eq2915}
-\Delta_p u_{a,A}(x)=0\,\,\,\,\,\,\,\mbox{in}\,\,B_A \setminus \overline{B_a}.
\end{align}
Now, for every $j\in\N^+$, let $a_j=\frac{1}{j}$ and $A_j=j$ and we constructs $u_{a_j,A_j}$ by reasoning as above. Then, as $j\to\infty$, a diagonal argument implies that there exist a limiting profile $u_\infty$ such that
$$ u_\infty \equiv u_{a_j,A_j} \,\,\,\,\,\,\,\mbox{in}\,\, B_{A_j} \setminus B_{a_j}, \quad \forall \,\,j\in\N^+$$
In particular,
\begin{align}\label{eq2916}
-\Delta_p u_\infty (x)=0\,\,\,\,\,\,\,\,\mbox{in}\,\,\R^N \setminus \{ 0 \}.
\end{align}
From \eqref{eq2912}, which holds with $a=a_j$ and $A=A_j$, it follows that
$$ \lim_{|x|\to\infty} u_\infty (x)=0\,\,\,\,\quad\mbox{and}\,\,\,\,\quad\lim_{|x|\to 0} u_\infty (x)= \infty.$$
Thus it follows from the uniqueness up to multipliers of $p$-harmonic maps in $\R^N\setminus \{0\}$ under suitable conditions at zero and at infinity in \cite[Theorem 2.1]{BS16} that
$$ u_\infty (x)= \frac{C}{|x|^\frac{N-p}{p-1}}$$
for some constants $C$.
Let the sequence $x_n$ be the same as in \eqref{eq2908} and set $y_n=\frac{x_n}{R_n}$. Then, by \eqref{eq2908}, it follows that
$$|\nabla u_{R_n}(y_n)| ={R_n}^\frac{N-1}{p-1} |\nabla u(x_n)|  \leq \theta_n \to 0,\quad\,\,\,\,\mbox{as}\,\,n\to\infty.$$
Since $|y_n|=1$, up to subsequences, we have that $y_n\to\overline{y} \in \partial B_1$.
Consequently, by the uniform convergence of the gradient indicated by \eqref{eq2914}, one has
$$ \nabla u_\infty(\overline{y})=0,$$
which is a contradiction since the fundamental solution $u_{\infty}$ has no critical points. This establishes the sharp asymptotic lower bound estimates for $|\nabla u|$. The sharp lower bound estimates for $|\nabla v|$ can be proved similarly.

Therefore, we have derived the sharp asymptotic estimates \eqref{eq0806} and \eqref{eq0806+} in Theorem \ref{th2.1-}, and hence finished our proof of Theorem \ref{th2.1-}. $\qquad\qquad\qquad\qquad\qquad\qquad\qquad\qquad\qquad\qquad\quad\square$

\section{Proof of Theorem \ref{th2}: radial symmetry and monotonicity of solutions}\label{sc4}

Thanks to the sharp asymptotic estimates \eqref{eq0806} and \eqref{eq0806+} in Theorem \ref{th2.1-}, we are ready to prove Theorem \ref{th2}. In this section, using the method of moving planes, we will show the radial symmetry and strictly radial monotonicity of positive $D^{1,p}$-weak solutions to the $D^{1,p}$-critical equations \eqref{eq1.1} by discussing the singular elliptic case $1<p<2$ and the degenerate elliptic case $p>2$ separately, and we complete the proof by discussing the uniqueness in the end. So we will divide this section into three sub-sections.

\medskip

Before carrying out the moving planes technique to prove Theorem \ref{th2}, we need some standard notations.

\smallskip

Let $\nu$ be any direction in $\R^N$, i.e. $\nu\in \R^N$ and $|\nu|=1$, for arbitrary $\lambda\in\R$, we define
$$ \Sigma_\lambda^\nu : =\{ x\in\R^N \mid x \cdot \nu < \lambda \}, $$
$$ T_\lambda^\nu : =\{ x\in\R^N \mid x \cdot \nu = \lambda \},$$
$$ u^\nu_\lambda (x) := u(x^\nu_\lambda), \quad \forall \,\, x\in\R^N,$$
and define the reflection $R^\nu_\lambda$ w.r.t. the hyperplane $T^\nu_\lambda$ by
$$ x_\lambda^\nu : =R^\nu_\lambda(x)= x + 2(\lambda-x\cdot \nu)\nu,\quad\,\,\,\,\forall \,\, x\in\R^N.$$
We denote the critical sets that $Z_u=\{ x\in\R^N \mid |\nabla u (x)|=0\}, Z_v=\{ x\in\R^N \mid |\nabla v (x)|=0\}$ and define
\begin{equation}\label{Z}
  Z_\lambda^\nu(u) :=\{ x\in\R^N \mid |\nabla u (x)|=|\nabla u_\lambda^\nu (x)|=0 \} \subseteq  Z_u,
\end{equation}
\begin{equation}\label{Z'}
  Z_\lambda^\nu(v) :=\{ x\in\R^N \mid |\nabla v (x)|=|\nabla v_\lambda^\nu (x)|=0 \} \subseteq  Z_v.
\end{equation}

Besides, we shall exploit the following lemma from \cite{EFRS}
\begin{lem}\label{tip}
     Let us consider $h(t)=t^s$ with $s>0$ and $b\in [\theta_1, \theta_2]$ with $\theta_i>0$. Then, for $0<a\leq b$, we have
    $$\frac{h(b)-h(a)}{b-a}\leq \overline{C}(s,\theta_1, \theta_2 ).$$
\end{lem}

\smallskip

\subsection{The degenerate elliptic case $2<p<N$}\label{sc4.1}

Since equation \eqref{eq1.1} is invariant under any rotations and translations, we may fix $\nu=e_1=(1,0,\cdots,0)$ and use the notations $\Sigma_\lambda: =\Sigma_\lambda^{e_1}$, $T_\lambda:=T_\lambda^{e_1}$ and $x_\lambda :=x_\lambda^{e_1}$ for the sake of simplicity.

Next, we define
$$ \Lambda:=\{ \lambda \in \R \mid u\leq u_\mu\, \, and \,\, v\leq v_\mu \,\,\, \mbox{in}\,\,\Sigma_\mu,\,\,\,\forall \,\, \mu \leq \lambda \},$$
and, if $\Lambda\neq\emptyset$, we set
$$\lambda_0 := \sup \Lambda.$$

Now we will prove the radial symmetry and strictly radial monotonicity in Theorem \ref{th2} for $2<p<N$.
\begin{proof}
We start moving planes from the left-hand side of the $e_1$-direction until its limiting position. The moving planes procedure consists of three main steps.

\smallskip

{\bf Step 1.} We will prove that $u\leq u_\lambda$ and $v\leq v_\lambda$ for $\lambda<0$ with $|\lambda|$ large, and hence $\Lambda\neq\emptyset$.

\smallskip

We may choose $(u-u_\lambda)^+ \chi_{\Sigma_\lambda} $ and $(v-v_\lambda)^+ \chi_{\Sigma_\lambda} $ as the test functions in the definition of $D^{1,p}(\R^N)$-weak solution of equation \eqref{eq1.1} (see Definition \ref{weak}) and get
$$ \int_{\Sigma_\lambda} |\nabla u|^{p-2} \langle \nabla u,\nabla(u-u_\lambda)^+ \rangle \mathrm{d}x= \int_{\Sigma_\lambda} u^{\alpha}v^{\beta} (u-u_\lambda)^+  \mathrm{d}x,$$
$$ \int_{\Sigma_\lambda} |\nabla v|^{p-2} \langle \nabla v,\nabla(v-v_\lambda)^+ \rangle \mathrm{d}x= \int_{\Sigma_\lambda} u^{\beta}v^{\alpha} (v-v_\lambda)^+  \mathrm{d}x.$$
 Since $(u_\lambda, v_\lambda)$ satisfies the equations \eqref{eq1.1} as well, analogously we have
$$ \int_{\Sigma_\lambda} |\nabla u_\lambda|^{p-2} \langle \nabla u_\lambda,\nabla(u-u_\lambda)^+ \rangle \mathrm{d}x= \int_{\Sigma_\lambda} u^{\alpha}_{\lambda}v^{\beta}_{\lambda} (u-u_\lambda)^+  \mathrm{d}x,$$
$$ \int_{\Sigma_\lambda} |\nabla v_\lambda|^{p-2} \langle \nabla v_\lambda,\nabla(v-v_\lambda)^+ \rangle \mathrm{d}x= \int_{\Sigma_\lambda} u^{\beta}_{\lambda}v^{\alpha}_{\lambda} (v-v_\lambda)^+  \mathrm{d}x.$$
By subtracting between the above two formulae, we obtain
\begin{align}\label{eq.mpm.u}
\int_{\Sigma_\lambda}  \langle  |\nabla u|^{p-2}\nabla u- |\nabla u_\lambda|^{p-2}\nabla u_\lambda,\nabla(u-u_\lambda)^+ \rangle \mathrm{d}x= \int_{\Sigma_\lambda} ( u^\alpha v^\beta -u^{\alpha}_{\lambda}v^{\beta}_{\lambda} ) (u-u_\lambda)^+  \mathrm{d}x,
\end{align}
and
\begin{align}\label{eq.mpm.v}
\int_{\Sigma_\lambda}  \langle  |\nabla v|^{p-2}\nabla v- |\nabla v_\lambda|^{p-2}\nabla v_\lambda,\nabla(v-v_\lambda)^+ \rangle \mathrm{d}x= \int_{\Sigma_\lambda} ( u^\beta v^\alpha -u^{\beta}_{\lambda}v^{\alpha}_{\lambda} ) (v-v_\lambda)^+  \mathrm{d}x.
\end{align}
Exploiting \eqref{eq0802}, we arrive at
\begin{align}\label{eq0803}
\int_{\Sigma_\lambda} (|\nabla u|+|\nabla u_\lambda|)^{p-2} |\nabla(u-u_\lambda)^+|^2 \mathrm{d}x \leq \frac{1}{C'} \int_{\Sigma_\lambda} (u^\alpha v^\beta -u^{\alpha}_{\lambda}v^{\beta}_{\lambda} ) (u-u_\lambda)^+  \mathrm{d}x,
\end{align}
\begin{align}\label{eq0803-}
    \int_{\Sigma_\lambda} (|\nabla v|+|\nabla v_\lambda|)^{p-2} |\nabla(v-v_\lambda)^+|^2 \mathrm{d}x \leq \frac{1}{C'} \int_{\Sigma_\lambda} (u^\beta v^\alpha -u^{\beta}_{\lambda}v^{\alpha}_{\lambda} ) (v-v_\lambda)^+  \mathrm{d}x.
\end{align}

\smallskip

Let us consider \eqref{eq0803} first. Using the fact that $u\geq u_\lambda$ in the support of $(u-u_\lambda)^+$, we get
\begin{align}\label{eq0804}
&\quad \int_{\Sigma_\lambda} ( u^\alpha v^\beta -u^{\alpha}_{\lambda}v^{\beta}_{\lambda} ) (u-u_\lambda)^+ \mathrm{d}x  \\
&= \int_{\Sigma_\lambda} (u^{\alpha}-u^\alpha_{\lambda})v^\beta (u-u_\lambda)^+ \mathrm{d}x + \int_{\Sigma_\lambda}u^\alpha_{\lambda}(v^\beta -v^{\beta}_{\lambda}) (u-u_\lambda)^+ \mathrm{d}x \nonumber\\
&\leq  \int_{\Sigma_\lambda}  (u^{\alpha}-u^\alpha_{\lambda})v^\beta (u-u_\lambda)^+\mathrm{d}x + \int_{\Sigma_\lambda}  u^{\alpha}
\frac{v^\beta -v^{\beta}_{\lambda}}{v-v_\lambda} (v-v_\lambda)^+  (u-u_\lambda)^+ \mathrm{d}x\nonumber\\
&=: I_1 + I_2. \nonumber
\end{align}

Similarly, we obtain
\begin{align}\label{eq0804-}
    &\quad \int_{\Sigma_\lambda} ( u^\beta v^\alpha -u^{\beta}_{\lambda}v^{\alpha}_{\lambda} ) (v-v_\lambda)^+ \mathrm{d}x  \\
&= \int_{\Sigma_\lambda} (v^{\alpha}-v^\alpha_{\lambda})u^\beta (v-v_\lambda)^+ \mathrm{d}x + \int_{\Sigma_\lambda}v^\alpha_{\lambda}(u^\beta -u^{\beta}_{\lambda}) (v-v_\lambda)^+ \mathrm{d}x \nonumber\\
&\leq  \int_{\Sigma_\lambda}  (v^{\alpha}-v^\alpha_\lambda)u^\beta (v-v_\lambda)^+\mathrm{d}x + \int_{\Sigma_\lambda}  \frac{u^\beta -u^{\beta}_{\lambda}}{u-u_\lambda}v^{\alpha}
(v-v_\lambda)^+  (u-u_\lambda)^+ \mathrm{d}x\nonumber\\
&=: J_1 + J_2. \nonumber
\end{align}

First, we estimate $I_1$. From Theorem \ref{th2.1-}, we obtain $c_0 \leq |x|^{-\frac{N-p}{p-1}}u, |x|^{-\frac{N-p}{p-1}}v \leq C_0 $ in $\Sigma_{\lambda}$ for large $|\lambda|$. By Lemma \ref{tip}, we obtain
$$\frac{(|x|^\frac{N-p}{p-1}u)^{\alpha}-(|x|^\frac{N-p}{p-1}u_\lambda)^{\alpha}}{|x|^\frac{N-p}{p-1}u-|x|^\frac{N-p}{p-1}u_\lambda} \leq \overline{C}_1\quad
\mbox{and} \quad
\frac{(|x|^\frac{N-p}{p-1}v)^{\beta}-(|x|^\frac{N-p}{p-1}v_\lambda)^{\beta}}{|x|^\frac{N-p}{p-1}v-|x|^\frac{N-p}{p-1}v_\lambda} \leq \overline{C}_2,$$
where $\overline{C}_1:=\overline{C}(c_0, C_0, \alpha), \overline{C}_2:=\overline{C}(c_0,C_0, \beta).$ Then we have
\begin{align}\label{eq0808}
I_1 &\leq C_0 \int_{\Sigma_\lambda} |x|^{-\frac{N-p}{p-1}\beta} |x|^{-\frac{N-p}{p-1}\alpha} \frac{(|x|^\frac{N-p}{p-1}u)^{\alpha}-(|x|^\frac{N-p}{p-1}u_\lambda)^{\alpha}}{|x|^\frac{N-p}{p-1}u-|x|^\frac{N-p}{p-1}u_\lambda} |x|^\frac{N-p}{p-1} [(u-u_\lambda)^+]^2 \mathrm{d}x \\
    &\leq C_0 \int_{\Sigma_\lambda} |x|^{-\frac{N-p}{p-1}(p^{\star}-2)}\frac{(|x|^\frac{N-p}{p-1}u)^{\alpha}-(|x|^\frac{N-p}{p-1}u_\lambda)^{\alpha}}{|x|^\frac{N-p}{p-1}u-|x|^\frac{N-p}{p-1}u_\lambda}  [(u-u_\lambda)^+]^2 \mathrm{d}x,\nonumber \\
    &\leq C_0 \overline{C}_1\int_{\Sigma_\lambda} |x|^{-\frac{N-p}{p-1}(p^{\star}-2)} [(u-u_\lambda)^+]^2 \mathrm{d}x. \nonumber
\end{align}

Now, in order to apply the weighted Hardy--Sobolev inequality (Lemma \ref{HDSI}), we set
$$ s:=-\frac{N-1}{p-1}(p-2)\quad\,\,\mbox{and}\quad\,\,\beta_1:=(p^\star -2)\frac{N-p}{p-1} +s-2.$$
Notice that
$$ s:=-\frac{N-1}{p-1}(p-2)=-\left(1-\frac{1}{p-1}\right)(N-1)>2-N$$
and
$$\beta_1 >0.$$
\begin{align}\label{eq0809}
I_1 &\leq C_0 \overline{C}_1 \frac{1}{|\lambda|^{\beta_1}} \int_{\Sigma_\lambda} |x|^{s-2} [(u-u_\lambda)^+]^2 \mathrm{d}x \\
&\leq C_0 \overline{C}_1 \frac{1}{|\lambda|^{\beta_1}} \left( \frac{2}{N+s-2} \right)^2 \int_{\Sigma_\lambda} |x|^s |\nabla(u-u_\lambda)^+|^2 \mathrm{d}x \nonumber\\
&\leq \frac{C_0 \overline{C}_1}{(c'_0)^{p-2}} \frac{1}{|\lambda|^{\beta_1}} \left( \frac{2}{N+s-2} \right)^2 \int_{\Sigma_\lambda} |\nabla u|^{p-2} |\nabla(u-u_\lambda)^+|^2 \mathrm{d}x \nonumber\\
&\leq \frac{C_0 \overline{C}_1}{(c'_0)^{p-2}} \frac{1}{|\lambda|^{\beta_1}} \left( \frac{2}{N+s-2} \right)^2 \int_{\Sigma_\lambda} \left(|\nabla u| +|\nabla u_\lambda|\right)^{p-2} |\nabla(u-u_\lambda)^+|^2 \mathrm{d}x, \nonumber
\end{align}
where $c'_0$ is given by Theorem \ref{th2.1-}. In the same way, we obtain
\begin{align}\label{eq0809-}
    J_1 \leq \frac{C_0 \overline{C}_1}{(c'_0)^{p-2}} \frac{1}{|\lambda|^{\beta_1}} \left( \frac{2}{N+s-2} \right)^2 \int_{\Sigma_\lambda} \left(|\nabla v| +|\nabla v_\lambda|\right)^{p-2} |\nabla(v-v_\lambda)^+|^2 \mathrm{d}x.
\end{align}

\smallskip

Now we estimate $I_2$. By the same way as estimates of $I_1$, we have
\begin{align}\label{eq0810}
    I_2 &\leq C_0 \overline{C}_2 \int_{\Sigma_\lambda} |x|^{-\frac{N-p}{p-1}(p^{\star}-2)} (v-v_\lambda)^+  (u-u_\lambda)^+ \mathrm{d}x \\
        &\leq \frac{C_0 \overline{C}_2}{2}\left(\int_{\Sigma_\lambda} |x|^{-\frac{N-p}{p-1}(p^{\star}-2)} [(v-v_\lambda)^+]^2 \mathrm{d}x + \int_{\Sigma_\lambda} |x|^{-\frac{N-p}{p-1}(p^{\star}-2)} [(u-u_\lambda)^+]^2 \mathrm{d}x\right) \nonumber\\
        &\leq \frac{C_0 \overline{C}_2}{2(c'_0)^{p-2}}\frac{1}{|\lambda|^{\beta_1}}\left( \frac{2}{N+s-2} \right)^2 \cdot \int_{\Sigma_\lambda}  (|\nabla u|+|\nabla u_\lambda|)^{p-2} |\nabla(u-u_\lambda)^+|^2 \mathrm{d}x \nonumber\\
        &\quad+ \frac{C_0 \overline{C}_2}{2(c'_0)^{p-2}}\frac{1}{|\lambda|^{\beta_1}}\left( \frac{2}{N+s-2} \right)^2 \cdot \int_{\Sigma_\lambda}  (|\nabla v|+|\nabla v_\lambda|)^{p-2} |\nabla(v-v_\lambda)^+|^2 \mathrm{d}x. \nonumber
\end{align}
At the end of the above inequality, we use the conclusion from $I_1$. Similarly, we have the estimates
\begin{align}\label{eq0810-}
    J_2&\leq \frac{C_0 \overline{C}_2}{2(c'_0)^{p-2}}\frac{1}{|\lambda|^{\beta_1}}\left( \frac{2}{N+s-2} \right)^2 \cdot \int_{\Sigma_\lambda}  (|\nabla u|+|\nabla u_\lambda|)^{p-2} |\nabla(u-u_\lambda)^+|^2 \mathrm{d}x \\
&\quad+ \frac{C_0 \overline{C}_2}{2(c'_0)^{p-2}}\frac{1}{|\lambda|^{\beta_1}}\left( \frac{2}{N+s-2} \right)^2 \cdot \int_{\Sigma_\lambda}  (|\nabla v|+|\nabla v_\lambda|)^{p-2} |\nabla(v-v_\lambda)^+|^2 \mathrm{d}x. \nonumber
\end{align}

Combining with \eqref{eq0803}- \eqref{eq0804-}, \eqref{eq0809}- \eqref{eq0810-}, we have
\begin{align}\label{eq0902}
&\int_{\Sigma_\lambda}  (|\nabla u|+|\nabla u_\lambda|)^{p-2} |\nabla(u-u_\lambda)^+|^2 \mathrm{d}x + \int_{\Sigma_\lambda} (|\nabla v|+|\nabla v_\lambda|)^{p-2} |\nabla(v-v_\lambda)^+|^2 \mathrm{d}x\\
&\leq \frac{C_0(\overline{C}_1+\overline{C}_2)}{C'(c'_0)^{p-2}|\lambda|^{\beta_1}} \left( \frac{2}{N+s-2} \right)^2 \int_{\Sigma_\lambda}\left(|\nabla u| +|\nabla u_\lambda|\right)^{p-2} |\nabla(u-u_\lambda)^+|^2 \mathrm{d}x \nonumber\\
&\quad+\frac{C_0(\overline{C}_1+\overline{C}_2)}{C'(c'_0)^{p-2}|\lambda|^{\beta_1}} \left( \frac{2}{N+s-2} \right)^2  \int_{\Sigma_\lambda}  (|\nabla v|+|\nabla v_\lambda|)^{p-2} |\nabla(v-v_\lambda)^+|^2 \mathrm{d}x,  \nonumber
\end{align}
where the positive constants $C_{0}$ and $c_{0}$ come from the sharp asymptotic estimates \eqref{eq0806} and \eqref{eq0806+} in Theorem \ref{th2.1-}, and $C'$ comes from the basic point-wise inequalities in Lemma \ref{BSICI}. For $\lambda$ sufficiently negative such that
$$\frac{C_0(\overline{C}_1+\overline{C}_2)}{C'(c'_0)^{p-2}|\lambda|^{\beta_1}} \left( \frac{2}{N+s-2} \right)^2 <1, $$
it follows that \eqref{eq0902} provides a contradiction unless $(u-u_\lambda)^+ =0$ and $(v-v_\lambda)^+ =0$. Therefore, for $\lambda$ sufficiently negative,
$$ u\leq u_\lambda \quad and \quad v\leq v_\lambda  \quad\,\,\,\mbox{in}\,\,\Sigma_\lambda.$$
Note that, by the strong comparison principle (Lemma \ref{th3.1}) and the fact that
\begin{align}\label{eq0903}
-\Delta_p u =u^\alpha v^\beta \leq u^{\alpha}_{\lambda}v^{\beta}_{\lambda}= -\Delta_p u_\lambda \,\, \quad\,\,\mbox{in}\,\,\Sigma_\lambda,
\end{align}
we deduce that, either $u<u_\lambda$ in $\Sigma_\lambda$ or $u=u_\lambda$ in $\Sigma_\lambda$. The conclusion about $v$ is the same.

Now we have proved $\Lambda\neq\emptyset$ and finished Step 1. By continuity, we have $u\leq u_{\lambda_0}$ and $v\leq v_{\lambda_0}$ in $\Sigma_{\lambda_0}$.

\smallskip

{\bf Step 2.} We are to show that $u=u_{\lambda_0}$ and $v=v_{\lambda_0}$ in $\Sigma_{\lambda_0}$.

Assume by contradiction that $u \not\equiv u_{\lambda_0}$ and $v \not\equiv v_{\lambda_0}$ in $\Sigma_{\lambda_0}$. Again, by the strong comparison principle (Lemma \ref{th3.1}) and \eqref{eq0903}, we deduce that $u<u_{\lambda_0}$ in $\Sigma_{\lambda_0}$. Let us take $R > 0$ large (to be determined later), $\epsilon>0$ and $\delta>0$ small, and set
$$ B^\varepsilon_R:=\Sigma_{\lambda_0 + \varepsilon} \setminus B_R, \quad\quad  S_\delta^\varepsilon:=(\Sigma_{\lambda_0 + \varepsilon} \setminus \Sigma_{\lambda_0 - \delta})\cap B_R  \quad \,\,\mbox{and}\,\,\quad  K_\delta:=\overline{B_R\cap\Sigma_{\lambda_0 - \delta}}.$$
It is clear that
$$ \Sigma_{\lambda_0 + \varepsilon}=B^\varepsilon_R \cup S_\delta^\varepsilon \cup K_\delta.$$
For $\delta>0$, we have $u_{\lambda_0}>u$ in $K_\delta$. Since $K_\delta$ is compact, there exists a small $\overline{\varepsilon}>0$ such that
$$u_{\lambda_0+\varepsilon}>u \,\,\,\,\quad \,\,\mbox{in}\,\,K_\delta$$
for any $0\leq\varepsilon\leq\overline{\varepsilon}$. For $0\leq\varepsilon\leq\overline{\varepsilon}$, we may argue as in Step 1 by choosing $(u-u_{\lambda_0+\varepsilon})^+ \chi_{\Sigma_{\lambda_0+\varepsilon}}$ as the test function. In the following, the integral in $B^\varepsilon_R$ can be estimated as in Step 1 and recalling that weighted Hardy-Sobolev inequality only requires the functions to vanish at infinity. Hence, we can deduce from the weighted Hardy-Sobolev inequality in Lemma \ref{HDSI},  H\"{o}lder's inequality and the sharp asymptotic estimates in Theorem \ref{th2.1-} that
\begin{align}\label{eq1001}
&\quad\,\,\int_{\Sigma_{\lambda_0+\varepsilon}} (|\nabla u|+|\nabla u_{\lambda_0+\varepsilon}|)^{p-2} |\nabla(u-u_{\lambda_0+\varepsilon})^+|^2 \mathrm{d}x \\
&\leq \frac{1}{C'} \int_{B^\varepsilon_R}  (u^\alpha v^\beta- u^\alpha_{\lambda_0 +\varepsilon} v^\beta_{\lambda_0 +\varepsilon} ) (u-u_{\lambda_0+\varepsilon})^+ \mathrm{d}x\nonumber\\
&\,\,+ \frac{1}{C'} \int_{S^\varepsilon_\delta} (u^\alpha-u^\alpha_{\lambda_0+\varepsilon}) v^\beta (u-u_{\lambda_0+\varepsilon})^+ \mathrm{d}x + \frac{1}{C'} \int_{S^\varepsilon_\delta} u^\alpha(v^\beta-v^\beta_{\lambda_0+\varepsilon}) (u-u_{\lambda_0+\varepsilon})^+ \mathrm{d}x \nonumber \\
&\leq \frac{C_0(2\overline{C}_1+\overline{C}_2)}{2C'(c'_0)^{p-2}R^{\beta_1}} \left( \frac{2}{N+s-2} \right)^2 \int_{\Sigma_{\lambda_0 + \varepsilon}}\left(|\nabla u| +|\nabla u_{\lambda_0+\varepsilon}|\right)^{p-2} |\nabla(u-u_{\lambda_0+\varepsilon})^+|^2 \mathrm{d}x \nonumber\\
&\quad +\frac{C_0 \overline{C}_2}{2C'(c'_0)^{p-2}R^{\beta_1}} \left( \frac{2}{N+s-2} \right)^2  \int_{\Sigma_{\lambda_0 + \varepsilon}}  (|\nabla v|+|\nabla v_{\lambda_0+\varepsilon}|)^{p-2} |\nabla(v-v_{\lambda_0+\varepsilon})^+|^2 \mathrm{d}x,  \nonumber  \\
&\quad + \frac{\|v\|^\beta_{\infty}}{C'} \int_{S^\varepsilon_\delta} \frac{u^\alpha-u^\alpha_{\lambda_0+\varepsilon}}{u-u_{\lambda_0+\varepsilon}}[(u-u_{\lambda_0+\varepsilon})^+]^2  \mathrm{d}x \nonumber \\
& \quad + \frac{\|u\|^\alpha_{\infty}}{C'} \int_{S^\varepsilon_\delta} \frac{v^\beta-v^\beta_{\lambda_0+\varepsilon}}{v-v_{\lambda_0+\varepsilon}}(v-v_{\lambda_0+\varepsilon})^+ (u-u_{\lambda_0+\varepsilon})^+ \mathrm{d}x \nonumber\\
&\leq \frac{C_0(2\overline{C}_1+\overline{C}_2)}{2C'(c'_0)^{p-2}R^{\beta_1}} \left( \frac{2}{N+s-2} \right)^2 \int_{\Sigma_{\lambda_0 + \varepsilon}}\left(|\nabla u| +|\nabla u_{\lambda_0+\varepsilon}|\right)^{p-2} |\nabla(u-u_{\lambda_0+\varepsilon})^+|^2 \mathrm{d}x \nonumber\\
&\quad+\frac{C_0\overline{C}_2}{2C'(c'_0)^{p-2}R^{\beta_1}} \left( \frac{2}{N+s-2} \right)^2  \int_{\Sigma_{\lambda_0 + \varepsilon}}  (|\nabla v|+|\nabla v_{\lambda_0+\varepsilon}|)^{p-2} |\nabla(v-v_{\lambda_0+\varepsilon})^+|^2 \mathrm{d}x,  \nonumber \\
&\quad + \frac{\overline{C}_1\|v\|^\beta_{\infty}}{C'} \int_{S^\varepsilon_\delta}[(u-u_{\lambda_0+\varepsilon})^+]^2 \mathrm{d}x  + \frac{\overline{C}_2\|u\|^\alpha_{\infty}}{2C'} \int_{S^\varepsilon_\delta}\left([(u-u_{\lambda_0+\varepsilon})^+]^2+[(v-v_{\lambda_0+\varepsilon})^+]^2\right) \mathrm{d}x \nonumber
%&\leq  \frac{(p-1) C_V  C_0^{p-2}}{C' c_0^{p-2}} \frac{1}{|R|^{\sigma p}} \left( \frac{2}{N+s-2} \right)^2 \int_{B^\varepsilon_R} |\nabla u +\nabla u_{\lambda_0+\varepsilon}|^{p-2} |\nabla(u-u_{\lambda_0+\varepsilon})^+|^2 \mathrm{d}x,\nonumber\\
%&\,\,+\frac{ p C C_0^{p^{\star}-2}}{C'  c_0^{p^{\star}-2}} \|u\|_{L^{p^{\star}}(\R^N)}^\frac{p(N-2p)}{N-p}  \frac{1}{|R|^{\beta_2}}\left( \frac{2}{N+s-2} \right)^2   \int_{B^\varepsilon_R}  (|\nabla u|+|\nabla u_{\lambda_0+\varepsilon}|)^{p-2} |\nabla(u-u_{\lambda_0+\varepsilon})^+|^2 \mathrm{d}x. \nonumber\\
%&\,\,+ \frac{(p-1)C_{2} \|u\|_\infty^{p-2}}{C'} C_p^2(S_\delta^\varepsilon) \int_{S_\delta^\varepsilon} (|\nabla u|+|\nabla u_{\lambda_0+\varepsilon}|)^{p-2} |\nabla(u-u_{\lambda_0+\varepsilon})^+|^2 \mathrm{d}x. \nonumber\\
%&\,\,+ \frac{p C \|u\|_\infty^{p^{\star}-2}}{C'} \|u\|_{L^{p^{\star}}(\R^N)}^\frac{p(N-2p)}{N-p} C_p^2(S_\delta^\varepsilon) \int_{S_\delta^\varepsilon} (|\nabla u|+|\nabla u_{\lambda_0+\varepsilon}|)^{p-2} |\nabla(u-u_{\lambda_0+\varepsilon})^+|^2 \mathrm{d}x.
\end{align}
where the positive constants $C_0, c'_0$ come from the sharp asymptotic estimates \eqref{eq0806} and \eqref{eq0806+} in Theorem \ref{th2.1-}, $\overline{C}$ comes from Lemma \ref{tip}.

\smallskip

To estimate the integral in $S_\delta^\varepsilon$, we will use the weighted Poincar\'{e} type inequality in Lemma \ref{lm.3}. For $x\in S_\delta^\varepsilon$, we let $\eta(x):=(u-u_{\lambda_0+\varepsilon})^+$. By the definition of $S_\delta^\varepsilon$ and $\eta(x)$, we have $\eta(x)=0$ on $\partial\Sigma_{\lambda_0+\varepsilon} \cup (\Sigma_{\lambda_0-\delta}\cap B_R)$, then we can define $\eta(x+te_{1}):=0$ for all $x\in \partial\Sigma_{\lambda_0+\varepsilon}\cap B_{R}$ and $t>0$. For all $x\in S_\delta^\varepsilon$, we can write
$$ \eta(x)=-\int_{0}^{+\infty}\frac{\partial \eta}{\partial x_1}(x+te_{1}) \mathrm{d}t=-\int_{0}^{\lambda_0+\varepsilon-x_{1}}\frac{\partial \eta}{\partial x_1}(x+te_{1}) \mathrm{d}t,$$
which in conjunction with integrating on $x_{2},\cdots,x_{N}$ and finite covering theorem imply that, for some $C>0$,
$$ |\eta(x)|\leq C\int_{S_\delta^\varepsilon} \frac{|\nabla \eta(y)|}{|x-y|^{N-1}}\mathrm{d}y, \quad\quad \,\,\,\,\forall \,\, x\in S_\delta^\varepsilon.$$
Thus $\eta$ satisfies \eqref{eq10.25.52}  with $\Omega=S_\delta^\varepsilon$. From Lemma \ref{lm.4}, Corollary \ref{re2333} or Remark \ref{rem0}, one has that $\rho=|\nabla u|^{p-2}$ satisfies \eqref{eq10.25.2}. So by the weighted Poincar\'{e} type inequality in Lemma \ref{lm.3}, we obtain
\begin{align}\label{eq1002}
\int_{S_\delta^\varepsilon} [(u-u_{\lambda_0+\varepsilon})^+]^2 \mathrm{d}x \leq C(S_\delta^\varepsilon) \int_{S_\delta^\varepsilon} (|\nabla u|+|\nabla u_{\lambda_0+\varepsilon}|)^{p-2} |\nabla(u-u_{\lambda_0+\varepsilon})^+|^2 \mathrm{d}x,
\end{align}
where $C(S_\delta^\varepsilon)\to 0$ if $|S_\delta^\varepsilon|\to 0$. By entirely similar way, we can also derive that
\begin{align}\label{eq1002'}
\int_{S_\delta^\varepsilon} [(v-v_{\lambda_0+\varepsilon})^+]^2 \mathrm{d}x \leq C(S_\delta^\varepsilon) \int_{S_\delta^\varepsilon} (|\nabla v|+|\nabla v_{\lambda_0+\varepsilon}|)^{p-2} |\nabla(v-v_{\lambda_0+\varepsilon})^+|^2 \mathrm{d}x.
\end{align}

Collecting \eqref{eq1001} and \eqref{eq1002}, we deduce
\begin{align}\label{eq1003}
&\quad \int_{\Sigma_{\lambda_0+\varepsilon}} (|\nabla u|+|\nabla u_{\lambda_0+\varepsilon}|)^{p-2} |\nabla(u-u_{\lambda_0+\varepsilon})^+|^2 \mathrm{d}x \\
&\leq \frac{C_0(2\overline{C}_1+\overline{C}_2)}{2C'(c'_0)^{p-2}R^{\beta_1}} \left( \frac{2}{N+s-2} \right)^2 \int_{\Sigma_{\lambda_0 + \varepsilon}}\left(|\nabla u| +|\nabla u_{\lambda_0+\varepsilon}|\right)^{p-2} |\nabla(u-u_{\lambda_0+\varepsilon})^+|^2 \mathrm{d}x \nonumber\\
&\quad +\frac{C_0 \overline{C}_2}{2C'(c'_0)^{p-2}R^{\beta_1}} \left( \frac{2}{N+s-2} \right)^2  \int_{\Sigma_{\lambda_0 + \varepsilon}}  (|\nabla v|+|\nabla v_{\lambda_0+\varepsilon}|)^{p-2} |\nabla(v-v_{\lambda_0+\varepsilon})^+|^2 \mathrm{d}x,  \nonumber \\
&\quad + \frac{2\overline{C}_1\|v\|^\beta_{\infty}+\overline{C}_2\|u\|_{\infty}^\alpha}{2C'} \cdot C(S_\delta^\varepsilon) \int_{S_\delta^\varepsilon} (|\nabla u|+|\nabla u_{\lambda_0+\varepsilon}|)^{p-2} |\nabla(u-u_{\lambda_0+\varepsilon})^+|^2 \mathrm{d}x \nonumber\\
&\quad + \frac{\overline{C}_2\|u\|^\alpha_{\infty}}{2C'} \cdot C(S_\delta^\varepsilon) \int_{S_\delta^\varepsilon} (|\nabla v|+|\nabla v_{\lambda_0+\varepsilon}|)^{p-2} |\nabla(v-v_{\lambda_0+\varepsilon})^+|^2 \mathrm{d}x. \nonumber
\end{align}

Arguing in the same way, we obtain
\begin{align}\label{eq1003-}
&\quad \int_{\Sigma_{\lambda_0+\varepsilon}} (|\nabla v|+|\nabla v_{\lambda_0+\varepsilon}|)^{p-2} |\nabla(v-v_{\lambda_0+\varepsilon})^+|^2 \mathrm{d}x \\
&\leq \frac{C_0(2\overline{C}_1+\overline{C}_2)}{2C'(c'_0)^{p-2}R^{\beta_1}} \left( \frac{2}{N+s-2} \right)^2 \int_{\Sigma_{\lambda_0 + \varepsilon}}\left(|\nabla v| +|\nabla v_{\lambda_0+\varepsilon}|\right)^{p-2} |\nabla(v-v_{\lambda_0+\varepsilon})^+|^2 \mathrm{d}x \nonumber\\
&\quad +\frac{C_0 \overline{C}_2}{2C'(c'_0)^{p-2}R^{\beta_1}} \left( \frac{2}{N+s-2} \right)^2  \int_{\Sigma_{\lambda_0 + \varepsilon}}  (|\nabla u|+|\nabla u_{\lambda_0+\varepsilon}|)^{p-2} |\nabla(u-u_{\lambda_0+\varepsilon})^+|^2 \mathrm{d}x,  \nonumber \\
&\quad + \frac{2\overline{C}_1\|u\|^\beta_{\infty}+\overline{C}_2\|v\|_{\infty}^\alpha}{2C'} \cdot C(S_\delta^\varepsilon) \int_{S_\delta^\varepsilon} (|\nabla v|+|\nabla v_{\lambda_0+\varepsilon}|)^{p-2} |\nabla(v-v_{\lambda_0+\varepsilon})^+|^2 \mathrm{d}x \nonumber\\
&\quad + \frac{\overline{C}_2\|v\|^\alpha_{\infty}}{2C'} \cdot C(S_\delta^\varepsilon) \int_{S_\delta^\varepsilon} (|\nabla u|+|\nabla u_{\lambda_0+\varepsilon}|)^{p-2} |\nabla(u-u_{\lambda_0+\varepsilon})^+|^2 \mathrm{d}x. \nonumber
\end{align}

So we have
\begin{align}\label{eq.u.v.}
    &\quad \int_{\Sigma_{\lambda_0 + \varepsilon}}  (|\nabla u|+|\nabla u_{\lambda_0+\varepsilon}|)^{p-2} |\nabla(u-u_{\lambda_0+\varepsilon})^+|^2 \mathrm{d}x \\
    &\quad + \int_{\Sigma_{\lambda_0 + \varepsilon}} (|\nabla v|+|\nabla v_{\lambda_0+\varepsilon}|)^{p-2} |\nabla(v-v_{\lambda_0+\varepsilon})^+|^2 \mathrm{d}x \nonumber\\
    &\leq \frac{C_0(\overline{C}_1+\overline{C}_2)}{C'(c'_0)^{p-2}R^{\beta_1}} \left( \frac{2}{N+s-2} \right)^2 \int_{\Sigma_{\lambda_0 + \varepsilon}}\left(|\nabla u| +|\nabla u_{\lambda_0+\varepsilon}|\right)^{p-2} |\nabla(u-u_{\lambda_0+\varepsilon})^+|^2 \mathrm{d}x \nonumber\\
    &\,\,+\frac{C_0(\overline{C}_1+\overline{C}_2)}{C'(c'_0)^{p-2}R^{\beta_1}} \left( \frac{2}{N+s-2} \right)^2  \int_{\Sigma_{\lambda_0 + \varepsilon}}  (|\nabla v|+|\nabla v_{\lambda_0+\varepsilon}|)^{p-2} |\nabla(v-v_{\lambda_0+\varepsilon})^+|^2 \mathrm{d}x \nonumber\\
&\,\,+\frac{2\overline{C}_1\|v\|^\beta_{\infty}+\overline{C}_2(\|v\|^\alpha_{\infty}+\|u\|_{\infty}^\alpha)}{2C'} \cdot C(S_\delta^\varepsilon) \int_{S_\delta^\varepsilon} (|\nabla u|+|\nabla u_{\lambda_0+\varepsilon}|)^{p-2} |\nabla(u-u_{\lambda_0+\varepsilon})^+|^2 \mathrm{d}x \nonumber\\
&\,\,+\frac{2\overline{C}_1\|u\|^\beta_{\infty}+\overline{C}_2(\|u\|^\alpha_{\infty}+\|v\|_{\infty}^\alpha)}{2C'} \cdot C(S_\delta^\varepsilon) \int_{S_\delta^\varepsilon} (|\nabla v|+|\nabla v_{\lambda_0+\varepsilon}|)^{p-2} |\nabla(v-v_{\lambda_0+\varepsilon})^+|^2 \mathrm{d}x. \nonumber
\end{align}
Now we take care of the variable parameters $R, \delta, \overline{\varepsilon}$. First we fix $R$ large so that
$$\frac{C_0(\overline{C}_1+\overline{C}_2)}{C'(c'_0)^{p-2}R^{\beta_1}} \left( \frac{2}{N+s-2} \right)^2 < \frac{1}{4} .$$
Then, recalling that $|S_\delta^\varepsilon|\to 0$ if $\delta\to 0$ and $\varepsilon\to 0$, we choose $\delta$ and $\overline{\varepsilon}$ small such that
$$\frac{2\overline{C}_1\|v\|^\beta_{\infty}+\overline{C}_2(\|v\|^\alpha_{\infty}+\|u\|_{\infty}^\alpha)}{2C'} \cdot C(S_\delta^\varepsilon) <\frac{1}{4}$$
and
$$\frac{2\overline{C}_1\|u\|^\beta_{\infty}+\overline{C}_2(\|u\|^\alpha_{\infty}+\|v\|_{\infty}^\alpha)}{2C'} \cdot C(S_\delta^\varepsilon)<\frac{1}{4}$$
for any $0\leq\varepsilon\leq\overline{\varepsilon}$. With such choice of parameters, by \eqref{eq.u.v.}, we have
$$\int_{\Sigma_{\lambda_0+\varepsilon}} (|\nabla u|+|\nabla u_{\lambda_0+\varepsilon}|)^{p-2} |\nabla(u-u_{\lambda_0+\varepsilon})^+|^2 \mathrm{d}x + \int_{\Sigma_{\lambda_0+\varepsilon}} (|\nabla v|+|\nabla v_{\lambda_0+\varepsilon}|)^{p-2} |\nabla(v-v_{\lambda_0+\varepsilon})^+|^2 \mathrm{d}x\leq 0,$$
which implies $(u-u_{\lambda_0+\varepsilon})^+=0$ and $(v-v_{\lambda_0+\varepsilon})^+=0$ in $\Sigma_{\lambda_0+\varepsilon}$ for any $0\leq\varepsilon\leq\overline{\varepsilon}$. This is a contradiction with the definition of $\lambda_0$.
%Consequently we have proved that necessarily
Thus, we prove that
$$ u=u_{\lambda_0} \quad \mbox{and} \quad v=v_{\lambda_0} \quad \quad\quad \,\,\mbox{in}\,\,\Sigma_{\lambda_0}.$$
Furthermore, $ u<u_{\lambda}$ in $\Sigma_\lambda$ for any $\lambda < \lambda_0$ and consequently $u$ is strictly monotone increasing in $\Sigma_{\lambda_0}$ so that $\frac{\partial u}{\partial x_{1}}\geq 0$ in $\Sigma_{\lambda_0}$. And we have the same consequence for v.

\medskip

{\bf Step 3.} $u$ and $v$ are radially symmetric and strictly decreasing about some point $x_0\in\R^N$.

By the strong maximum principle of the linearized operator \cite[Theorem 1.2]{LDBS}, it actually follows that
\begin{equation}\label{sd}
  \frac{\partial u}{\partial x_{1}}> 0 \quad \mbox{and} \quad \frac{\partial v}{\partial x_{1}}> 0\qquad \text{in} \,\, \Sigma_{\lambda_0}.
\end{equation}

%By the strong maximum principle for the linearized operator
%\cite[Theorem 1.2]{LDBS} it follows that actually $u_{x_1} > 0$ in $\Sigma_{\lambda_0}$.

It is known that for any direction $\nu$, $u$ and $v$ are symmetric about the plane $T^\nu_{\lambda_0(\nu)}$. Now consider all the $N$ orthogonal directions $\{e_k\}_{k=1}^N$ in $\R^N$, and one can obtain in entirely similar way that $u$ and $v$ are strictly increasing in each direction $e_k$ and $\frac{\partial u}{\partial e_{k}}, \frac{\partial v}{\partial e_{k}}>0$ in $\Sigma^{e_k}_{\lambda_0(e_k)}$ and symmetric about the plane $T^{e_k}_{\lambda_0(e_k)}$. This shows that $u$ and $v$ are symmetric with respect to a point $x_0\in\R^N$ $(x_0 = \bigcap^N_{k=1}T^{e_k}_{\lambda_0(e_k)})$, which is the unique critical point for $u$ and $v$. Furthermore, by exploiting the moving plane procedure in any direction $\nu\in S^{N-1}$, we finally get that $u$ and $v$ are radial and (strictly) radially decreasing w.r.t. $x_{0}$. This completes our proof of radial symmetry and decreasing for $2<p<N$.
\end{proof}

\subsection{The singular elliptic case $1<p<2$}\label{sc4.2}

In this sub-section, we carry out our proof of Theorem \ref{th2} in the singular elliptic case $1<p<2$ via the moving planes method. In the following, we define the set
$$\overline{\Lambda}(\nu):=\{ \lambda\in\R \mid u\geq u_\mu^\nu\,\, \mbox{and} \,\, v\geq v_\mu^\nu\,\quad \mbox{in}\,\,\Sigma_\mu^\nu,\,\,\,\,\forall \,\, \mu\geq\lambda \},$$
and, if $\overline{\Lambda}(\nu)\neq\emptyset$, we define
$$ \lambda_0(\nu):=\inf \overline{\Lambda}(\nu).$$

\subsubsection{Weak comparison principles}
For the singular elliptic case $1<p<2$, in order to show the radial symmetry and strictly radial monotonicity of the weak solution $u$ to \eqref{eq1.1} via the method of moving planes, we need weak comparison principles in unbounded domains. To this end, we will use the following Poincar\'{e} type inequality from \cite{LDMR} (see Lemma \ref{PTI}).

Let us consider a ball $B(P,R)$ with center $P\in T_\lambda^\nu$ and radius $R$. Set
$$B_\lambda^\nu:=B(P,R)\cap\Sigma_\lambda^\nu,$$
and for any $A\subset\R^N$, let
$$ M_A:=\max\{\sup_A (|\nabla u|+|\nabla u^{\nu}_{\lambda}|), \sup_A (|\nabla v|+|\nabla v^{\nu}_{\lambda}|)\}.$$
For any $A\subset \Sigma_\lambda^\nu$, define
$$A^\prime:=R_\lambda^\nu(A).$$

We apply the following Poincar\'{e} type inequality to functions on $B_\lambda^\nu$, vanishing on $T^\nu_\lambda$.
\begin{lem}[Poincar\'{e} type inequality, \cite{LDMR}]\label{PTI}
Let $w\in W^{1,q}_{loc} (\R^N)$ with $1\leq q< \infty$ vanish on $T^\nu_\lambda$, and let $(supp\,w)\cap B_\lambda^\nu =A_1\cup A_2$ for some disjoint measurable sets $A_i$. Then there exists a constant $C=C(N)$ such that
$$ \| w \|_{L^q(B_\lambda^\nu)} \leq C |B_\lambda^\nu\cap supp \, w|^\frac{1}{Nq} \left( |A_1|^\frac{1}{Nq^\prime} \| \nabla w \|_{L^q(A_1)} + |A_2|^\frac{1}{Nq^\prime} \| \nabla w \|_{L^q(A_2)}  \right),$$
where $ \frac{1}{q}+ \frac{1}{q^\prime} =1 $.
\end{lem}

Now, with the help of Lemma \ref{PTI}, we are able to prove the following weak comparison principle.
\begin{lem}\label{lm.4.1}
Assume $N>2$, $1<p<2$ and $(u, v)$ is a pair of positive weak solutions of \eqref{eq1.1}. Let $R_0$ be the same as in Theorem \ref{th2.1-}.

{\bf $(i)$} There exists $R_1>R_0$, $R_1$ depending on $N,p,R_0$, such that if
$$\Sigma_\lambda^\nu \cap R_\lambda^\nu(B_{R_1}) = \emptyset,$$
then $u^{\nu}_{\lambda} \leq u$ and $v^{\nu}_{\lambda} \leq v$ in $\Sigma_\lambda^\nu$;

{\bf $(ii)$} Let $\Sigma_\lambda^\nu \cap R_\lambda^\nu(B_{R_1}) \neq \emptyset$ and let $R_2=R_2(R_1,\lambda,\nu)>2R_{1}$ be such that $\Sigma_\lambda^\nu \cap R_\lambda^\nu(B_{R_1}) \subset B(P_\lambda^\nu,R_2)$, where $P_\lambda^\nu$ is the projection of the origin on $T_\lambda^\nu$.
Then there exists $R>R_2$, depending on $R_2$, such that, for any $\overline{R}\geq R$, there exist positive numbers $\delta$ and $M$, depending $\overline{R},\lambda,p,N$ and the $L^\infty$--norms of $u,v,|\nabla u|,|\nabla v|$, with the property that whenever
$$ (supp(u^{\nu}_{\lambda}-u)^+\cup supp(v^{\nu}_{\lambda}-v)^+)\cap\Sigma^{\nu}_{\lambda}\cap B(P_\lambda^\nu,\overline{R})=A \cup B,$$
with $|A\cap B|=0$, $M_A<M$ and $|A|+|B|<\delta$, we have
$$ u^{\nu}_{\lambda} \leq u\,\,\,\,\mbox{and}\,\,\,\, v^{\nu}_{\lambda} \leq v \,\,\quad \mbox{in}\,\,\Sigma_\lambda^\nu.$$
\end{lem}

\begin{proof}

%We claim that there exists $R_1 > R_0$, with $R_0$ be as in Theorem \ref{th2.1}, such that if
%$$ \Sigma_\lambda \cap R_\lambda^\nu(B_{R_1}) = \emptyset $$
%then $u\geq u_\lambda$ in $\Sigma_\lambda$.

Let us take $(u^{\nu}_{\lambda}-u)^+\chi_{\Sigma^{\nu}_\lambda}$ and $(v^{\nu}_{\lambda}-v)^+\chi_{\Sigma^{\nu}_\lambda}$ as the test functions in the Definition \ref{weak} of weak solutions. Proceeding as in the deduction of \eqref{eq0803} and \eqref{eq0803-}, we get
\begin{align}\label{eq1004}
&\quad \int_{\Sigma_\lambda^\nu} (|\nabla u|+|\nabla u^{\nu}_{\lambda}|)^{p-2} |\nabla(u^{\nu}_{\lambda}-u)^+|^2 \mathrm{d}x \\
&\leq \frac{1}{C'} \int_{\Sigma_\lambda^\nu} ( (u^\nu_\lambda)^\alpha (v^\nu_\lambda)^\beta-u^\alpha v^\beta)(u^{\nu}_{\lambda} - u)^+  \mathrm{d}x  \nonumber\\
&= \frac{1}{C'}\int_{\Sigma^\nu_\lambda} ((u^\nu_{\lambda})^\alpha-u^\alpha)(v^\nu_{\lambda})^\beta (u^\nu_{\lambda}-u)^+ \mathrm{d}x + \frac{1}{C'}\int_{\Sigma^\nu_\lambda}u^\alpha((v^\nu_{\lambda})^\beta -v^{\beta}) (u^\nu_{\lambda}-u)^+ \mathrm{d}x  \nonumber \\
%&\leq \frac{p-1}{C'} \int_{\Sigma_\lambda}  V(x) u^{p-2} [(u-u_\lambda)^+]^2 \mathrm{d}x + \frac{1}{C'}\int_{\Sigma_\lambda} P_\lambda(x) u^{p-1}   (u-u_\lambda)^+ \mathrm{d}x\nonumber\\
&=:II_1 + II_2, \nonumber
\end{align}
\begin{align}\label{eq1004'}
&\quad \int_{\Sigma_\lambda^\nu} (|\nabla v|+|\nabla v^{\nu}_{\lambda}|)^{p-2} |\nabla(v^{\nu}_{\lambda}-v)^+|^2 \mathrm{d}x \\
&\leq \frac{1}{C'} \int_{\Sigma_\lambda^\nu} ( (v^\nu_\lambda)^\alpha (u^\nu_\lambda)^\beta-v^\alpha u^\beta)(v^{\nu}_{\lambda} - v)^+  \mathrm{d}x  \nonumber\\
&= \frac{1}{C'}\int_{\Sigma^\nu_\lambda} ((v^\nu_{\lambda})^\alpha-v^\alpha)(u^\nu_{\lambda})^\beta (v^\nu_{\lambda}-v)^+ \mathrm{d}x + \frac{1}{C'}\int_{\Sigma^\nu_\lambda}v^\alpha((u^\nu_{\lambda})^\beta -u^{\beta}) (v^\nu_{\lambda}-v)^+ \mathrm{d}x  \nonumber \\
%&\leq \frac{p-1}{C'} \int_{\Sigma_\lambda}  V(x) u^{p-2} [(u-u_\lambda)^+]^2 \mathrm{d}x + \frac{1}{C'}\int_{\Sigma_\lambda} P_\lambda(x) u^{p-1}   (u-u_\lambda)^+ \mathrm{d}x\nonumber\\
&=:II'_1 + II'_2. \nonumber
\end{align}

\medskip

\noindent {\bf Proof of $(i)$.} Let $R > R_0$, with $R_0$ be the same as in Theorem \ref{th2.1-}, to be determined later. By the sharp asymptotic estimates for $u$ in Theorem \ref{th2.1-}, there exists $R_1= R_1(R) > R$ such that
$$ \inf_{B_R} u \geq \sup_{\R^N\setminus B_{R_1}} u, \quad \mbox{and} \quad \inf_{B_R} v \geq \sup_{\R^N\setminus B_{R_1}} v,$$
where $B_{r}:=B_{r}(0)$ for any $r>0$. If $\lambda>0$ is sufficiently large such that
$$\Sigma_\lambda^\nu \cap R_\lambda^\nu(B_{R_1}) = \emptyset,$$
then we have
$$u\geq u^{\nu}_{\lambda} \quad \,\mbox{in}\,\, B_R, \,\,\quad \,\, supp(u^{\nu}_{\lambda}-u)^+\cap\Sigma^{\nu}_{\lambda}\subset \Sigma_\lambda^\nu\setminus B_R,$$
and
$$v\geq v^{\nu}_{\lambda} \quad \,\mbox{in}\,\, B_R, \,\,\quad \,\, supp(v^{\nu}_{\lambda}-v)^+\cap\Sigma^{\nu}_{\lambda}\subset \Sigma_\lambda^\nu\setminus B_R.$$

Since $1<p<2$, by Lemma \ref{tip}, Theorem \ref{th2.1-} and \eqref{a6},
%the fact that $u^p -1$ is convex in $u$ and the fact that $u_\lambda>u$
%in the support of $(u_\lambda-u)^+$, we obtain First, we estimate $II_1$. From Theorem \ref{th2.1} and Lemma \ref{rm.2},
we obtain that
\begin{align}\label{eq1005}
II_1 &= \frac{1}{C'} \int_{\Sigma_\lambda^\nu}   ((u^\nu
_{\lambda})^\alpha-u^\alpha)(v^\nu_{\lambda})^\beta (u^{\nu}_{\lambda}-u )^+ \mathrm{d}x \\
     &\leq \frac{C_0}{C'} \int_{\Sigma_\lambda^\nu\setminus B_R} |x^\nu_\lambda|^{-\frac{N-p}{p-1}\beta}|x^\nu_\lambda|^{-\frac{N-p}{p-1}\alpha}\cdot \frac{(|x^\nu_\lambda|^{\frac{N-p}{p-1}}u^\nu_\lambda)^\alpha-(|x^\nu_\lambda|^{\frac{N-p}{p-1}}u)^\alpha}{|x^\nu_\lambda|^{\frac{N-p}{p-1}}u^\nu_\lambda-|x^\nu_\lambda|^{\frac{N-p}{p-1}}u} |x^\nu_\lambda|^{\frac{N-p}{p-1}} [(u^{\nu}_{\lambda}-u )^+]^2 \mathrm{d}x \nonumber\\
     &\leq \frac{C_0 \overline{C}_1}{C'} \int_{\Sigma_\lambda^\nu\setminus B_R}  \frac{1}{|x^\nu_\lambda|^{\frac{N-p}{p-1}(p^\star-2)}} [(u^{\nu}_{\lambda}-u )^+]^2 \mathrm{d}x. \nonumber
\end{align}
 Similar to \eqref{eq0809}, using the fact that  $|x|<|x^{\nu}_{\lambda}|$ in $\Sigma^{\nu}_\lambda$ and the weighted Hardy-Sobolev inequality, we get
\begin{align}\label{eq10061}
II_1 &\leq \frac{C_0 \overline{C}_1}{C'} \sup_{\Sigma_\lambda^\nu\setminus B_R} \left(\frac{1}{|x|^{\beta_1}}\right)
      \int_{\Sigma_\lambda^\nu\setminus B_R} |x|^{s-2} [(u^{\nu}_{\lambda}-u )^+]^2 \mathrm{d}x \\
      &\leq \frac{C_0 \overline{C}_1}{C'} \sup_{\Sigma_\lambda^\nu\setminus B_R} \left(\frac{1}{|x|^{\beta_1}}\right)
      \left(\frac{2}{N+s-2}\right)^2 \int_{\Sigma_\lambda^\nu\setminus B_R} |x|^s |\nabla(u^{\nu}_{\lambda}-u )^+|^2 \mathrm{d}x, \nonumber
\end{align}
where $s =-\frac{N-1}{p-1}(p-2)>2-N$ and $\beta_1=(p^\star -2)\frac{N-p}{p-1} +s-2>0$. By Theorem \ref{th2.1-}, and the fact that $|x|<|x^{\nu}_{\lambda}|$ in $\Sigma^{\nu}_\lambda$, we have
$$|\nabla u|+|\nabla u^{\nu}_{\lambda}| \leq C'_0 \left( \frac{1}{|x|^\frac{N-1}{p-1}} + \frac{1}{|x^{\nu}_{\lambda}|^\frac{N-1}{p-1}} \right) \leq
  \frac{2C'_0}{|x|^\frac{N-1}{p-1}} \,\,\quad \mbox{in}\,\,\Sigma_\lambda^\nu\setminus B_R,$$
which implies
$$(|\nabla u|+|\nabla u^{\nu}_{\lambda}|)^{p-2} \geq
  \frac{(2C'_0)^{p-2}}{|x|^{\frac{N-1}{p-1}(p-2)}}=(2C'_0)^{p-2} |x|^s  \quad \,\,\mbox{in}\,\,\Sigma_\lambda^\nu\setminus B_R.$$
This yields
\begin{align}\label{eq1006}
II_1 \leq  \frac{ C_0\overline{C}_1}{C'(2C'_0)^{p-2}} \frac{1}{R^{\beta_1}}
           \left(\frac{2}{N+s-2}\right)^2\int_{\Sigma_\lambda^\nu} (|\nabla u|+|\nabla u^{\nu}_{\lambda}|)^{p-2} |\nabla(u^{\nu}_{\lambda}-u )^+|^2 \mathrm{d}x.
\end{align}

For $II_2$, arguing in the same way with Young inequality, we get
\begin{small}
\begin{align}\label{eq1007}
\quad II_2 &\leq \frac{C_0 \overline{C}_2}{C'} \int_{\Sigma_\lambda^\nu\setminus B_R}  \frac{1}{|x^\nu_\lambda|^{\frac{N-p}{p-1}(p^\star-2)}} (v^{\nu}_{\lambda}-v )^+(u^{\nu}_{\lambda}-u )^+ \mathrm{d}x \\
    &\leq \frac{ C_0\overline{C}_2}{2C'(2C'_0)^{p-2}} \frac{1}{R^{\beta_1}}
    \left(\frac{2}{N+s-2}\right)^2\int_{\Sigma_\lambda^\nu} (|\nabla u|+|\nabla u^{\nu}_{\lambda}|)^{p-2} |\nabla(u^{\nu}_{\lambda}-u )^+|^2 \mathrm{d}x \nonumber \\
    & \quad + \frac{ C_0\overline{C}_2}{2C'(2C'_0)^{p-2}} \frac{1}{R^{\beta_1}}
    \left(\frac{2}{N+s-2}\right)^2\int_{\Sigma_\lambda^\nu} (|\nabla v|+|\nabla v^{\nu}_{\lambda}|)^{p-2} |\nabla(v^{\nu}_{\lambda}-v )^+|^2 \mathrm{d}x. \nonumber
\end{align}\end{small}

We get
\begin{align}\label{eq1010}
        &\quad \int_{\Sigma_\lambda^\nu}(|\nabla u|+|\nabla u^{\nu}_{\lambda}|)^{p-2} |\nabla (u^{\nu}_{\lambda}-u)^+|^2 \mathrm{d}x \\
        &\leq \frac{C_0(2\overline{C}_1+\overline{C}_2)}{2C'(2C'_0)^{p-2}} \frac{1}{R^{\beta_1}}
    \left(\frac{2}{N+s-2}\right)^2\int_{\Sigma_\lambda^\nu} (|\nabla u|+|\nabla u^{\nu}_{\lambda}|)^{p-2} |\nabla(u^{\nu}_{\lambda}-u )^+|^2 \mathrm{d}x \nonumber \\
    & \quad + \frac{ C_0\overline{C}_2}{2C'(2C'_0)^{p-2}} \frac{1}{R^{\beta_1}}
    \left(\frac{2}{N+s-2}\right)^2\int_{\Sigma_\lambda^\nu} (|\nabla v|+|\nabla v^{\nu}_{\lambda}|)^{p-2} |\nabla(v^{\nu}_{\lambda}-v )^+|^2 \mathrm{d}x \nonumber
\end{align}
As for $v$, by \eqref{eq1004'} and similar argument, we can also get
\begin{align}\label{eq1010-}
        &\quad \int_{\Sigma_\lambda^\nu}(|\nabla v|+|\nabla v^{\nu}_{\lambda}|)^{p-2} |\nabla (v^{\nu}_{\lambda}-v)^+|^2 \mathrm{d}x \\
        &\leq \frac{C_0(2\overline{C}_1+\overline{C}_2)}{2C'(2C'_0)^{p-2}}\frac{1}{R^{\beta_1}}\left(\frac{2}{N+s-2}\right)^2
        \int_{\Sigma_\lambda^\nu} (|\nabla v|+|\nabla v^{\nu}_{\lambda}|)^{p-2} |\nabla(v^{\nu}_{\lambda}-v )^+|^2 \mathrm{d}x  \nonumber\\
        & + \frac{ C_0\overline{C}_2}{2C'(2C'_0)^{p-2}} \frac{1}{R^{\beta_1}}\left(\frac{2}{N+s-2}\right)^2
        \int_{\Sigma_\lambda^\nu}(|\nabla u|+|\nabla u^{\nu}_{\lambda}|)^{p-2} |\nabla (u^{\nu}_{\lambda}-u)^+|^2 \mathrm{d}x.  \nonumber
\end{align}
Adding up \eqref{eq1010} and \eqref{eq1010-}, we can now choose $R > R_0$ such that
$$\frac{C_0(\overline{C}_1+\overline{C}_2)}{C'(2C'_0)^{p-2}} \frac{1}{R^{\beta_1}}\left(\frac{2}{N+s-2}\right)^2 <1$$
and then take $R_1 = R_1(R)$.  Then \eqref{eq1010} and \eqref{eq1010-} yield that $(u^{\nu}_{\lambda}-u)^+=0$ and $(v^{\nu}_{\lambda}-v)^+=0$ a.e. in $\Sigma_\lambda^\nu$, which implies that $u\geq u^{\nu}_{\lambda}$ and $v \geq v^\nu_\lambda$ in $\Sigma_\lambda^\nu$ for any $\lambda>R_1$.

\medskip

\noindent {\bf Proof of $(ii)$.}  Now suppose that $\Sigma_\lambda^\nu\cap R_\lambda^\nu(B_{R_1}) \neq \emptyset$ and let $R>R_2$ to be determined later. For $\overline{R}\geq R$, we define $K:=B(P_\lambda^\nu,\overline{R})$, where $P_\lambda^\nu \in T_\lambda^\nu$ is the projection of the origin on $T_\lambda^\nu$. Observe that, for $R$ sufficiently large, in $\Sigma_\lambda^\nu\cap K^c \cap supp(u^{\nu}_{\lambda}-u)^+$, we have $u \leq u^{\nu}_{\lambda}$ and the sharp asymptotic estimates \eqref{eq0806} and \eqref{eq0806+} hold. Furthermore, $|x| \approx |x^{\nu}_{\lambda}|$. Thus, similar to the proof of {(i)} and the estimate of $II_2$, we get
\begin{align}\label{eq1201}
&\quad \int_{\Sigma_\lambda^\nu} (|\nabla u|+|\nabla u^{\nu}_{\lambda}|)^{p-2} |\nabla(u^{\nu}_{\lambda}-u)^+|^2 \mathrm{d}x \\
&\leq \frac{1}{C'}\int_{\Sigma^\nu_\lambda \setminus K} ((u^\nu_{\lambda})^\alpha(v^\nu_{\lambda})^\beta- u^\alpha v^\beta)(u^\nu_{\lambda}-u)^+ \mathrm{d}x + \frac{1}{C'}\int_{A^\nu_\lambda}((u^\nu_{\lambda})^\alpha(v^\nu_{\lambda})^\beta- u^\alpha v^\beta) (u^\nu_{\lambda}-u)^+\mathrm{d}x  \nonumber\\
&\leq \frac{C_0(2\overline{C}_1+\overline{C}_2)}{2C'(2C'_0)^{p-2}} \frac{1}{R^{\beta_1}}\left(\frac{2}{N+s-2}\right)^2 \int_{\Sigma_\lambda^\nu}(|\nabla u|+|\nabla u^{\nu}_{\lambda}|)^{p-2} |\nabla (u^{\nu}_{\lambda}-u)^+|^2 \mathrm{d}x \nonumber\\
& \quad +\frac{ C_0\overline{C}_2}{2C'(2C'_0)^{p-2}}\frac{1}{R^{\beta_1}}\left(\frac{2}{N+s-2}\right)^2\int_{\Sigma_\lambda^\nu} (|\nabla v|+|\nabla v^{\nu}_{\lambda}|)^{p-2} |\nabla(v^{\nu}_{\lambda}-v )^+|^2 \mathrm{d}x  \nonumber\\
&\quad+\frac{1}{C'}\int_{A^\nu_\lambda} ((u^\nu_{\lambda})^\alpha-u^\alpha) (v^\nu_{\lambda})^\beta (u^\nu_{\lambda}-u)^+ \mathrm{d}x+ \frac{1}{C'}\int_{A^\nu_\lambda} u^\alpha((v^\nu_{\lambda})^\beta -v^{\beta})(u^\nu_{\lambda}-u)^+ \mathrm{d}x \nonumber
\end{align}
where $A_\lambda^\nu:=K\cap\Sigma_\lambda^\nu$. Similar to the estimates \eqref{eq1005}-\eqref{eq1006} on $II_1$ in the proof of {(i)}, we can obtain the above estimates.

Since $(u^{\nu}_{\lambda}-u )^+$ vanishes on $T_\lambda^\nu$ and belongs to $W^{1,p}_{loc}(\R^N)$ and $(supp(u^{\nu}_{\lambda}-u)^+\cup supp(v^{\nu}_{\lambda}-v)^+) \cap\Sigma^{\nu}_{\lambda}\cap B(P_\lambda^\nu,\overline{R})=A \cup B$, by the Poincar\'{e} type inequality in Lemma \ref{PTI}, we have
\begin{align}\label{eq1203}
&\quad \frac{1}{C'}\int_{A^\nu_\lambda}((u^\nu_{\lambda})^\alpha-u^\alpha) (v^\nu_{\lambda})^\beta (u^\nu_{\lambda}-u)^+ \mathrm{d}x \\
&\leq \frac{\|v^\nu_\lambda\|^\beta_\infty }{C'} \int_{A_\lambda^\nu} \frac{(u^\nu_{\lambda})^\alpha-u^\alpha}{u^{\nu}_{\lambda}-u}[(u^{\nu}_{\lambda}-u )^+]^2 \mathrm{d}x \nonumber\\
&\leq \frac{\|v^\nu_\lambda\|^\beta_\infty\overline{C}_1}{C'} \int_{A_\lambda^\nu}[(u^{\nu}_{\lambda}-u )^+]^2 \mathrm{d}x \nonumber\\
&\leq \frac{\|v^\nu_\lambda\|^\beta_\infty \overline{C}_1}{C'} C |A_\lambda^\nu|^\frac{1}{N} \left( |A|^\frac{1}{N} \int_{A} |\nabla(u^{\nu}_{\lambda}-u )^+|^2+|B|^\frac{1}{N} \int_{B} |\nabla(u^{\nu}_{\lambda}-u )^+|^2 \mathrm{d}x \mathrm{d}x \right)\nonumber\\
&\leq \frac{\|v^\nu_\lambda\|^\beta_\infty \overline{C}_1}{C'}C |A_\lambda^\nu|^\frac{1}{N} \left(|A_\lambda^\nu|^\frac{1}{N} M_A^{2-p}+|B|^\frac{1}{N} M_K^{2-p} \right)  \nonumber\\
&\quad \times \int_{\Sigma_\lambda} (|\nabla u|+|\nabla u^{\nu}_{\lambda}|)^{p-2} |\nabla(u^{\nu}_{\lambda}-u )^+|^2 \mathrm{d}x, \nonumber
\end{align}
where $M_A:=\max\{ \sup\limits_A (|\nabla u|+|\nabla u^{\nu}_{\lambda}|), \sup\limits_A (|\nabla v|+|\nabla v^{\nu}_{\lambda}|)\}$. In the same way, we get
\begin{align}\label{eq1203-}
    &\quad \frac{1}{C'}\int_{A^\nu_\lambda}u^\alpha((v^\nu_{\lambda})^\beta -v^{\beta}) (u^\nu_{\lambda}-u)^+ \mathrm{d}x \\
    &\leq \frac{\|u^\nu_\lambda\|^\alpha_\infty\overline{C}_2}{C'} \int_{A_\lambda^\nu}(v^{\nu}_{\lambda}-v )^+(u^{\nu}_{\lambda}-u )^+ \mathrm{d}x \nonumber\\
    &\leq \frac{\|u^\nu_\lambda\|^\alpha_\infty \overline{C}_2}{2C'}C |A_\lambda^\nu|^\frac{1}{N} \left(|A_\lambda^\nu|^\frac{1}{N} M_A^{2-p}+ |B|^\frac{1}{N} M_K^{2-p} \right)  \nonumber\\
&\quad \times \left(\int_{\Sigma_\lambda} (|\nabla v|+|\nabla v^{\nu}_{\lambda}|)^{p-2} |\nabla(v^{\nu}_{\lambda}-v )^+|^2 +(|\nabla u|+|\nabla u^{\nu}_{\lambda}|)^{p-2} |\nabla(u^{\nu}_{\lambda}-u )^+|^2\right)\mathrm{d}x. \nonumber
\end{align}

Combining this inequality with \eqref{eq1201}--\eqref{eq1203-} and the fact that
$$\left(\Sigma^{\nu}_{\lambda}\cap (supp(u^{\nu}_{\lambda}-u)^+\cup supp(v^{\nu}_{\lambda}-v)^+)\right)\subset (\R^N \setminus B(P_\lambda^\nu,\overline{R}))\cup (A \cup B),$$
we get
\begin{align}\label{eq1204}
&\quad\int_{\Sigma^{\nu}_\lambda} (|\nabla u|+|\nabla u^{\nu}_{\lambda}|)^{p-2} |\nabla(u^{\nu}_{\lambda}-u)^+|^2 \mathrm{d}x \\
&\leq \frac{C_0(2\overline{C}_1+\overline{C}_2)}{2C'(2C'_0)^{p-2}} \frac{1}{R^{\beta_1}}\left(\frac{2}{N+s-2}\right)^2 \int_{\Sigma_\lambda^\nu}(|\nabla u|+|\nabla u^{\nu}_{\lambda}|)^{p-2} |\nabla (u^{\nu}_{\lambda}-u)^+|^2 \mathrm{d}x \nonumber\\
& \quad +\frac{ C_0\overline{C}_2}{2C'(2C'_0)^{p-2}}\frac{1}{R^{\beta_1}}\left(\frac{2}{N+s-2}\right)^2\int_{\Sigma_\lambda^\nu} (|\nabla v|+|\nabla v^{\nu}_{\lambda}|)^{p-2} |\nabla(v^{\nu}_{\lambda}-v )^+|^2 \mathrm{d}x  \nonumber\\
&\quad +\frac{\overline{C}_1\|v^\nu_\lambda\|^\beta_\infty }{C'}C |A_\lambda^\nu|^\frac{1}{N} \left(  |A_\lambda^\nu|^\frac{1}{N} M_A^{2-p}+|B|^\frac{1}{N} M_K^{2-p}\right)  \nonumber\\
&\quad \quad \times \int_{\Sigma_\lambda} (|\nabla u|+|\nabla u^{\nu}_{\lambda}|)^{p-2} |\nabla(u^{\nu}_{\lambda}-u )^+|^2 \mathrm{d}x \nonumber\\
 &\quad +\frac{\overline{C}_2\|u^\nu_\lambda\|^\alpha_\infty }{2C'}C |A_\lambda^\nu|^\frac{1}{N} \left(|A_\lambda^\nu|^\frac{1}{N} M_A^{2-p}+ |B|^\frac{1}{N} M_K^{2-p} \right)  \nonumber\\
&\quad \quad \times \left(\int_{\Sigma_\lambda} (|\nabla v|+|\nabla v^{\nu}_{\lambda}|)^{p-2} |\nabla(v^{\nu}_{\lambda}-v )^+|^2 +(|\nabla u|+|\nabla u^{\nu}_{\lambda}|)^{p-2} |\nabla(u^{\nu}_{\lambda}-u )^+|^2\right)\mathrm{d}x. \nonumber
\end{align}

Similarly, we get
\begin{align}\label{eq1204-}
&\quad\int_{\Sigma^{\nu}_\lambda} (|\nabla v|+|\nabla v^{\nu}_{\lambda}|)^{p-2} |\nabla(v^{\nu}_{\lambda}-v)^+|^2 \mathrm{d}x \\
&\leq \frac{C_0(2\overline{C}_1+\overline{C}_2)}{2C'(2C'_0)^{p-2}} \frac{1}{R^{\beta_1}}\left(\frac{2}{N+s-2}\right)^2 \int_{\Sigma_\lambda^\nu}(|\nabla v|+|\nabla v^{\nu}_{\lambda}|)^{p-2} |\nabla (v^{\nu}_{\lambda}-v)^+|^2 \mathrm{d}x \nonumber\\
& \quad +\frac{ C_0\overline{C}_2}{2C'(2C'_0)^{p-2}}\frac{1}{R^{\beta_1}}\left(\frac{2}{N+s-2}\right)^2\int_{\Sigma_\lambda^\nu} (|\nabla u|+|\nabla u^{\nu}_{\lambda}|)^{p-2} |\nabla(u^{\nu}_{\lambda}-u )^+|^2 \mathrm{d}x  \nonumber\\
&\quad + \frac{\overline{C}_1\|u^\nu_\lambda\|^\beta_\infty }{C'}C |A_\lambda^\nu|^\frac{1}{N} \left( |A_\lambda^\nu|^\frac{1}{N} M_A^{2-p}+ |B|^\frac{1}{N} M_K^{2-p} \right)  \nonumber\\
&\quad \quad\times \int_{\Sigma_\lambda} (|\nabla v|+|\nabla v^{\nu}_{\lambda}|)^{p-2} |\nabla(v^{\nu}_{\lambda}-v )^+|^2 \mathrm{d}x \nonumber\\
 &\quad +\frac{\overline{C}_2\|v^\nu_\lambda\|^\alpha_\infty }{2C'}C |A_\lambda^\nu|^\frac{1}{N} \left(  |A_\lambda^\nu|^\frac{1}{N} M_A^{2-p}+|B|^\frac{1}{N} M_K^{2-p} \right)  \nonumber\\
&\quad \quad \times \left(\int_{\Sigma_\lambda} (|\nabla u|+|\nabla u^{\nu}_{\lambda}|)^{p-2} |\nabla(u^{\nu}_{\lambda}-u )^+|^2 +(|\nabla v|+|\nabla v^{\nu}_{\lambda}|)^{p-2} |\nabla(v^{\nu}_{\lambda}-v )^+|^2\right)\mathrm{d}x. \nonumber
\end{align}

Then we add \eqref{eq1204} and \eqref{eq1204-} together
\begin{align}\label{sum}
& \quad\int_{\Sigma^{\nu}_\lambda} (|\nabla u|+|\nabla u^{\nu}_{\lambda}|)^{p-2} |\nabla(u^{\nu}_{\lambda}-u)^+|^2 \mathrm{d}x +\int_{\Sigma^{\nu}_\lambda} (|\nabla v|+|\nabla v^{\nu}_{\lambda}|)^{p-2} |\nabla(v^{\nu}_{\lambda}-v)^+|^2 \mathrm{d}x \\
&\leq \frac{C_0(\overline{C}_1+\overline{C}_2)}{C'(2C'_0)^{p-2}} \frac{1}{R^{\beta_1}}\left(\frac{2}{N+s-2}\right)^2 \int_{\Sigma_\lambda^\nu}(|\nabla u|+|\nabla u^{\nu}_{\lambda}|)^{p-2} |\nabla (u^{\nu}_{\lambda}-u)^+|^2 \mathrm{d}x \nonumber\\
& \quad +\frac{C_0(\overline{C}_1+\overline{C}_2)}{C'(2C'_0)^{p-2}}\frac{1}{R^{\beta_1}}\left(\frac{2}{N+s-2}\right)^2\int_{\Sigma_\lambda^\nu} (|\nabla v|+|\nabla v^{\nu}_{\lambda}|)^{p-2} |\nabla(v^{\nu}_{\lambda}-v )^+|^2 \mathrm{d}x  \nonumber\\
&\quad +\frac{2\overline{C}_1\|v^\nu_\lambda\|^\beta_\infty + \overline{C}_2(\|v^\nu_\lambda\|^\alpha_\infty+\|u^\nu_\lambda\|^\alpha_\infty) }{C'}C |A_\lambda^\nu|^\frac{1}{N} \left(  |A_\lambda^\nu|^\frac{1}{N} M_A^{2-p}+|B|^\frac{1}{N} M_K^{2-p} \right)  \nonumber\\
&\quad \quad \times \int_{\Sigma_\lambda} (|\nabla u|+|\nabla u^{\nu}_{\lambda}|)^{p-2} |\nabla(u^{\nu}_{\lambda}-u )^+|^2 \mathrm{d}x \nonumber\\
&\quad +\frac{2\overline{C}_1\|u^\nu_\lambda\|^\beta_\infty + \overline{C}_2(\|u^\nu_\lambda\|^\alpha_\infty+\|v^\nu_\lambda\|^\alpha_\infty )}{C'}C|A_\lambda^\nu|^\frac{1}{N} \left(  |A_\lambda^\nu|^\frac{1}{N} M_A^{2-p}+|B|^\frac{1}{N} M_K^{2-p} \right)  \nonumber\\
&\quad \quad\times \int_{\Sigma_\lambda} (|\nabla v|+|\nabla v^{\nu}_{\lambda}|)^{p-2} |\nabla(v^{\nu}_{\lambda}-v )^+|^2 \mathrm{d}x. \nonumber
\end{align}
Let
$$\hat{G}(R):=\frac{ C_0(\overline{C}_1+\overline{C}_2)}{C'(2C'_0)^{p-2}} \frac{1}{R^{\beta_1}}\left(\frac{2}{N+s-2}\right)^2.$$
It is easy to see that $\hat{G}(R)\to 0$ as $R\to\infty$ for $i=1,2.$
Now, we can choose $R$ so that $\hat{G}(\overline{R})\leq \hat{G}(R) <\frac{1}{4}$ for each $\overline{R}\geq R$. Let
$$K_1:=2\overline{C}_1\|v^\nu_\lambda\|^\beta_\infty + \overline{C}_2(\|v^\nu_\lambda\|^\alpha_\infty+\|u^\nu_\lambda\|^\alpha_\infty) ,$$
$$K_2:=2\overline{C}_1\|u^\nu_\lambda\|^\beta_\infty + \overline{C}_2(\|u^\nu_\lambda\|^\alpha_\infty+\|v^\nu_\lambda\|^\alpha_\infty ).$$
Then it is enough to choose
$$ \delta:= \frac{1}{|A_\lambda^\nu|}\left(\frac{C'}{4CM_K^{2-p}} \right)^N\min\left\{ \frac{1}{K_1^N}, \frac{1}{K_2^N} \right\}$$
and
$$ M^{2-p}:=\frac{C'}{4C|A_\lambda^\nu|^\frac{2}{N} }\min \left\{\frac{1}{K_1}, \frac{1}{K_2} \right\}.$$
With such choices of $\delta$ and $M$, it is easy to see that, whenever
$$ (supp(u^{\nu}_{\lambda}-u)^+\cup supp(v^{\nu}_{\lambda}-v)^+)\cap K=A\cup B,$$
$$|A\cap B|=0,\,|A\cup B|<\delta,\,M_A<M,$$
\eqref{sum} implies that $(u^{\nu}_{\lambda}-u)^+=0$ and $(v^{\nu}_{\lambda}-v)^+=0$ in $\Sigma_\lambda^\nu$, i.e., $u^{\nu}_{\lambda}\leq u$ and $v^{\nu}_{\lambda}\leq v$ in $\Sigma_\lambda^\nu$. This concludes our proof of Lemma \ref{lm.4.1}.
\end{proof}

\begin{rem}\label{re.4.1}
Observe that, for given $\lambda_0$ and $\nu_0$, $\overline{R}$ can be chosen so large such that, for all $\lambda$ in a small neighbourhood of $\lambda_0$, say, $[ \lambda_0-\theta_0 ,\lambda_0 + \theta_0 ]$ with $\theta_{0}>0$ small, and for all $\nu$ in $\overline{I_{\theta_0}(\nu_0)}:=\{ \nu \mid |\nu|=1,|\nu-\nu_0|\leq \theta_0 \}$, we have
$$ \overline{R}>R=R(\lambda,\nu)\,\quad \mbox{and}\,\quad B(P_\lambda^\nu,R)\subset B(P_{\lambda_0}^{\nu_0},\overline{R}), $$
where $R=R(\lambda,\nu)$ is given by Lemma \ref{lm.4.1} $(ii)$. Furthermore, $\delta$ and $M$ can also be chosen uniformly for all these $\lambda$ and $\nu$ in a $\theta_0$--neighbourhood of $(\lambda_0,\nu_0)$, by replacing $|A_\lambda^\nu|$ with $|B(P_{\lambda_0}^{\nu_0},\overline{R})|$ on the R.H.S of \eqref{eq1204}.
\end{rem}

\begin{lem}\label{lm.4.2}
Assume $(u,v)$ is a pair of positive weak solutions of \eqref{eq1.1} for $1<p<2$. Let $\nu_0$ be any unit vector in $\R^N$ and $\lambda_0\in\overline{\Lambda}(\nu_0)$. Let $\overline{R}$, $\delta$ and $M$ be given by Remark \ref{re.4.1} corresponding to $\theta_0-$neighbourhood of $(\lambda_0,\nu_0)$.

If we have
\begin{equation}\label{assump}
  u>u_{\lambda_0}^{\nu_0},\quad\,\,\mbox{in} \,\,\,(\Sigma_{\lambda_0}^{\nu_0}\setminus Z_{\lambda_0}^{\nu_0}(u))  \cap  B(P_{\lambda_0}^{\nu_0},\overline{R}),
\end{equation}
\begin{equation}\label{assump-}
  v>v_{\lambda_0}^{\nu_0},\quad\,\,\mbox{in} \,\,\,(\Sigma_{\lambda_0}^{\nu_0}\setminus Z_{\lambda_0}^{\nu_0}(v))  \cap  B(P_{\lambda_0}^{\nu_0},\overline{R}),
\end{equation}
then there exists $0<\theta_1 < \theta_0$ such that
%{\bf $(i)$} $u\geq u_\lambda^{\nu_0}$ in $\Sigma_\lambda^{\nu_0} \,\,\forall\,\lambda\in(\lambda_0-\theta_1,\lambda_0+\theta_1)$,\\
%{\bf $(ii)$}
$$ u\geq u^{\nu}_{\lambda},\,\,v\geq v^{\nu}_{\lambda} \quad \mbox{in}\,\, \Sigma_\lambda^\nu,\,\,\,\,\quad \forall\,\,(\lambda,\nu)\in (\lambda_0-\theta_1,\lambda_0+\theta_1)\times I_{\theta_1}(\nu_0),$$
where $I_{\theta_1}(\nu_0):=\{ \nu \,|\, |\nu|=1, |\nu-\nu_0|<\theta_1 \}$.
\end{lem}
\begin{proof}
Let $K:=B(P_{\lambda_0}^{\nu_0},\overline{R})$, where $P_{\lambda_0}^{\nu_0}$ is the projection of the origin on $T_{\lambda_0}^{\nu_0}$, and let $\delta$ and $M$ be chosen as in Remark \ref{re.4.1} uniformly for all $(\lambda,\nu)$ in a $\theta_0$-neighbourhood of $(\lambda_0,\nu_0)$.

Since $\overline{R}\geq R>R_{2}>2R_{1}>2R_{0}$, $\Sigma_\lambda^\nu \cap R_\lambda^\nu(B_{R_1}) \subset B(P_\lambda^\nu,R_2) \subset B(P_\lambda^\nu,R)\subset B(P_{\lambda_0}^{\nu_0},\overline{R})$ and $\Sigma_\lambda^\nu \cap R_\lambda^\nu(B_{R_1}) \neq \emptyset$ imply $B_{R_0}\subset B(P_{\lambda_0}^{\nu_{0}},\overline{R})$, by the sharp asymptotic estimates in Theorem \ref{th2.1-}, we get
$$Z_u\cup Z_v\subset K.$$
On the other hand, by Remark \ref{re2334}, the Lebesgue measures of $Z_{u}$ and $Z_{v}$ are zero, i.e., $|Z_u|=|Z_v|=0$.  Thus we can choose a small neighborhood $B$ of $\partial\left(K\cap \Sigma_{\lambda_0}^{\nu_0}\right)$ with $|B|<\frac{\delta}{4}$ and another neighbourhood $O$ of $(Z_{\lambda_0}^{\nu_0}(u)\cup Z_{\lambda_0}^{\nu_0}(v)) \cap K$ with $|O|<\frac{\delta}{4}$ and $M_O^{\lambda_0,\nu_0}<\frac{M}{2}$, where
$$ M_O^{\lambda,\nu}=\max\{\sup_O (|\nabla u|+|\nabla u^{\nu}_{\lambda}|), \sup_O (|\nabla v|+|\nabla v^{\nu}_{\lambda}|)\} .$$

Note that $K_0=(\overline{K\cap\Sigma_{\lambda_0}^{\nu_0}}) \setminus (B\cup O)$ is a compact set, so by our assumption \eqref{assump}, one has, there exists some positive constant $m>0$ such that
$$ u-u_{\lambda_0}^{\nu_0}>m>0,\  v-v_{\lambda_0}^{\nu_0}>m>0 \quad\quad \,\mbox{in}\,\,\, K_0. $$
Then the continuity with respect to $\lambda$ and $\nu$ implies that, there exists $\theta_1=\theta_1(\lambda_0,\nu_0,K)<\theta_0$, such that
\begin{align*}
& u-u^{\nu}_{\lambda}> \frac{m}{2}>0,\ v-v^{\nu}_{\lambda}> \frac{m}{2}>0 \qquad \,\mbox{in}\,K_0, \\
& M_{O\cap \Sigma_\lambda^\nu}^{\lambda,\nu}< M, \\
& |B\cap B(P_\lambda^\nu,R) \cap \Sigma_\lambda^\nu |< \frac{\delta}{2},\\
& |O\cap B(P_\lambda^\nu,R) \cap \Sigma_\lambda^\nu |< \frac{\delta}{2},\\
& (B(P_\lambda^\nu,R)\cap \Sigma_\lambda^\nu)\subset (B(P_{\lambda_0}^{\nu_0},\overline{R})\cap \Sigma_{\lambda_0}^{\nu_0}) \cup B
\end{align*}
for all $(\lambda,\nu)\in (\lambda_0-\theta_1, \lambda_0+\theta_1)\times I_{\theta_1}(\nu_0)$, where $R=R(\lambda_0,\nu_0)$ is given by Lemma \ref{lm.4.1} $(ii)$. Then we have
\begin{align*}
     & \quad(supp(u^{\nu}_{\lambda}-u)^+\cup supp(v^{\nu}_{\lambda}-v)^+) \cap B(P_\lambda^\nu,R)\\
     &\subset (B\cup O) \cap B(P_\lambda^\nu,R) \\
     &= (B\cap B(P_\lambda^\nu,R)) \cup (O\cap B(P_\lambda^\nu,R))
\end{align*}
for all $(\lambda,\nu)\in (\lambda_0-\theta_1, \lambda_0+\theta_1)\times I_{\theta_1}(\nu_0)$, which implies that
\begin{align*}
    & \quad(supp(u^{\nu}_{\lambda}-u)^+\cup supp(v^{\nu}_{\lambda}-v)^+) \cap B(P_\lambda^\nu,R)\cap \Sigma_\lambda^\nu\\
    &\subset(B\cap B(P_\lambda^\nu,R)\cap \Sigma_\lambda^\nu) \cup (O\cap B(P_\lambda^\nu,R)\cap \Sigma_\lambda^\nu)
\end{align*}
satisfies the conditions of Lemma \ref{lm.4.1} $(ii)$. Therefore, we derive
$$ u\geq u^{\nu}_{\lambda},\ v\geq v^{\nu}_{\lambda} \quad \,\,\,\,\,\mbox{in}\,\, \Sigma_\lambda^\nu$$
for all $(\lambda,\nu)\in (\lambda_0-\theta_1,\lambda_0+\theta_1)\times I_{\theta_1}(\nu_0)$. This completes our proof of Lemma \ref{lm.4.2}.
\end{proof}

\begin{rem}\label{re.4.2}
Under the same assumptions as Lemma \ref{lm.4.2}, if, for some $\lambda_0 \in \overline{\Lambda}(\nu_0)$, there exist open neighbourhoods $B$ and $O$ of the sets $\partial\left(\Sigma^{\nu_0}_{\lambda_0} \cap B(P^{\nu_0}_{\lambda_0},\overline{R})\right)$ and $\left(Z_{\lambda_0}^{\nu_0}(u) \cup Z_{\lambda_0}^{\nu_0}(v)\right)\cap B(P^{\nu_0}_{\lambda_0},\overline{R})$ respectively such that
$$ u>u^{\nu_0}_{\lambda_0},\,\,v>v_{\lambda_0}^{\nu_0}\quad\, \mbox{in}\,\,\overline{\left(B(P^{\nu_0}_{\lambda_0},\overline{R})\cap \Sigma^{\nu_0}_{\lambda_0}\right)}\setminus (B\cup O) $$
with $|B|+|O|<\frac{\delta}{2}$ and $M_O^{\nu_0,\lambda_0}<\frac{M}{2}$, where $\delta$ and $M$ are chosen as in Remark \ref{re.4.1} uniformly for all  $(\lambda,\nu)\in (\lambda_0-\theta_0, \lambda_0+\theta_0)\times I_{\theta_0}(\nu_0)$. Then there exists a sufficiently small $\theta_1>0$ such that Lemma \ref{lm.4.2} holds.
\end{rem}

\subsubsection{Proof of the radial symmetry and strictly radial monotonicity}

In this subsection, we are to complete the proof of Theorem \ref{th2} in the singular elliptic case $1<p<2$.

\begin{proof}
Along any direction $\nu$ in $\R^N$, we start moving the planes $T^{\nu}_{\lambda}$ from $\lambda$ near $+\infty$ until its limiting position $T^{\nu}_{\lambda_{0}(\nu)}$.

\medskip

{\bf Step 1.} We will show that $\overline{\Lambda}(\nu)\neq\emptyset$ and $\lambda_0(\nu)$ is well-defined and finite.

We claim that $\overline{\Lambda}(\nu) \supset (R_1,+\infty)$, where $R_1=R_1(N,p,R_0)>0$ is given by Lemma \ref{lm.4.1} $(i)$. If $\lambda > R_1$,
$$ \Sigma_\lambda^\nu \cap R_\lambda^\nu(B_{R_1})=\emptyset.$$
Hence using Lemma \ref{lm.4.1} $(i)$, we have
$$ u\geq u_\lambda^\nu\,\,\mbox{and}\,\,v\geq v^\nu_\lambda \,\quad\,\,\mbox{in}\,\, \Sigma_\lambda^\nu $$
for any $\lambda\in\overline{\Lambda}(\nu)$ and $\lambda > R_1$.

\medskip

{\bf Step 2.} We are to prove $u \equiv u^\nu_{\lambda_0(\nu)}$ and $v \equiv v^\nu_{\lambda_0(\nu)}$ in $\Sigma^\nu_{\lambda_0(\nu)}$.

At the limiting position $T^{\nu}_{\lambda_0(\nu)}$, by continuity, we have $u\geq u^\nu_{\lambda_0(\nu)}$ in $\Sigma^\nu_{\lambda_0(\nu)}$. Next, by the strong comparison principle in Lemma \ref{th3.2} and the fact that
\begin{align}\label{eq2305}
-\Delta_p u =u^\alpha v^\beta \geq (u^\nu_\lambda)^\alpha (v^\nu_\lambda)^\beta= -\Delta_p u^{\nu}_{\lambda_0(\nu)} \quad\,\,\mbox{in}\,\,\Sigma^{\nu}_{\lambda_0(\nu)},
\end{align}
we deduce that, either $u>u^\nu_{\lambda_0(\nu)}$  or $u=u^\nu_{\lambda_0(\nu)}$ in all $\mathcal{C}^\nu$, where $\mathcal{C}^\nu$ is a connected component of $\Sigma^\nu_{\lambda_0(\nu)} \setminus Z^\nu_{\lambda_0(\nu)}(u)$.

Noting that the sharp lower bound of $|\nabla u|$ in Theorems \ref{th2.1-} indicates that the critical set $Z_u=\{x \in \R^N \mid |\nabla u(x)| = 0 \}$ of the solution $u$ is a closed set belonging to $B_{R_0}(0)$, i.e.,
$$Z^\nu_{\lambda_0(\nu)}(u)\subset Z_u \subset B_{R_0}.$$
Meanwhile, it follows from Corollary \ref{re2333} that $|Z_{u}|=0$ (see Remark \ref{re2334}), and from Lemma \ref{lm.7} that $\Omega\setminus Z_u$ is connected for any smooth bounded domain $\Omega\subset\mathbb{R}^{N}$ with connected boundary such that $B_{R_0}(0)\subseteq\Omega$. As a consequence, one has $\Sigma^\nu_{\lambda_0(\nu)}\setminus Z_u$ is connected. Due to the closeness of $Z_{u}$ and the fact that $|Z_{u}|=0$, we can deduce that $\Sigma^\nu_{\lambda_0(\nu)}\setminus Z^\nu_{\lambda_{0}(\nu)}(u)$ is connected.

\smallskip

If $u\equiv u^\nu_{\lambda_0}$ in $\Sigma^\nu_{\lambda_0(\nu)}\setminus Z^\nu_{\lambda_{0}(\nu)}(u)$, then Step 2 is finished. Otherwise, we assume on the contrary that
\begin{align}\label{eq2305009}
u > u^\nu_{\lambda_0(\nu)} \quad\,\,\,\mbox{in}\,\,\Sigma^\nu_{\lambda_0(\nu)} \setminus Z^\nu_{\lambda_0(\nu)}(u),
\end{align}
%If there exists a point $x_0\in \Sigma^\nu_{\lambda_0(\nu)}$ such that $u(x_0)> u^\nu_{\lambda_0} (x_0)$. Now we will prove that {\bf $(ii)$} occurs.
%Suppose by contradiction that there does not exist any component $C^\nu$ of $\Sigma^\nu_{\lambda_0(\nu)} \setminus Z^\nu_{\lambda_0(\nu)}$ such that $C^\nu \cap B(P^\nu_{\lambda_0(\nu)},\overline{R})\neq\emptyset$ and $u\equiv u^\nu_{\lambda_0(\nu)}$ in $C^\nu$. Using Theorem \ref{th3.2}, we have $u>u^\nu_{\lambda_0}$ on $(B(P^\nu_{\lambda_0(\nu)},\overline{R}) \cap \Sigma^\nu_{\lambda_0(\nu)})\setminus Z^\nu_{\lambda_0(\nu)}$.
then Lemma \ref{lm.4.2} implies that $u\geq u_\lambda^\nu$ in $\Sigma_\lambda^\nu$ for any $\lambda\in(\lambda_0(\nu)-\theta_1, \lambda_0(\nu))$ with $\theta_{1}>0$ small, which contradicts the definition of $\lambda_0(\nu)$. As a consequence, we have proved that
$$ u\equiv u^\nu_{\lambda_0(\nu)} \,\,\quad \mbox{in}\,\, \Sigma^\nu_{\lambda_0(\nu)} \setminus Z^\nu_{\lambda_0(\nu)}(u).$$
Furthermore, by the definition of $Z^\nu_{\lambda_0(\nu)}(u)$ and $u\in C_{loc}^{1,\alpha}(\R^N)$, it follows that
\begin{align}\label{eq23050091}
u\equiv u^\nu_{\lambda_0(\nu)} \quad\,\,\mbox{in} \,\,\Sigma^\nu_{\lambda_0(\nu)}.
\end{align}

On the other hand, according to the strong comparison principle Lemma \ref{th3.2}, one knows that
$ u > u^\nu_{\lambda}$ in $\Sigma^\nu_\lambda\setminus Z^\nu_{\lambda(\nu)}(u)$ for any $\lambda > \lambda_0(\nu)$.    %consequently $u$ is monotone increasing in $\Sigma^\nu_{\lambda_0(\nu)} \setminus Z^\nu_{\lambda_0(\nu)}$ so that $u_{x_1} \geq 0$ in $\Sigma^\nu_{\lambda_0(\nu)} \setminus Z^\nu_{\lambda_0(\nu)}$.
Consequently, $u$ is strictly monotone increasing in $\Sigma^\nu_{\lambda_0(\nu)}\setminus Z^\nu_{\lambda_0(\nu)}(u)$ for any direction $\nu$, and hence $\frac{\partial u}{\partial \nu} \geq 0$ in $\Sigma^\nu_{\lambda_0(\nu)} \setminus Z^\nu_{\lambda_0(\nu)}(u)$. With the same argument, we have the same conclusion about $v$.
%
%Furthermore $ u > u^\nu_{\lambda}$ in $\Sigma^\nu_\lambda\setminus Z^\nu_{\lambda(\nu)}$ for any $\lambda > \lambda_0(\nu)$ and  %consequently $u$ is monotone increasing in $\Sigma^\nu_{\lambda_0(\nu)} \setminus Z^\nu_{\lambda_0(\nu)}$ so that $u_{x_1} \geq 0$ in $\Sigma^\nu_{\lambda_0(\nu)} \setminus Z^\nu_{\lambda_0(\nu)}$.
%consequently $u$ in $\Sigma^\nu_{\lambda_0(\nu)}\setminus Z^\nu_{\lambda_0(\nu)}$ is monotonically increasing in the $\nu$ direction, i.e. $\frac{\partial u}{\partial \nu} \geq 0$ in $\Sigma^\nu_{\lambda_0(\nu)} \setminus Z^\nu_{\lambda_0(\nu)}$.

\medskip

{\bf Step 3.} We will show $u$ and $v$ are radial and (strictly) radially decreasing about some point $x_0\in\R^N$.

It is known that for any direction $\nu$, $u, v$ are symmetric about the plane $T^\nu_{\lambda_0(\nu)}$. Now consider all the $N$ orthogonal directions $\{e_k\}_{k=1}^N$ in $\R^N$, and one can obtain in entirely similar way that $u, v$ are strictly increasing in each direction $e_k$ in $\Sigma^{e_k}_{\lambda_0(e_k)}$ and symmetric about the plane $T^{e_k}_{\lambda_0(e_k)}$. This shows that $u,v$ are symmetric with respect to a point $x_0\in\R^N$ $(x_0 = \bigcap^N_{k=1}T^{e_k}_{\lambda_0(e_k)})$, which is the unique critical point for $u, v$. Furthermore, by exploiting the moving plane procedure in any direction $\nu\in S^{N-1}$, we finally get that $u, v$ are radial and (strictly) radially decreasing w.r.t. $x_{0}$. That is, positive solutions $u,v$ must assume the form
\[u(x)=u(|x-x_{0}|)=u(r), v(x)=v(|x-x_{0}|)=v(r) \]
some $x_{0}\in\mathbb{R}^{N}$, where $(u(r), v(r))$ with $r=|x-x_0|$ are the positive radial solutions to \eqref{eq1.1} with $u'(0)=v'(0)=0$ and $u'(r)\leq0, v'(r)\leq0 $ for any $r>0$.

\smallskip

Next, for $1<p<2$, we will prove that
\begin{equation}\label{final}
  u^\prime (r) < 0, \,\,v^\prime (r) < 0,\quad \,\,\,\forall \,\, r>0.
\end{equation}
We will prove \eqref{final} by contradiction arguments. We assume on the contrary that there exists $r_0>0$ such that $u^\prime (r_0) = 0$. Since Lemma \ref{re2333} and Remark \ref{re2334} implies that $|Z_u|=0$, one has $u'(r)<0$ for $r\in(0,r_{0})\cup(r_{0},+\infty)$, so $u(r)>u(r_0)$ in $[0,r_0)$. Using the Strong Maximum Principle and H\"{o}pf's Lemma in Lemma \ref{hopf} (c.f. \cite{JLV}, see also Theorem 2.1 in \cite{LDBS04}, Theorem 2.4 in \cite{LDSMLMSB} or \cite{LD,PS}) to the positive radial solution $\tilde{u}(x):=u(x)-u(r_0)=u(|x|)-u(r_{0})$ of
\begin{align*}
\left\{ \begin{array}{ll} \displaystyle
-\Delta_p \tilde{u}  = u^\alpha (r) v^\beta(r) > 0  \,\,\,\,&\mbox{in}\, \,B_{r_0}, \\
\tilde{u}>0   &\mbox{in}\,\,  B_{r_0},\\
\tilde{u}=0   &\mbox{on}\,\, \partial B_{r_0},
\end{array}
\right.\hspace{1cm}
\end{align*}
we obtain $\tilde{u}^\prime (r_0)=u'(r_{0})<0$, which is absurd. Hence \eqref{final} holds, i.e., $u'(r)<0$ for any $r>0$. Similarly, we have $v'(r)<0$ for any $r>0$ as well.

\end{proof}

\subsection{Proof of the uniqueness and classification result}\label{sc5}

\begin{proof}

  Let $(u, v)$ be a pair of solutions to system \eqref{eq1.1}. In subsections 4.1 and 4.2, we have proved that $u$ and $v$ are radially symmetric and monotone decreasing about some point $x_0$. Without loss of generality, we may assume that $x_0=0$.

    Since our solution $u$ is radially symmetric, we can write the first equation in \eqref{eq1.1} as
    $$\frac{-(r^{N-1}|u'(r)|^{p-2}u'(r))'}{r^{N-1}}=u^\alpha(r)v^\beta(r),$$
    where $r=|x|$.
    Recalling that $u'(r)<0$ and $v'(r)<0$, we get
    \begin{align}\label{Key}
        (r^{N-1}(-u'(r))^{p-1})'=r^{N-1}u^\alpha(r)v^\beta(r).
    \end{align}
    Integrating both sides of the above equation from $0$ to $r$ yields
    \begin{align}\label{eq5.10}
      -u'(r)=\left(\frac{1}{r^{N-1}} \int_0^r \tau^{N-1}u^\alpha(\tau) v^\beta(\tau) \mathrm{d}\tau\right)^{\frac{1}{p-1}},
        \end{align}
    which implies
    \begin{align}\label{eq5.1}
        u(0)-u(r)=\int_0^r \left(\frac{1}{s^{N-1}} \int_0^s \tau^{N-1}u^\alpha(\tau) v^\beta(\tau) \mathrm{d}\tau\right)^{\frac{1}{p-1}}\mathrm{d}s.
        \end{align}
    Similarly, we have
     \begin{align}\label{eq5.10'}
      -v'(r)=\left(\frac{1}{r^{N-1}} \int_0^r \tau^{N-1}v^\alpha(\tau) u^\beta(\tau) \mathrm{d}\tau\right)^{\frac{1}{p-1}},
        \end{align}
      which implies
    \begin{align}\label{eq5.1-}
        v(0)-v(r)=\int_0^r \left(\frac{1}{s^{N-1}} \int_0^s \tau^{N-1}u^\beta(\tau) v^\alpha(\tau) \mathrm{d}\tau\right)^{\frac{1}{p-1}}\mathrm{d}s.
    \end{align}
    Now we show that $u(0)=v(0)$. Otherwise, we may  suppose that
    \begin{align}\label{ineq5.1}
        u(0)<v(0).
    \end{align}
    Then by continuity, for small $r>0$,
    \begin{align}\label{ineq5.2}
        u(r)<v(r).
    \end{align}
    In other words, there exists an $R>0$, such that
    \begin{align}\label{ineq5.2-}
        u(r)<v(r), \qquad \forall \,\,\,r\in\,(0,R).
    \end{align}
    Let us define
    \begin{align}\label{a2}
    R_0:=\sup\{R>0|\,u(r)<v(r),\ \ \forall \,0<r<R\}.
    \end{align}
    Then $0<R_0 \leq +\infty$ and $u(R_0)=v(R_0)$, since $u(+\infty)=v(+\infty)=0$ by Theorem \ref{th2.1-}. By the definition of $R_0$ and $\alpha\leq\beta$, we have that
    $$u^\alpha v^\beta\geq u^\beta v^\alpha, \qquad \forall \,\, r\,\,\in \,\,(0, R_0).$$
    By \eqref{eq5.1} and \eqref{eq5.1-}, we have
    \begin{align}\label{ineq5.3}
    &0>u(0)-v(0)\\
    &\quad=\int_0^{R_0} \left(\frac{1}{s^{N-1}} \int_0^s \tau^{N-1}u^\alpha v^\beta(\tau) \mathrm{d}\tau\right)^{\frac{1}{p-1}}\mathrm{d}s\nonumber\\
    &\quad \quad -\int_0^{R_0}\left(\frac{1}{s^{N-1}} \int_0^s \tau^{N-1}u^\beta v^\alpha(\tau) \mathrm{d}\tau\right)^{\frac{1}{p-1}}\mathrm{d}s\geq 0.\nonumber
    \end{align}
    This is absurd. Similarly, we can show that $u(0)>v(0)$ is also impossible. Therefore, we must have $u(0)=v(0)$.

Now we aim to show that $u\equiv v$ by using contradiction argument. Assume on the contrary that $v\neq u$. Define
\begin{align}\label{24q}
r_0:=\sup\{r\geq0|\, u(s)=v(s),\ \ \forall \,s\leq r\}.
\end{align}
From $u(0)=v(0)$ and $u\ne v$, we know $0\leq r_0<+\infty$. Without loss of generality, we may assume that, there exists $\epsilon_{0}>0$ small enough, such that $u(r) > v(r)$ and $u'(r)>v'(r)$ for any $r\in(r_0,r_0+\epsilon_{0})$. Let us define
    \begin{align}\label{a3}
    R_*:=\sup\{R>0|\,u(r)>v(r),\ \ \forall \,r_{0}<r<R\}.
    \end{align}
    Then $r_{0}+\epsilon_0\leq R_* \leq +\infty$ and $u(R_*)=v(R_*)$, since $u(+\infty)=v(+\infty)=0$ by Theorem \ref{th2.1-}. We will show that $R_*=+\infty$.

    By subtracting equation \eqref{eq5.10'} from equation \eqref{eq5.10}, we can get, for any $r\in(r_0,R_*)$,
\begin{align}\label{25q}
	&\quad \quad u'(r)-v'(r) \\
&=\left(\frac{1}{r^{N-1}} \int_0^r \tau^{N-1}v^\alpha(\tau) u^\beta(\tau) \mathrm{d}\tau\right)^{\frac{1}{p-1}}-\left(\frac{1}{r^{N-1}} \int_0^r \tau^{N-1}u^\alpha(\tau) v^\beta(\tau) \mathrm{d}\tau\right)^{\frac{1}{p-1}}\geq 0, \nonumber
\end{align}
and $u'(r)-v'(r)$ is monotone increasing w.r.t. $r$ on $(r_0,R_*)$, and hence
\begin{equation}\label{a4}
  u'(r)-v'(r)\geq u'\left(r_0+\frac{\epsilon_{0}}{2}\right)-v'\left(r_0+\frac{\epsilon_{0}}{2}\right)=:\delta_0>0, \qquad \forall \, r\in\Big[r_0+\frac{\epsilon_{0}}{2},R_*\Big).
\end{equation}
If $R_*<+\infty$, then $u(R_*)-v(R_*)\geq u(r_0+\epsilon_{0})-v(r_0+\epsilon_{0})>0$, which contradicts with $u(R_*)=v(R_*)$. Thus we must have $R_*=+\infty$.

By Theorem \ref{th2.1-} and \eqref{a4}, we derive that, for any $r\in[r_0+\frac{\epsilon_{0}}{2},+\infty)$,
\begin{align}\label{a5}
  0<\delta_0\leq u'(r)-v'(r)\leq \frac{C}{r^{\frac{N-1}{p-1}}}\rightarrow 0, \qquad \text{as} \,\, r\rightarrow+\infty,
\end{align}
which is absurd. Thus we must have $u\equiv v$, and $u$ solves the single $D^{1,p}(\mathbb{R}^{N})$-critical $p$-Laplacian equation
\begin{equation}\label{a1}
  -\Delta_{p} u=u^{p^{\star}-1} \qquad \text{in} \,\, \mathbb{R}^{N}.
\end{equation}
From the classification results for \eqref{a1} in \cite{LDSMLMSB,BS16,VJ16}, we deduce that positive solutions $(u,v)$ must assume the form
\[u(x)=v(x)=\left[ \frac{\lambda^{\frac{1}{p-1}}\left(N^{\frac{1}{p}}\left(\frac{N-p}{p-1}\right)^{\frac{p-1}{p}}\right)}{\lambda^{\frac{p}{p-1}}+|x-x_0|^{\frac{p}{p-1}}} \right]^{\frac{N-p}{p}}.\]
This concludes our proof of Theorem \ref{th2}.
\end{proof}

\end{document}